\documentclass[11pt,reqno]{amsart}
\usepackage{amsmath,amsthm,amssymb,amsfonts,mathrsfs}
\usepackage[margin=1in]{geometry}
\usepackage[colorlinks=true, pdfstartview=FitV, linkcolor=blue, citecolor=blue, urlcolor=blue]{hyperref}
\usepackage{enumerate,cite,times,color,multirow,graphicx}

\title[Fractional Primitive Equations]{Well-posedness and ill-posedness of the primitive equations with fractional horizontal dissipation}

\author[E. Abdo]{Elie Abdo}
\address[E. Abdo]
{Department of Mathematics \\
American University of Beirut \\
Beirut 1107-2020, Lebanon.} 
\email{ea94@aub.edu.lb}

\author[Q. Lin]{Quyuan Lin}
\address[Q. Lin]
{School of Mathematical and Statistical Sciences \\
Clemson University\\
Clemson, SC 29634, USA.}
\email{quyuanl@clemson.edu}

\author[C. Tan]{Changhui Tan}
\address[C. Tan]
{Department of Mathematics \\ 
University of South Carolina\\
Columbia, SC 29208, USA.} 
\email{tan@math.sc.edu}


\allowdisplaybreaks

\newtheorem{theorem}{Theorem}[section]
\newtheorem{lemma}[theorem]{Lemma}

\newtheorem{proposition}[theorem]{Proposition}
\theoremstyle{definition}
\newtheorem{definition}{Definition}
\theoremstyle{remark}
\newtheorem{remark}{Remark}[section]
\numberwithin{equation}{section}

\def\pa{\partial}
\def\lh{\Lambda_h}
\def\dd{\mathrm{d}}
\def\ddt{\frac{\dd}{\dd t}}
\def\dxdz{\,\dd x\dd z}
\def\D{\mathcal{D}}
\def\N{\mathbb{N}}
\def\R{\mathbb{R}}
\def\T{\mathbb{T}}
\def\Z{\mathbb{Z}}
\def\Et{\widetilde{E}}
\def\Ft{\widetilde{F}}
\def\H{\mathrm{H}}
\def\I{\mathrm{I}}
\def\II{\mathrm{II}}
\def\III{\mathrm{III}}
\def\W{\vec{W}}
\def\s{\mathrm{s}}
\def\eps{\varepsilon}

\begin{document}
\begin{abstract}
The primitive equations (PE) are a fundamental model in geophysical fluid dynamics. While the viscous PE are globally well-posed, their inviscid counterparts are known to be ill-posed.

In this paper, we study the two-dimensional incompressible PE with fractional horizontal dissipation. We identify a sharp transition between local well-posedness and ill-posedness at the critical dissipation exponent $\alpha = 1$. In the critical regime, this dichotomy exhibits a new phenomenon: the transition depends delicately on the balance between the size of the initial data and the viscosity coefficient. Our results precisely quantify the horizontal dissipation required to transition from inviscid instability to viscous regularity. 
We also establish a global well-posedness theory to the fractional PE, with sufficient dissipation $\alpha\geq\frac65$. 
\end{abstract} 

\maketitle


\section{Introduction}\label{sec:intro}

\subsection{The primitive equations}
In the study of oceanic and atmospheric dynamics at the planetary scale, the vertical scale (a few kilometers for the ocean, 10-20 kilometers for the atmosphere) is much smaller than the horizontal scale (many thousands of kilometers). Accordingly, the large-scale ocean and atmosphere satisfy the hydrostatic balance based on scale analysis, meteorological observations, and historical data. By virtue of this, the \emph{primitive equations} (PE), also known as the hydrostatic Navier–Stokes equations, are derived as the asymptotic limit of the small aspect ratio between the vertical and horizontal length scales from the Navier-Stokes equations (NSE) \cite{azerad2001mathematical,li2019primitive,li2022primitive,furukawa2020rigorous}. Because of its impressive accuracy, the following 3D viscous PE has become a widely used model in geophysical fluid dynamics (see, e.g., \cite{blumen1972geostrophic,gill1976adjustment,gill1982atmosphere,hermann1993energetics,holton1973introduction,kuo1997time,plougonven2005lagrangian,rossby1938mutual} and references therein):
\begin{subequations}\label{PE-system}
\begin{align}
    &\pa_t V + V\cdot \nabla_h  V + w\pa_z V -\nu_h \Delta_h V - \nu_z \pa_z^2 V +f_0 V^\perp + \nabla_h  p = 0 , \label{PE-1}
    \\
    &\pa_z p = 0, \label{PE-2}
    \\
    &\nabla_h  \cdot V + \pa_z w =0. \label{PE-3}
\end{align}
\end{subequations}
Here the horizontal velocity $V = (u, v)$, the vertical velocity $w$, and the pressure $p$ are the unknown quantities. The 2D horizontal gradient and Laplacian are denoted by $\nabla_h = (\partial_{x}, \partial_{y})$ and $\Delta_h = \pa_x^2 + \pa_y^2$, respectively. The nonnegative constants $\nu_h$ and $\nu_z$ represent horizontal and vertical viscosity, respectively. The Coriolis parameter is denoted by $f_0 \in \R$, and $V^\perp = (-v, u)$ denotes a 90-degree rotation in the horizontal plane. Equation \eqref{PE-2} expresses the hydrostatic pressure balance, while \eqref{PE-3} enforces incompressibility. Note that we drop the evolution of temperature in system \eqref{PE-system} for simplicity.

The PE system is typically studied in a periodic channel $\{(x,y,z):(x,y)\in\mathbb T^2,\, z\in [0,1]\}$, with boundary conditions $w|_{z=0,1}=0$ when $\nu_z=0$ and 
$(w,\partial_z V)|_{z=0,1}=0$ when $\nu_z>0$. The divergence-free condition \eqref{PE-3} and the boundary conditions imply that 
\[
    w(x,y,z,t)= - \int_{0}^z \nabla_h \cdot V(x,y,\tilde{z},t) \,\dd\tilde{z},
\]
so that $w$ is a diagnostic variable and can be recovered from $V$. This introduces a loss of one horizontal derivative in $w$, 
compared to NSE.
As a result, the PE system was considered to be more challenging than NSE due to this intricate nonlinear structure. 

However, thanks to the anisotropic structure, the 3D PE system with full viscosity ($\nu_h>0$ and $\nu_z>0$) was proved to be globally well-posed first in the pioneer work \cite{cao2007global}. See also \cite{kobelkov2006existence} for an alternative approach, \cite{kukavica2007regularity} for different boundary conditions, as well as \cite{hieber2016global} for some progress towards relaxing the smoothness on the initial data by using the semigroup method. Moreover, in a series of work \cite{cao2016global,cao2017strong,cao2020global}, global well-posedness has been established when the system has only horizontal viscosity ($\nu_h>0$ and $\nu_z=0$). 

On the other hand, when the horizontal viscosity vanishes $(\nu_h=0)$ the behavior of the PE system is completely different. With $\nu_z>0$, the system is ill-posed in Sobolev spaces \cite{renardy2009ill} but local well-posedness can be obtained by considering some additional weak dissipation \cite{cao2020well}, or assuming the initial data being Gevrey regular and convex \cite{gerard2020well}, or being analytic in the horizontal direction \cite{paicu2020hydrostatic,lin2022effect}. Global well-posedness of smooth solutions remains open.

In the inviscid case $(\nu_h=\nu_z=0)$, the PE system exhibits the Kelvin-Helmholtz instability, and the solutions are ill-posed in Sobolev spaces and Gevrey classes $G^\sigma$ with order $\sigma>1$ \cite{renardy2009ill,han2016ill,ibrahim2021finite}. With either some special structures (local Rayleigh condition) on the initial data in $2D$, or real analyticity in all directions for general initial data in both $2D$ and $3D$, the local well-posedness can be achieved \cite{brenier1999homogeneous,brenier2003remarks,ghoul2022effect,grenier1999derivation,kukavica2011local,kukavica2014local,masmoudi2012h}. Finally, smooth solutions of the inviscid PEs can form singularity in finite time \cite{cao2015finite,wong2015blowup,ibrahim2021finite,collot2023stable}.

A review of the aforementioned results reveals that horizontal viscosity plays a crucial role in determining the local well-posedness or ill-posedness, as well as the global existence or finite-time blow-up of smooth solutions to the PE system. In contrast to the Navier-Stokes and Euler equations, such sensitivity to horizontal viscosity is a distinctive feature of the PE system. This makes it both natural and compelling to investigate the effect of horizontal viscosity and to explore the regimes between the viscous and inviscid cases.

\subsection{Fractional dissipation}
In many fluid systems, the viscous and inviscid models are connected through a family of fractional dissipations, modeled through fractional Laplacian:
\[
\Lambda^\alpha f (x) := c_{\alpha,d}\,\text{ p.v.}\int_{\R^d}\frac{f(x)-f(y)}{|x-y|^{d+\alpha}}\,\dd y,\quad c_{\alpha,d}=\tfrac{2^\alpha\Gamma(\frac{d+\alpha}{2})}{\pi^{\frac d2}|\Gamma(-\frac\alpha2)|},
\]
for $\alpha\in(0,2)$, where \text{p.v.} stands for the principle value.
When $\alpha$ approaches $0$, the fractional Laplacian becomes the identity operator, which corresponds to the inviscid system. When $\alpha$ approaches $2$, the fractional Laplacian becomes $-\Delta$, leading to the viscous system.
Here, we list several fractional fluid systems in which a critical exponent $\alpha$ separates inviscid-like and viscous-like behavior.
\begin{itemize}
\item $1D$ fractional Burgers equation
\[\pa_t u + u\pa_x u = - \Lambda^\alpha u.\]
It is globally well-posed when $\alpha\geq1$, whereas finite-time blow-up occurs when $\alpha<1$ \cite{kiselev2008blow}.

\item $2D$ fractional surface quasi-geostrophic (SQG) equation
\[\pa_t \omega + u\cdot\nabla \omega = - \Lambda^\alpha \omega,\quad u=\nabla^\perp\Lambda^{-1}\omega.\]
It is globally well-posed when $\alpha>1$ \cite{constantin1999behavior}, as well as the when $\alpha=1$, known as the critical SQG equation \cite{kiselev2007global,caffarelli2010drift, constantin2012nonlinear}. Unlike the fractional Burgers equation, singularity formations are generally unknown for $\alpha<1$. 

\item $2D$ fractional Boussinesq system
\[\pa_t \omega + u\cdot\nabla \omega = - \Lambda^\alpha \omega + \pa_x\theta, \quad 
\pa_t\theta+u\cdot\nabla\theta = -\Lambda^\beta\theta, \quad u = \nabla^\perp\Delta^{-1}\omega.\]
Solutions are globally well-posed when $\alpha+\beta>1$. Partial results are know for the critical case $\alpha+\beta=1$. The supercritical case  $\alpha+\beta<1$ are generally open. See e.g. the recent paper \cite{stefanov2025global} and references therein. 

\item $1D$ fractional Euler-alignment system in collective dynamics
\[
 \pa_t\rho+\pa_x(\rho u) = 0,\quad
 \pa_t u + u\pa_xu =\Lambda^\alpha(\rho u) - u\Lambda^\alpha\rho.
\]
The dynamics of $u$ reduces to the fractional Burgers equation if we enforce $\rho\equiv1$. However, unlike the fractional Burgers equation, solutions to the factional Euler-alignment system are globally regular when $\alpha>0$ \cite{do2018global,shvydkoy2017eulerian}.
\end{itemize}

\subsection{Main results}
We consider the following two-dimensional PE system with fractional horizontal dissipation (FPE):
\begin{subequations}\label{FPE}
    \begin{align}
        &\pa_t u + u \pa_x u + w \pa_z u + \pa_x p + \nu_h \Lambda_h^{\alpha} u = 0, \label{FPE-1}
        \\
        &\pa_z p = 0, \label{FPE-2}
        \\
        & \pa_x u + \pa_z w = 0, \label{FPE-3}
    \end{align}
\end{subequations}
defined on a 2D periodic channel 
\[\Omega:= \{(x,z): x\in \T, z\in[0,1]\},\]
where $(u,w)$ are the horizontal and vertical velocities respectively, $\lh^\alpha = (-\pa_x^2)^{\frac\alpha2}$ denotes the horizontal fractional Laplacian, with coefficient $\nu_h>0$ and $\alpha\in(0,2)$. We assume that there is no vertical dissipation ($\nu_z=0$). 
$\T$ denotes the 1D periodic domain with length $1$. We further  impose the boundary condition:
\begin{equation}\label{eq:boundary}
	w(x,0,t)=w(x,1,t)=0,
\end{equation}
so together with \eqref{FPE-3}, the vertical velocity $w$ is uniquely determined by $u$  by 
\begin{equation}\label{eq:wexplicit}
w(x,z,t) = -\int_0^z \pa_x u(x,\tilde{z},t)\,\dd\tilde{z}.
\end{equation}
The system \eqref{FPE} can be viewed as the hydrostatic limit of the NSE with fractional horizontal dissipation, and the derivation follows analogously to \cite{azerad2001mathematical,li2019primitive,li2022primitive,furukawa2020rigorous}. 
\medskip

Our first set of results provides a sharp distinction between the local well-posedness and ill-posedness of solutions to the FPE \eqref{FPE}.

Recall that for the inviscid case ($\alpha=0$), linear instability was studied in \cite{renardy2009ill}, leading to ill-posedness of the linearized equation in any Sobolev space, with nonlinear instability further established in \cite{han2016ill}. In contrast, the viscous case ($\alpha=2$) is well-posed both locally and globally.

For the FPE, we identify $\alpha=1$ as the critical index marking the \emph{sharp} transition between these behaviors. Our results are summarized below:

\renewcommand{\arraystretch}{1.5}
\begin{table}[ht]
\caption{Local well-posedness and ill-posedness of FPE}\label{tab:local}
\begin{tabular}{cclccc}
	\hline\hline
	Inviscid PE & $\alpha=0$ & & & Ill-posedness & \\ \hline\hline	
	\multirow{4}{*}{Fractional PE} &$0<\alpha<1$ & \multicolumn{2}{l}{Supercritical regime} & Ill-posedness & Theorem \ref{thm:nonlinear-ill}\\ \cline{2-6}
	&\multirow{2}{*}{$\alpha=1$} & \multirow{2}{*}{Critical regime\quad} & $\|u_0\| \gg \nu_h$ & Ill-posedness & Theorem \ref{thm:critical-nonlinear-ill}\\ \cline{4-6}
	& &  & $\|u_0\| \ll \nu_h$ & Well-posedness & Theorem \ref{thm:GWP:crit}\\ \cline{2-6}
	& $1<\alpha<2$ & \multicolumn{2}{l}{Subcritical regime} & Well-posedness & Table \ref{tab:global}\\ \hline\hline
	Viscous PE & $\alpha=2$ & & & Well-posedness & \\ \hline\hline	
\end{tabular}	
\end{table}

We emphasize that in the critical regime ($\alpha=1$), the well-posedness of solutions depends on the interplay between the horizontal viscosity coefficient $\nu_h$ and the size of the initial data $u_0$, measured in suitable norms. This reveals a new type of distinction between well-posedness and ill-posedness at the refined critical level.
All the ill-posedness results are analogous to that of the inviscid PE, driven by the Kelvin-Helmholtz instability. See the theorems quoted in the last column of Table \ref{tab:local} for detailed statements of the results. 
\medskip

Our next line of investigation concerns the well-posedness results. We aim to establish local well-posedness theory for FPE in the subcritical regime. Moreover, we seek to understand whether these local-in-time results can be extended to global ones.

To contrast with the ill-posedness results, we work on Sobolev spaces. A major challenge in propagating Sobolev regularity arises from the vertical transport term $w\pa_zu$. The explicit form \eqref{eq:wexplicit} of the vertical velocity $w$ incurs a loss of one $x$-derivative, which is a key source of instability and ill-posedness in the supercritical regime. In the subcritical regime, this loss may be compensated by the fractional horizontal dissipation. However, due to the absence of vertical dissipation, it becomes necessary to control the $z$-derivative in the vertical transport term through the horizontal dissipation, ideally by exploiting the incompressibility condition \eqref{FPE-3}.

To find an appropriate balance between the $x$- and $z$-derivatives, we introduce a class of anisotropic Sobolev spaces defined by
\[
\|f\|_k := \Big(\sum_{j=0}^kE_j(f)\Big)^{\frac12},\quad k\in\N,
\]
where the energy level $E_k(f)$ is given by
\[
E_k(f):=\sum_{j=0}^k\|\lh^{\frac\alpha2j}\pa_z^{k-j}f\|_{L^2}^2.
\]
Roughly speaking, this framework compensates one vertical derivative with $\frac{\alpha}{2}$ horizontal derivatives. It enables us to derive energy estimates and establish local well-posedness within this class of Sobolev spaces. A key component of the analysis is the development of nontrivial anisotropic estimates to control the vertical transport term, leveraging our proposed framework.

We further establish global well-posedness results, summarized in Table \ref{tab:global}. For a precise description, see the summary in Section \ref{sec:subsummary} and the theorems cited therein.
\begin{table}[ht]
\caption{Global well-posedness for the subcritical regime}\label{tab:global}
\begin{tabular}{ccccc}
	\hline\hline
	  & $\alpha=1$ & $1<\alpha<\frac65$ & $\frac65\leq\alpha<\frac32$ & $\frac32\leq\alpha<2$ \\ \hline\hline	
	Small initial data & $\qquad\checkmark\qquad$ & $\qquad\quad\checkmark\quad\qquad$ & $\qquad\quad\checkmark\quad\qquad$ & $\qquad\quad\checkmark\quad\qquad$\\ \hline
	\quad~General initial data\quad & $\times$ & $?$ & $\checkmark$ & $\checkmark$\\ \hline\hline	
\end{tabular}	
\end{table}

Global well-posedness with small initial data is established in the subcritical regime $\alpha > 1$ and in the critical regime $\alpha = 1$, complementing the local well-posedness results.

For general initial data, we obtain a Beale-Kato-Majda type regularity criterion 
\[
	\int_0^T \|\lh^{\frac{3-\alpha}2}\pa_zu(t)\|_{L^2}^2\,\dd t<\infty,
\]
which ensures the boundedness of $\|u(t)\|_k$ up to time $T$, for any $k\in\N$. The criterion is optimal with respect to the norm $\|\cdot\|_k$. The derivation relies critically on the use of a borderline fractional Leibniz rule (see Lemma \ref{lem:improvedLeib}).

When $\alpha\geq\frac32$, the criterion holds as a consequence of anisotropic energy estimates of $E_0(u)$ and $E_1(u)$. Hence, global regularity follows directly.

The case $\alpha<\frac32$ is supercritical with respect to the energy $E_k$. Remarkably, we obtain an improved estimate that breaks the energy scaling, allowing us to deduce the regularity criterion and establish global well-posedness for $\alpha\geq\frac65$. 

We conjecture that the threshold $\frac65$ can be further lowered. However, the global regularity of general solutions for $\alpha \in (1, \frac65)$ remains an open question for future investigation.

\subsection{Outline of the paper}
The rest of the paper is organized as follows. 
In Section \ref{sec:pre}, we set up the notations and provide some preliminaries that will be used throughout this work. In particular, we provide a detailed proof of the borderline fractional Leibniz rule on torus (Lemma~\ref{lem:improvedLeib}) in Appendix~\ref{sec:improvedLeib}.
In Section \ref{sec:super}, we study linear and nonlinear instabilities of the supercritical FPE system $(\alpha<1)$. 
In Section \ref{sec:sub}, we study local and global well-posedness theory for supercritical FPE system $(\alpha>1)$. 
Finally, in Section \ref{sec:critical}, we investigate the critical FPE system $(\alpha=1)$, and show that well-posedness and ill-posedness depend on the relative strength between the horizontal viscosity coefficient $\nu_h$ and the initial data $u_0$.

\section{Preliminaries}\label{sec:pre}
In this section, we introduce the notations and collect several useful lemmas that will be used throughout the analysis.

\subsection{Notations}
We use $(x,z)$ to represent the horizontal and vertical variables.

For $1 \leq p \leq \infty$, we denote $L^p = L^p(\Omega)$ the Lebesgue spaces of measurable functions $f$ on $\Omega$, with
\[
 \|f\|_{L^p} = \|f\|_{L^p(\Omega)} = \begin{cases}
 	\displaystyle\left(\int_{\Omega} |f(x,z)|^p \dxdz\right)^{1/p},& 1\leq p<\infty,\\
 	\,\,\displaystyle\underset{(x,z)\in\Omega}{\mathrm{ess~sup}} \,\,\, |f(x,z)|, & p=\infty.
 \end{cases}
\]
We denote by $\|\cdot\|_{L^p_x}$ and $\|\cdot\|_{L^p_z}$ for $L^p$ norms in $x$ and $z$ variables, respectively. 

The Sobolev space $H^1=H^1(\Omega)$ is defined with the norm
\[
	\|f\|_{H^1}^2 = \|f\|_{L^2}^2 + \|\pa_xf\|_{L^2}^2 + \|\pa_zf\|_{L^2}^2.
\]

For $s>0$, we denote by $\lh^s$ the horizontal fractional Laplacian, defined by
\[
\lh^s f(x,z) = \sum_{k\in \Z} |k|^s \hat{f}_k(z) e^{2\pi ikx},
\]
where $\{\hat{f}_k\}_{k\in\Z}$ are the Fourier coefficients
\begin{equation}\label{eqn:FC}
 \hat{f}_k(z) = \int_\T f(x,z)e^{-2\pi ikx}\,\dd x.   
\end{equation}
We also denote by $\mathcal D(\Lambda_h^s)$ the subspace of $L^2(\Omega)$ satisfying
\[
\mathcal D(\Lambda_h^s):=\left\{f\in L^2(\Omega): \|\Lambda_h^s f\|_{L^2} = \left(\int_0^1 \sum\limits_{k\in \mathbb Z} |k|^{2s} |\hat{f}_k(z)|^2 \,\dd z\right)^{\frac12}  < \infty  \right\}.
\]
Note that the operator $\lh^s$ commutes with spatial derivatives $\pa_x$ and $\pa_z$. Moreover, for $f,g\in \mathcal D(\Lambda_h^s)$ we have
\begin{equation}\label{eqn:l-commute}
    \int_\Omega \lh^s f \cdot g \dxdz = \int_\Omega f\cdot\lh^s g \dxdz.
\end{equation}

We denote by $\H$ the space of $L^2$ functions with zero mean in $z$, namely 
\[
  \H:= \Big\{f\in L^2(\Omega): \int_0^1 f(x,z)\,\dd z = 0,\,\, \forall~x\in\T\Big\}.
\]
\begin{lemma}\label{lem:H}
 Let $u$ be a solution to \eqref{FPE} with initial data $u_0\in\H$. Then for any $t\geq0$, we have $u(t)\in\H$.
\end{lemma}
\begin{proof}
 From the incompressibility condition \eqref{FPE-3} and the boundary condition \eqref{eq:boundary}, we obtain
\[
   \pa_x \int_0^1 u(x,z,t) \,\dd z =  \int_0^1 \pa_xu(x,z,t) \,\dd z = -\int_0^1 \pa_zw(x,z,t)\,\dd z = -w(x,1,t)+w(x,0,t)=0,
\]
for any $x\in\T$ and $t\geq0$. Taking primitive in $x$ yields
\[
 \int_0^1 u(x,z,t)\, \dd z = C(t),\quad \forall~x\in\T,\,\, t\geq0.
\]
Now we integrate \eqref{FPE-1} over $\Omega$ and get
\[
 \frac{\dd}{\dd t}\int_\Omega u\dxdz = \int_\Omega \Big((\pa_xu + \pa_zw)u-\pa_xp-\nu_h\lh^\alpha u\Big)\dxdz = 0.
\]
We then conclude with
\[
  C(t) = \int_\Omega u(x,z,t) \dxdz = \int_\Omega u_0(x,z) \dxdz=0.
\]
\end{proof}
Lemma \ref{lem:H} shows that the solution $u(t)$ of \eqref{FPE} stays in the space $\H$ if the initial condition $u_0\in\H$. Having zero mean in $z$ allows us to apply the Poincar\'e inequality:
\begin{equation}\label{eqn:zero-mean-z}
 \|u\|_{L^p_z}\leq C\|\pa_zu\|_{L^p_z},\quad 1\leq p\leq\infty.
\end{equation}

\begin{remark}\label{rem:Poincare}
 Similar to \eqref{eqn:zero-mean-z}, we have the Poincar\'e inequality on $w$
 \begin{equation}\label{eq:Poincarew}
   \|w\|_{L^p_z}\leq C\|\pa_zw\|_{L^p_z} = C\|\pa_xu\|_{L^p_z},\quad 1\leq p\leq\infty,
 \end{equation}
 thanks to the boundary condition \eqref{eq:boundary}.
 However, the Poincar\'e inequality does not necessarily hold for $\pa_z^ku$ with $k\geq1$. 
\end{remark}

For a Banach space $(X, \|\cdot\|_{X})$ and $p\in [1,\infty]$, we denote the Lebesgue spaces $ L^p(0,T; X)$ of functions $f:X\times[0,T]\to \R$ satisfying 
\[
\int_{0}^{T} \|f(t)\|_{X}^p \dd t  <\infty
\]
with the usual convention when $p = \infty$.

The universal constant $C$ appearing below may change from line to line. We also use the notation $a \lesssim b$ to represent $a\leq Cb$.

\subsection{Fractional Leibniz rules}
We will make use of the \emph{fractional Leibniz rules} (also known as fractional product estimates) in our energy estimates.
\begin{lemma}[Fractional Leibniz rule]\label{lem:Leibniz}
For $s\geq0$, and $\frac1{p_1}+\frac1{q_1}=\frac1{p_2}+\frac1{q_2}=\frac12$, we have 
\begin{equation}\label{eq:Leibniz}
\|\lh^s (fg)\|_{L^2_x} 
\lesssim \|f\|_{L^{p_1}_x} \|\lh^s g\|_{L^{q_1}_x} + \|\lh^s f\|_{L^{p_2}_x} \|g\|_{L^{q_2}_x}.
\end{equation}
\end{lemma}
The result follows from the classical Kato-Ponce commutator estimates \cite{kato1988commutator}. For a detailed proof adapted to the periodic domain $\T$, we refer the reader to \cite{benyi2025fractional} and references therein.

The following lemma presents an improved Leibniz rule, which plays a crucial role in handling the critical case $\alpha = 1$.

\begin{lemma}[Borderline fractional Leibniz rule]\label{lem:improvedLeib}
Let $f\in L^{\infty}(\T) \cap \mathcal{D}(\Lambda_h)$ and $h \in \mathcal{D}(\Lambda_h^{\frac{1}{2}})$. Then the following inequality holds:
\begin{equation} \label{eq:improvedLeib}
\|\lh^{\frac{1}{2}} (fg)\|_{L^2_x} \lesssim 
\|f\|_{L^{\infty}_x} \|\lh^{\frac{1}{2}} g\|_{L^2_x} + \|\lh f\|_{L^2_x} \|g\|_{L^2_x}.
\end{equation}
\end{lemma}

Note that applying Lemma \ref{lem:Leibniz} with $s=1/2$, $p_1 = p_2 = \infty$, and $q_1 = q_2 = 2$ would yield \eqref{eq:improvedLeib}, provided the inequality
$\|\lh^{\frac{1}{2}} f\|_{L^\infty_x} \lesssim \|\lh f\|_{L^2_x}$
holds. However, this inequality is false, as the borderline Sobolev embedding $ H^{1/2}(\T) \hookrightarrow L^\infty(\T)$ does not hold. To establish Lemma \ref{lem:improvedLeib}, we instead require an improved version of \eqref{eq:Leibniz}:
\begin{equation}\label{eq:BMO}
\|\lh^s (fg)\|_{L^2_x} 
\lesssim \|f\|_{L^\infty_x} \|\lh^s g\|_{L^2_x} + \|\lh^s f\|_{\mathrm{BMO}_x} \|g\|_{L^2_x},	
\end{equation}
where the term $\|\lh^s f\|_{L^\infty_x}$ is replaced by $\|\lh^s f\|_{\mathrm{BMO}_x}$, with $\mathrm{BMO}$ denoting the space of functions of bounded mean oscillation. 
Thanks to the embedding $H^{1/2}\hookrightarrow\mathrm{BMO}$, we can deduce \eqref{eq:improvedLeib} from \eqref{eq:BMO}.

The estimate \eqref{eq:BMO} has been established on the real line $x\in\R$, see, for instance, \cite{li2019kato}. We include a proof of Lemma \ref{lem:improvedLeib} in Appendix \ref{sec:improvedLeib} to adapt the inequality to our periodic setting $x\in\T$.

\section{Ill-posedness for the supercritical case}\label{sec:super}
In this section, we consider the FPE system \eqref{FPE} in the supercritical regime, where $\alpha\in(0,1)$.

For the inviscid PE system ($\alpha=0$), a linear ill-posedness theory has been established in \cite{renardy2009ill}, proving local-in-time ill-posedness in Sobolev spaces. This was later extended to a nonlinear ill-posedness theory in \cite{han2016ill}. We will show that similar ill-posedness results—both linear and nonlinear—also hold for the supercritical FPE system.

We start with the observation that the horizontal shear flow 
\[(u,w,p)=(U(z),0,0)\]
is a steady solution of system \eqref{FPE}. Considering a small perturbation $(\tilde u, \tilde w, \tilde p)$ around this steady solution, we obtain 
\begin{subequations}\label{pert-system-nonlinear}
    \begin{align}
        &\partial_t \tilde{u} + (U+\tilde{u}) \partial_x \tilde{u} + \tilde{w} \partial_z (U+\tilde{u}) + \partial_x \tilde{p} + \nu_h \Lambda_h^\alpha \tilde{u} = 0, \label{pert-1-nonlinear}
        \\
        &\partial_z \tilde{p} = 0,\label{pert-2-nonlinear}
        \\
        &\partial_x \tilde{u} + \partial_z \tilde{w} = 0. \label{pert-3-nonlinear}
    \end{align}
\end{subequations}
We will demonstrate that this steady shear flow is unstable under both linear and nonlinear perturbations.

\subsection{Linear ill-posedness}
We consider the linear part of system \eqref{pert-system-nonlinear}:
\begin{subequations}\label{pert-system}
    \begin{align}
        &\partial_t \tilde{u} + U \partial_x \tilde{u} + \tilde{w} \partial_z U + \partial_x \tilde{p} + \nu_h \Lambda_h^\alpha \tilde{u} = 0, \label{pert-1}
        \\
        &\partial_z \tilde{p} = 0,\label{pert-2}
        \\
        &\partial_x \tilde{u} + \partial_z \tilde{w} = 0. \label{pert-3}
    \end{align}
\end{subequations}
In this section, we follow closely to \cite{renardy2009ill} to prove the linear ill-posedness of system \eqref{pert-system}.

Thanks to the divergence free condition \eqref{pert-3}, we introduce a stream function $\psi$ such that 
\begin{equation*}
    (\tilde u, \tilde w) = (\pa_z\psi, -\pa_x\psi).
\end{equation*}
Differentiating \eqref{pert-1} with respect to $z$, we obtain an equation for $\psi$:
\begin{subequations}\label{sys:psi}
\begin{equation}\label{eq:psi}
    \pa_t \pa_{zz}\psi+ U(z)\, \pa_x\pa_{zz}\psi - U''(z)\, \pa_x\psi + \nu_h \Lambda_h^\alpha \pa_{zz}\psi = 0.
\end{equation}
subject to the initial condition
\begin{equation}\label{eq:psiinit}
	\psi(x,z,0) = \psi_0(x,z),	
\end{equation}
and the boundary condition
\begin{equation}\label{eq:psibdy}
	\psi(x,0,t) = \psi(x,1,t) = 0.
\end{equation}
\end{subequations}

We will obtain linear instability around a shear flow with the following property. 
\begin{lemma}[\cite{renardy2009ill}]
 There exists an analytic shear flow $U(z)$ such that the equation 
 \begin{equation}\label{eq:Ucond}
  \int_0^1 \big(U(z) - i\gamma\big)^{-2}\,\dd z=0	
 \end{equation}
 for some $\gamma>0$.
\end{lemma}

An explicit example of the shear flow is $U(z)=\tanh(L(z-\frac12))$ with large enough $L$, see \cite{chen1991sufficient}. 
In general, as discussed in \cite[Lemma 1]{renardy2009ill}, any flow $U$ that is odd with respect to $z=\frac12$ and satisfies 
\[\int_0^1 U(z)^{-2}\,\dd z<\infty\]
will satisfy \eqref{eq:Ucond}. Note that the integrability prevents $U$ from being smooth at $z=\frac12$. But a smooth approximation of $U$ would also satisfy \eqref{eq:Ucond}. 

\begin{proposition}\label{prop:supercritical-ill}
    Let $\alpha\in(0,1)$. Consider the system \eqref{sys:psi} with a shear flow $U$ satisfying \eqref{eq:Ucond}.
    Then, for each $n\in \Z_+$, the system has a solution of the form 
    \begin{equation}\label{eq:psis}
        \psi_n(x,z,t)=\chi(z) e^{2\pi inx} e^{n \beta_n t},
    \end{equation}
    where the analytic function $\chi(z)$ is given by
    \begin{equation}\label{eqn:chi}
        \chi(z) := (U(z)- i\gamma)\int_{0}^z (U(y) - i\gamma)^{-2} \,\dd y,
    \end{equation}    
    and the parameter $\beta_n$ satisfies
    \begin{equation}\label{eq:betan}
    	\beta_n = 2\pi\gamma - (2\pi)^\alpha\nu_h n^{-(1-\alpha)},	
    \end{equation}
	which is strictly positive when $n$ is large enough.
\end{proposition}

\begin{proof}
    Let $n\in\Z_+$. Inserting \eqref{eq:psis} into \eqref{eq:psi}, we obtain
    \begin{equation}\label{eqn:ill-1}
        \big(n\beta_n + 2\pi inU(z) +  \nu_h (2\pi n)^\alpha\big) \chi''(z) -  2\pi in U''(z) \chi(z) = 0,
    \end{equation}
and the boundary condition \eqref{eq:psibdy} implies
 \[\chi(0)=\chi(1)=0.\]
 
Equation \eqref{eqn:ill-1} has the form of the hydrostatic Orr-Sommerfeld equation
\[
\big(U(z)-c\big) \chi''(z) -  U''(z) \chi(z) = 0,\quad\text{where}\quad c = i\left(\frac{\beta_n}{2\pi}+\frac{\nu_h}{(2\pi n)^{1-\alpha}}\right).
\]
The general solutions to the equation read:
\[\chi(z)= \big(U(z)-c\big)\Big(k_1 + k_2\int_0^z\big(U(y)-c\big)^{-2}\,\dd y\Big).\]
The boundary condition $\chi(0)=0$ implies $k_1=0$. The other boundary condition $\chi(1)=0$ is satisfied thanks to \eqref{eq:Ucond}, when taking $c=i\gamma$, or equivalently, $\beta_n$ satisfies \eqref{eq:betan}.

Since $\gamma>0$ and $\alpha<1$, from the expression \eqref{eq:betan} we have $\beta_n>0$ when $n$ is sufficiently large.
\end{proof}

Proposition~\ref{prop:supercritical-ill} shows that the high frequency part of the solutions to \eqref{sys:psi} exhibits fast growth of order $e^{\mathcal{O}(nt)}$. This leads to the Kelvin-Helmholtz type instability. 

\begin{theorem}\label{thm:illposedness}
	Let $s\geq0$. There exists a solution $\psi$ to the linearized system \eqref{pert-system}, such that $\psi_0\in H^s$, but $\psi(t)\not\in H^s$ for any $t>0$.
\end{theorem}
\begin{proof}
  Define
  \[\psi(x,z,t) = \sum_{n=1}^\infty a_n\text{Re}\psi_n(x,z,t)=\text{Re}\left( \chi(z)\sum_{n=1}^\infty a_ne^{2\pi inx}e^{n\beta_n t}\right),\quad a_n=\frac{1}{n^{s+1}}.\]
  From Proposition~\ref{prop:supercritical-ill}, we know $\psi$ is a solution of \eqref{pert-system}. 
  For $t=0$, we have 
  \[\|\psi_0\|_{H^s}^2\lesssim \|\chi\|_{H^s_z}^2\sum_{n=1}^\infty a_n^2 n^{2s}\lesssim \sum_{n=1}^\infty n^{-2}<\infty.\]
  On the other hand, for $t>0$, we have
  \[\|\psi(\cdot,z,t)\|_{H^s_x}^2 = |\chi(z)|^2\sum_{n=1}^\infty a_n^2 n^{2s} e^{2n\beta_nt}=\infty,\]
  for any $z\in[-1,1]$. Hence, $\psi(t)\not\in H^s$.
\end{proof}

Note that the linear instability can be extended to any Gevrey class $G^\sigma$ with order $\sigma>1$, by choosing $a_n=e^{-n^{1/\sigma}}$ in Theorem \ref{thm:illposedness}. To avoid the instability, one needs to consider analytical functions. 

\subsection{Nonlinear ill-posedness}\label{subsec:sup-non-ill}
The linear instability result can be extended to the nonlinear system \eqref{pert-system-nonlinear}. 
Differentiating \eqref{pert-system-nonlinear} with respect to $z$, we obtain an equation for $\psi$:
\begin{equation}\label{eq:psifully}
    \pa_t \pa_{zz}\psi+ U(z)\, \pa_x\pa_{zz}\psi - U''(z)\, \pa_x\psi + \nu_h \Lambda_h^\alpha \pa_{zz}\psi = -\pa_z\psi\, \pa_x\pa_{zz}\psi + \pa_x\psi \pa_{zzz}\psi.
\end{equation}
The equation lies in the abstract framework for nonlinear instability, introduced in \cite{han2016ill}, on $\omega:=\partial_{zz}\psi$, which takes the form
\begin{equation}\label{eq:omeganl}
  \pa_t\omega - \mathcal{L}\omega = \mathcal{Q}(\omega, \omega),
\end{equation}
where
\[
 \mathcal{L}\omega = -U\pa_{x}\omega + U'' \pa_x \psi - \nu_h \lh^\alpha \omega,\quad \mathcal{Q}(\omega,\omega) = -\pa_z\psi \pa_x \omega + \pa_x \psi \partial_z \omega.
\]
For any $n\in\Z_+$, let $\eps=\frac1n$ and apply the rescaling $(\s,y)=(\frac t\eps, \frac x\eps)$. Under this transformation, the rescaled vorticity $\omega(\s,y,z)$ satisfies
\begin{equation}\label{eq:omegasy}
 \pa_\s \omega - L\omega = Q(\omega,\omega)+ R_1\omega,
\end{equation}
where the linear term can be decomposed into a leading order term $L$ and a remainder $R_1$, given by
\[
L\omega:= - U \pa_y \omega  + U'' \pa_y \psi,\quad R_1\omega := -\eps^{1-\alpha}\nu_h\Lambda_y^\alpha\omega,
\]
and the nonlinear term is
\[
Q(\omega,\omega) := -\pa_z\psi \pa_y \omega + \pa_y \psi \partial_z \omega.
\]

From Proposition~\ref{prop:supercritical-ill}, we know that the leading order linearized equation 
\[\pa_\s\omega-L\omega=0, \quad \omega_0(y,z) = \chi''(z)e^{2\pi iy}\] admits a growing solution
\begin{equation}\label{eq:growingsol}
 \omega(y,z,\s) = e^{L\s}\omega_0(y,z) = \chi''(z)e^{2\pi i y} e^{2\pi\gamma\s},
\end{equation}
indicating linear instability. We will show that the right-hand side of \eqref{eq:omegasy} remains controlled, thereby ensuring that the instability persists in the nonlinear system, with the same initial condition $\omega_0$.

To proceed, we follow closely to \cite{han2016ill}. Since there is a loss of derivatives in the quadratic term $Q$, we consider $\W=(\omega,\pa_y\omega, \pa_z\omega)^\top$, which solves
\begin{equation}\label{eq:vecW}
  \partial_s \W - \mathrm{L}\W = \mathrm{Q}(\W,\W) +  \varepsilon\mathrm{R}_1 \W,
\end{equation}
where
\begin{equation}\label{def-L-E} 
 \mathrm{L}: = \begin{pmatrix} L &0 &0 \\ 0 &  L &0\\ 0& -U' + U'' \pa_z \psi(\cdot) + U''' \psi(\cdot) & - U \partial_y \end{pmatrix} , \quad 
 \mathrm{R}_1:= -\varepsilon^{-\alpha}\nu_h\Lambda_y^\alpha \mathbb{I}_{3},
\end{equation}
and for any two vector fields $\vec{V} = (v_1, v_2, v_3)^\top$ and $\W = (w_1, w_2, w_3)^\top$, the quadratic term
\[
\mathrm{Q}(\vec{V},\W): = \begin{pmatrix} - \partial_z \psi (v_1) w_2 + \psi(v_2) w_3 \\ - \partial_z \psi (v_2)  w_2 - \partial_z \psi (v_1) \partial_y w_2 + \partial_y \psi(v_2) w_3 + \psi(v_2) \partial_y w_3\\  -v_1 w_2 - \partial_z \psi (v_1) \partial_z w_2 + \partial_z \psi(v_2) w_3 + \psi(v_2) \partial_z w_3 \end{pmatrix}.
\]
Here by convention $\psi(f)$ solves $\partial_z^2 \psi(f) = f$ with $\psi_{\vert_{z=0,1}} = 0$.

Now, we directly apply the abstract framework in \cite{han2016ill} to \eqref{eq:vecW}, using the analytic function space $X_{\delta, \delta'}$, equipped with the following norm on $f=f(y,z)$: 
\begin{equation}\label{vort-norm0}
 \|f\|_{\delta, \delta'}: =\sum_{n\in \mathbb{Z}} \sum_{k\ge 0} \|\pa_z^k \hat f_n(\cdot) \|_{L^2_z} \frac{|\delta'|^{k }}{ k!} e^{\delta |n| } ,
\end{equation}
for any $\delta, \delta'>0$, where $\hat{f}_n(z)$ is defined in \eqref{eqn:FC}. The estimates on $\mathrm{L}$ and $\mathrm{Q}$ follow exactly as in \cite{han2016ill}, verifying their hypotheses {\bf (H.1)}--{\bf (H.4)}. The remaining term $\mathrm{R}_1$ verifies {\bf (H.5)} as follows:
\[
  \|\mathrm{R}_1 \W\|_{\delta,\delta'} \lesssim \varepsilon^{-\alpha} \sum_{n\in \mathbb{Z}} \sum_{k\ge 0} \|\partial_z^k \hat {\W}_n(\cdot) \|_{L^2_z} \frac{|\delta'|^{k }}{ k!} |n|^{\alpha} e^{\delta |n| } \lesssim \varepsilon^{-\alpha} \|\partial_y \W\|_{\delta,\delta'}.
\]

\begin{remark}
  In {\bf (H.5)} of \cite{han2016ill}, the coefficient of the bound has order $1$. However, one can repeat their proof to improve {\bf (H.5)} with coefficient having order $\varepsilon^{-\alpha}$ for any $\alpha\in(0,1)$. Indeed, $\varepsilon \mathrm{R}_1$ has order $\varepsilon^{1-\alpha}$, which is a lower order term when $\alpha<1$.
\end{remark}

We obtain the following ill-posedness result analogous to \cite[Theorem 2.1]{han2016ill}.

\begin{theorem}\label{thm:nonlinear-ill}
    Let $\alpha\in(0,1)$ and $s\geq0$. Denote $\{\omega_n\}_{n\in\Z_+}$ the solution of the equation \eqref{eq:omeganl} with initial condition $\omega_{n0}(x,z)=\chi''(z)e^{2\pi inx}$. Then, we have
\begin{equation}\label{eq:illnl}
\lim_{n \to \infty} \frac{ \| \omega_n \|_{L^2([0,t_n] \times \Omega)}}{\| \omega_{n0} \|_{H^s(\Omega)}}=+\infty,
\end{equation}
with $t_n=\mathcal{O}(n^{-1}\log n)$ which goes to 0 as $n\to\infty$.
\end{theorem}

Note that \eqref{eq:illnl} indicates ultrafast growth of the solution $\omega_n$ for large frequency $n$. This behavior is also exhibited by the solution to the linearized system, specifically for $\pa_z^2 \psi_n$, where $\psi_n$ is explicitly given by \eqref{eq:psis}. Therefore, the nonlinear system inherits the same ill-posedness characteristics as the linearized one.

\section{Well-posedness for the subcritical case}\label{sec:sub}
In this section, we study the local and global well-posedness for system \eqref{FPE} with $\alpha\in(1,2)$.

\subsection{Hydrostatic vorticity and maximum principle}\label{sec:MP}
The hydrostatic vorticity $\omega=\pa_z u$ satisfies the drift-diffusion equaion:
\begin{equation}\label{eq:omega}
    \pa_t \omega + u\pa_x \omega + w \pa_z \omega + \nu_h \lh^\alpha \omega  = 0.
\end{equation}

Similar to the 2D NSE and 2D Euler equations, there is no vortex stretching term in \eqref{eq:omega}. Hence, we have the maximum principle on $\omega$.
\begin{proposition}[Maximum principle]\label{prop:MP}
	Suppose $\omega_0\in L^\infty$. Then
	\begin{equation}\label{eq:omegaMP}
	\|\omega(\cdot,t)\|_{L^\infty}\leq\|\omega_0\|_{L^\infty},\quad \forall~t\geq0.	
	\end{equation}
\end{proposition}
\begin{remark}\label{rem:3D}
The maximum principle \eqref{eq:omegaMP} does not hold in 3D due to vorticity stretching. The absence of an a priori $L^\infty$ bound on $\omega$ significantly increases the complexity of the 3D system, making its analysis more challenging.
\end{remark}

Thanks to \eqref{eqn:zero-mean-z}, one can apply the Poincar\'e inequality and deduce the a priori bound:
\begin{equation}\label{eq:uLinfty}
 \|u(\cdot,t)\|_{L^\infty}\lesssim \|\omega(\cdot,t)\|_{L^\infty}\leq\|\omega_0\|_{L^\infty},\quad \forall~t\geq0.
\end{equation}
The a priori bounds \eqref{eq:omegaMP} and \eqref{eq:uLinfty} will be used throughout the rest of the analysis.

\subsection{Energy levels}
To establish local and global well-posedness theory, we perform a priori energy estimates to the FPE system \eqref{FPE}.

We highlight a key difficulty: given the ill-posedness result in Theorem \ref{thm:illposedness}, the only viable approach is to leverage the strong horizontal dissipation $\lh^\alpha u$ to control the nonlinear transport, particularly the vertical transport term $w\pa_z u$, which involves $z$-derivatives. Since no vertical dissipation is present, a crucial challenge is to identify an effective mechanism that allows horizontal dissipation to regulate vertical derivatives.

To address this challenge, we introduce a sequence of discrete \emph{energy levels} on which energy estimates will be conducted. Define
\[E_0:=\|u\|_{L^2}^2,\quad E_1:=\|\lh^{\frac{\alpha}{2}}u\|_{L^2}^2+\|\pa_zu\|_{L^2}^2,\quad E_2:=\|\lh^{\alpha}u\|_{L^2}^2+\|\lh^{\frac{\alpha}{2}}\pa_zu\|_{L^2}^2+\|\pa_z^2u\|_{L^2}^2,\]
and in general
\begin{equation}\label{eq:Ek}
    E_k := \sum_{j=0}^k \|\lh^{\frac\alpha2 j} \partial_z^{k-j} u\|_{L^2}^2,\quad k\in\N.
\end{equation}
Roughly speaking, controlling one $z$-derivative requires $\frac\alpha2$ $x$-derivatives. 
We also denote $\Et_{k+1}$ as the dissipation corresponding to $E_k$, namely
\begin{equation}\label{eq:Etk}
    \Et_{k+1} := \sum_{j=0}^k \|\lh^{\frac\alpha2 (j+1)} \partial_z^{k-j} u\|_{L^2}^2,\quad k\in\N.
\end{equation}
Note that we have the relation
\[E_k = \Et_k+\|\pa_z^k u\|_{L^2}^2.\]

We will establish a priori bound on each discrete energy level, which we denoted by
\begin{equation}\label{eq:Yk}
 Y_k(t) := E_k(t) + \nu_h\int_0^t \Et_{k+1}(\tau)\,\dd \tau,\quad k\in\N.
\end{equation}

\begin{remark}\label{rem:contenergy}
 Given $T>0$, we call any norms obeying $Y_k\in L^\infty(0,T)$ to be at the energy level $\frac{k\alpha}2$. For instance, the following norms belongs to the energy level $\frac\alpha2$:
 \[
  \|\lh^{\frac\alpha2} u\|_{L^\infty(0,T; L^2)},\quad \|\lh^\alpha u\|_{L^2(0,T; L^2)},\quad \|\pa_zu\|_{L^\infty(0,T; L^2)}, \quad\text{and}\quad \|\lh^{\frac\alpha2} \pa_zu\|_{L^2(0,T; L^2)}.
 \]
 The discrete energy levels can be extended to a continuous range of energy levels. For instance, the following norms belong to the energy level $s$:
  \[\|\lh^s u\|_{L^\infty(0,T; L^2)},\quad \|\lh^{s+\frac\alpha2} u\|_{L^2(0,T; L^2)},\quad \|\lh^{s-\frac\alpha2} \pa_zu\|_{L^\infty(0,T; L^2)}, \quad\|\lh^s \pa_zu\|_{L^2(0,T; L^2)},\quad\ldots\]
\end{remark}

\subsection{Energy estimate on \texorpdfstring{$E_0$}{E0}}
The a priori estimate on $E_0$ can be obtained by directly taking the $L^2$ inner product of \eqref{FPE-1} with $u$:
\begin{align*}
   \ddt\|u\|_{L^2}^2 & = -\int_\Omega \Big(u\pa_x(u^2) + w\pa_z(u^2) - 2p\pa_xu  + 2\nu_h u\lh^\alpha u\Big)\dxdz\\
   & = \int_\Omega \Big((\pa_xu+\pa_zw)u^2+2w\pa_zp - 2\nu_h(\lh^{\frac{\alpha}2}u)^2\Big)\dxdz = -2\nu_h\|\lh^{\frac{\alpha}2}u\|_{L^2}^2,
\end{align*}
where we have used the relations \eqref{FPE-2}, \eqref{FPE-3} and the boundary condition \eqref{eq:boundary}. This implies 
\begin{align}\label{est:bound-u-1}
    \|u(t)\|_{L^2}^2 + 2\nu_h\int_0^t \|\lh^{\frac\alpha2} u(s)\|_{L^2}^2 \dd s = \|u_0\|_{L^2}^2,
\end{align}
for any $t\geq 0$. Therefore, for any $T>0$,
\begin{equation}\label{bound:u-1}
    u\in L^\infty(0, T; \mathrm{H})\cap L^2(0, T; \D(\lh^{\frac\alpha2})).
\end{equation}
We have the a priori bounds at the energy level $0$:
\begin{equation}\label{eq:Energy0}
 E_0\in L^\infty(0,T),\quad  \Et_1\in L^1(0,T),	 \quad \text{and therefore}\quad Y_0 \in L^\infty(0,T).
\end{equation}

\subsection{Energy estimate on \texorpdfstring{$E_1$}{E1} and the strong solution}
Next, we perform an a priori energy estimate at the level of $Y_1$. The estimate of $E_1$ consists of two components, which we estimate separately.

For $\|\lh^{\frac\alpha2}u\|_{L^2}^2$, by taking the $L^2$ inner product of equation \eqref{FPE-1} with $\lh^{\alpha} u$, we obtain
\begin{align*}
    \frac12 \ddt \|\lh^{\frac\alpha2} u\|_{L^2}^2 + \nu_h \|\lh^{\alpha} u\|_{L^2}^2 = -\int_\Omega u\partial_x u \lh^{\alpha} u\dxdz -\int_\Omega w\omega \lh^{\alpha} u\dxdz =: \I_{11}+\I_{12}.
\end{align*}
For $\I_{11}$, we apply H\"older's inequality and obtain
\[\I_{11} \leq \|u\|_{L^\infty} \|\lh u\|_{L^2} \|\lh^{\alpha} u\|_{L^2}\lesssim \|\lh u\|_{L^2} \|\lh^{\alpha} u\|_{L^2},\]
where we absorb $\|u\|_{L^\infty}$ to a constant thanks to the a priori bound \eqref{eq:uLinfty}.
For $\I_{12}$, we have
\[
    \I_{12} \leq  \|w\|_{L^2}\|\omega\|_{L^\infty}\|\lh^{\alpha}u\|_{L^2}\lesssim \|\lh u\|_{L^2}\|\lh^\alpha u\|_{L^2},
\]    
where we absorb $\|\omega\|_{L^\infty}$ to a constant thanks to the a priori bound \eqref{eq:omegaMP}, and we use the Poincar\'e inequality to obtain
\[\|w\|_{L^2}^2=\int_\T\|w(x)\|_{L^2_z}^2\,\dd x\lesssim \int_\T\|\pa_zw(x)\|_{L^2_z}^2\,\dd x = \int_\T\|\pa_xu(x)\|_{L^2_z}^2\,\dd x = \|\lh u\|_{L^2}^2.\]

In the case when $\alpha\in(1,2)$, by the Gagliardo-Nirenberg interpolation inequality, we have
\[
   \|\lh u\|_{L^2} \lesssim \|\lh^{\frac\alpha2} u\|_{L^2}^{\frac{2\alpha-2}{\alpha}} \|\lh^\alpha u\|_{L^2}^{\frac2\alpha-1}.
\]
Applying Young's inequality, we obtain
\[\I_{11}+\I_{12}\leq C \|\lh^{\frac\alpha2} u\|_{L^2}^{\frac{2\alpha-2}{\alpha}}\|\lh^\alpha u\|_{L^2}^{\frac2\alpha}\leq \frac{\nu_h}{2}\|\lh^\alpha u\|_{L^2}^2 + C\|\lh^{\frac\alpha2} u\|_{L^2}^2,\]
where the first term can be absorbed by the dissipation, and the second term is integrable in time thanks to \eqref{est:bound-u-1}.
Indeed, by the Gr\"onwall inequality, we infer that for any $t\geq 0$
\begin{equation}\label{est:u-4}
    \|\lh^{\frac\alpha2} u(t)\|_{L^2}^2 + \nu_h \int_0^t \|\lh^{\alpha} u(s)\|_{L^2}^2 \,\dd s \leq \|\lh^{\frac\alpha2} u_0\|_{L^2}^2+C\int_0^t \|\lh^{\frac\alpha2} u(s)\|_{L^2}^2 \dd s \leq C,
\end{equation}
where the constant $C$ depends on $\|u_0\|_{L^2}$, $\|\lh^{\frac\alpha2} u_0\|_{L^2}$, $\|\omega_0\|_{L^\infty}$, and the parameter $\alpha$.
Hence, for any $T>0$,
\begin{equation}\label{bound:u-4}
     u\in L^\infty(0, T; \D(\lh^{\frac\alpha2}))\cap L^2(0, T; \D(\lh^{\alpha})).
\end{equation}

For $\|\omega\|_{L^2}^2$, by taking the $L^2$ inner product of equation \eqref{eq:omega} with $\omega$, we obtain
\[
 \ddt\|\omega\|_{L^2}^2=\int_\Omega \Big((\pa_xu+\pa_zw)\omega^2-2\nu_h(\lh^{\frac\alpha2}\omega)^2\Big)\dxdz = -2\nu_h\|\lh^{\frac\alpha2}\omega\|_{L^2}^2.
\]
which leads to the identity
\begin{equation}\label{est:omega-1}
    \|\omega(t)\|_{L^2}^2 + 2\nu_h\int_0^t \|\lh^{\frac\alpha2} \omega(s)\|_{L^2}^2 \,\dd s = \|\omega_0\|_{L^2}^2.
\end{equation}
Together with the maximum principle \eqref{eq:omegaMP}, for any $T>0$ we have
\begin{equation}\label{bound:omega-1}
    \omega\in L^\infty(0, T; L^\infty)\cap L^2(0, T; \D(\lh^{\frac\alpha2})).
\end{equation}

Combining \eqref{est:u-4} and \eqref{est:omega-1}, we conclude with the a priori  bounds at the energy level $\frac\alpha2$:
\begin{equation}\label{eq:Energy1}
 E_1\in L^\infty(0,T),\quad  \Et_2\in L^1(0,T), \quad\text{and therefore}\quad Y_1\in L^\infty(0,T).
\end{equation}

Solutions satisfying \eqref{eq:Energy1} are often referred to as \emph{strong solutions}, as such regularities can guarantee \eqref{FPE-1} holds almost everywhere. 
We state the definition below.
\begin{definition}[Strong solution]\label{def:strong}
    Let $u_0\in \mathcal D(\lh^{\frac\alpha2})\cap\H$, $\omega_0 = \partial_z u_0 \in L^\infty$, and $T>0$. We say $u$ is a \emph{strong solution} of \eqref{FPE} on $[0,T]$ with initial condition $u(0)=u_0$ if \eqref{FPE-1} holds a.e. in $\Omega\times [0,T]$. Moreover,
    \begin{align*}
        &u\in C([0,T];\H)\cap L^\infty(0, T; \D(\lh^{\frac\alpha2})\cap L^\infty)\cap L^2(0, T; \D(\lh^\alpha)),
        \\
        &\omega=\partial_z u\in L^\infty(0, T; L^\infty)\cap L^2(0, T; \D(\lh^{\frac\alpha2})).
    \end{align*}
\end{definition}

The a priori energy bounds in \eqref{eq:Energy1} allow us to construct global strong solutions via a viscous approximation. We now state the existence theorem and provide its proof.
\begin{theorem}[Global existence of strong solutions]\label{thm:sub-1}
    Let $u_0\in \mathcal D(\lh^{\frac\alpha2})\cap \H$ and 
    $\omega_0 \in L^\infty$. For any time $T>0$, there exists at least one strong solution to \eqref{FPE} with $u(0)=u_0$ on $[0,T]$. 
\end{theorem}

\begin{proof}
For each $\varepsilon\in (0,1)$, we consider the following regularization scheme
\begin{subequations}\label{FPE-reg}
    \begin{align}
        &\partial_t u^{\varepsilon} + u^{\varepsilon} \partial_x u^{\varepsilon} + w^{\varepsilon} \partial_z u^{\varepsilon} + \partial_x p^{\varepsilon} + \nu_h \Lambda_h^{\alpha} u^{\varepsilon} - \varepsilon \Delta u^{\varepsilon} = 0, \label{FPE-1-reg}
        \\
        &\partial_z p^{\varepsilon} = 0, \label{FPE-2-reg}
        \\
        & \partial_x u^{\varepsilon} + \partial_z w^{\varepsilon} = 0, \label{FPE-3-reg}
    \end{align}
\end{subequations}
with initial condition $u^{\varepsilon}(0) = \mathcal J_{\varepsilon} u_0$ and boundary condition $(w^\varepsilon, \partial_z u^\varepsilon)|_{z=\{0,1\}}=0$, where $\mathcal J_{\varepsilon}$ is the standard mollifying operator on $\Omega$. As the initial condition is smooth and the regularized system has full viscosity, following \cite{cao2007global} one has the global existence and uniqueness of strong solutions $u^\varepsilon$ for each $\varepsilon\in(0,1)$. By virtue of the \textit{a priori} estimates \eqref{bound:u-4} and \eqref{bound:omega-1}, we have the following uniform bounds: 
\begin{equation}\label{bound:u-reg}
     u^\varepsilon \text{ are uniformly bounded in }  L^\infty(0, T; L^\infty\cap \H\cap \mathcal D(\lh^{\frac\alpha2}))\cap L^2(0, T; \mathcal D(\lh^{\alpha})),
\end{equation}
\begin{equation}\label{bound:omega-reg}
     \omega^\varepsilon= \partial_z u^\varepsilon \text{ are uniformly bounded in } L^\infty(0, T; L^\infty)\cap L^2(0, T; \mathcal D(\Lambda_h^{\frac\alpha2})).
\end{equation}
By the Banach-Alaoglu theorem, we have
\begin{equation}\label{conv-u-weak}
    u^\varepsilon \rightharpoonup u \text{ weakly in } L^2(0,T; H^1\cap \H)
\end{equation}
with $u$ and $\omega=\partial_z u$ satisfying the same regularity as $u^\varepsilon$ and $\omega^\varepsilon$. 

In order to apply Aubin-Lions compactness lemma, we should also derive a uniform bound for $\partial_t u^\varepsilon$. Taking the inner product of \eqref{FPE-1-reg} with a test function $\phi \in H^1 \cap \H$, we have 
\begin{align*}
    |(\partial_t u^\varepsilon, \phi)| \leq &|(u^\varepsilon \partial_x u^\varepsilon, \phi)| + |(w^\varepsilon \partial_z u^\varepsilon, \phi)| + \nu_h |(\lh^\alpha u^\varepsilon, \phi)| + \varepsilon |(\lh u^\varepsilon, \Lambda_h \phi)| + \varepsilon |(\omega^\varepsilon,  \pa_z \phi)|
    \\
    \leq &C(\|u^\varepsilon\|_{L^\infty}^2 + \|\lh u^\varepsilon\|_{L^2} \|\omega^\varepsilon\|_{L^\infty} + \|\lh u^\varepsilon\|_{L^2} +  \|\omega^\varepsilon\|_{L^\infty}) \|\phi\|_{H^1}.
\end{align*}
Thanks to \eqref{bound:u-reg} and \eqref{bound:omega-reg}, we deduce from above that 
\begin{equation}\label{bound:ut-reg}
     \partial_t u^\varepsilon \text{ are uniformly bounded in }  L^2(0, T; H^{-1}),
\end{equation}
where $H^{-1}$ denotes the dual space of $H^1\cap \H$. Notice that \eqref{bound:u-reg} and \eqref{bound:omega-reg} implies that 
\begin{equation*}
     u^\varepsilon \text{ are uniformly bounded in }  L^\infty(0, T; \D(\lh^{\frac\alpha2})\cap L_x^2H_z^1\cap \H).
\end{equation*}
Since $\D(\lh^{\frac\alpha2})\cap L_x^2H_z^1\cap \H \hookrightarrow \H$ and $\H\hookrightarrow H^{-1}$ are compact, by invoking the Aubin-Lions compactness lemma we obtain the strong convergence
\begin{equation}\label{conv-u-strong}
    u^\varepsilon \to u \text{ strongly in } C([0,T],\H).
\end{equation}
The limit function $u$ along with its hydrostatic vorticity $\omega=\partial_z u$ satisfy the required regularities. 

Next we show that $u$ is indeed a strong solution. Consider any test function $\phi\in L^\infty(0,T; H^2\cap \H)$. From \eqref{FPE-1-reg} we have
\begin{align*}
    (\partial_t u^{\varepsilon} + u^{\varepsilon} \partial_x u^{\varepsilon} + w^{\varepsilon} \partial_z u^{\varepsilon}  + \nu_h \Lambda_h^{\alpha} u^{\varepsilon} - \varepsilon \Delta u^{\varepsilon}, \phi)=0,
\end{align*}
where the $L^2$ inner product is taken in both spatial and temporal variables. Thanks to \eqref{bound:ut-reg} and Banach Alaoglu theorem, we know that $\partial_t u^\varepsilon \rightharpoonup \partial_t u$ in $L^2(0,T; H^{-1})$. Therefore, $(\partial_t u^\varepsilon ,\phi) \to (\partial_t u ,\phi)$. Due to \eqref{bound:u-reg}, \eqref{bound:omega-reg}, and \eqref{conv-u-weak}, one has $(\nu_h \Lambda_h^{\alpha} u^{\varepsilon}, \phi) \to (\nu_h \Lambda_h^{\alpha} u, \phi)$ and $(\varepsilon \Delta u^\varepsilon, \phi) \to 0$. Thanks to \eqref{conv-u-strong}, we have
\begin{align*}
    |(u^\varepsilon \partial_x u^\varepsilon - u\partial_x u,\phi)| &= \frac12 |(u^\varepsilon)^2 - u^2 , \phi_x)| 
    \\
    &\leq C\|u^\varepsilon-u\|_{L^2(0,T;L^2)}  (\|u^\varepsilon\|_{L^\infty(0,T;L^\infty)} + \|u\|_{L^\infty(0,T;L^\infty)}) \|\phi\|_{L^2(0,T;H^1)} \to 0.
\end{align*}
For the nonlinear term $w^\varepsilon\partial_z u^\varepsilon$, denote $w=-\int_0^z u(x,\tilde{z})\,\dd\tilde{z}$. We have
\begin{align*}
    |(w^\varepsilon \partial_z u^\varepsilon - w\partial_z u,\phi)| \leq &|((w^\varepsilon-w)\partial_z u^\varepsilon,\phi)| + |w(\partial_z u^\varepsilon - \partial_z u),\phi)| := A_1 + A_2.
\end{align*}
By integration by parts and anisotropic estimates, one has
\begin{align*}
    A_1 \leq & \,|((\partial_x u^\varepsilon - \partial_x u) u^\varepsilon,\phi)|  + |((w^\varepsilon-w) u^\varepsilon,\partial_z\phi)|
    \\
    \leq &\, |((u^\varepsilon - u) \partial_x u^\varepsilon,\phi)| + |((u^\varepsilon - u) u^\varepsilon, \partial_x\phi)| 
    \\
    &+ |(\int_0^z (u^\varepsilon-u)\,\dd\tilde{z}\, \partial_x u^\varepsilon,\partial_z\phi)| + |(\int_0^z (u^\varepsilon-u)\,\dd\tilde{z}\, u^\varepsilon, \partial_{xz} \phi)|
    \\
    \leq & C \|u^\varepsilon-u\|_{L^2(0,T;L^2)} \|u^\varepsilon\|_{L^2(0,T;H^1)} \|\phi\|_{L^\infty(0,T; H^2)} \to 0.
\end{align*} 
For $A_2$, notice that we have $w\phi\in L^2(0,T;L^2)$. As $\pa_z u^\varepsilon \rightharpoonup \pa_z u$ weakly in $L^2(0,T; L^2)$, we conclude that $A_2\to 0$.
Therefore, we obtain
\begin{align}\label{equality:exist}
    (\partial_t u + u \partial_x u + w \partial_z u  + \nu_h \Lambda_h^{\alpha} u, \phi)=0
\end{align}
for $\phi\in L^\infty(0,T;H^2\cap \H)$. By virtue of the regularity of $u$, one can follow similar as in the estimate of $\partial_t u^\varepsilon$ to show that $\partial_t u \in L^2(0,T; \H)$ (note that we do not have the higher order term $\varepsilon \Delta u$). By a density argument, \eqref{equality:exist} holds for $\phi\in L^2(0,T;\H)$, and thus \eqref{FPE-1} holds a.e. in $\Omega\times [0,T]$. This finishes the proof.
\end{proof}

The estimates at the energy level $\frac\alpha2$ may not be sufficient to ensure the uniqueness of strong solutions for all $\alpha\in(1,2)$. The following theorem shows that uniqueness is ensured by the a priori control \eqref{eq:uniqueness}, which corresponds to the energy level ${\frac{3-\alpha}2}$ (see Remark \ref{rem:contenergy}). Therefore, uniqueness holds when $\alpha\geq\frac32$, since in this case ${\frac{3-\alpha}2}\leq\frac\alpha2$, which is a priori bounded. For $\alpha<\frac32$, however, an additional a priori bound \eqref{eq:uniqueness} at the energy level $\frac{3-\alpha}2$ is needed.

\begin{theorem}[Uniqueness of strong solutions]\label{thm:uniqueness}
Under the same assumptions in Theorem \ref{thm:sub-1}, we have:
\begin{itemize}
  \item For $\alpha\in[\frac32,2)$, the strong solution is unique. 
  \item For $\alpha\in(1,\frac32)$, there exists at most one strong solution provided that
  \begin{equation}\label{eq:uniqueness}
	\omega=\partial_z u\in L^\infty(0, T; \D(\lh^{\frac32-\alpha}))\cap L^2(0, T; \D(\lh^{\frac{3-\alpha}2})).  	
  \end{equation}
\end{itemize}
 
\end{theorem}
\begin{proof}
Let $(u_1, w_1, \omega_1, p_1)$ and $(u_2, w_2, \omega_2, p_2)$ be two strong solutions on $[0,T]$ with initial condition $u_1(0)=u_{10}$ and $u_2(0)=u_{20}$. Denote by $(u,w,\omega,p)=(u_1-u_2, w_1-w_2, \omega_1-\omega_2, p_1-p_2)$ and $u_0=u_{10}-u_{20}$. Then we have
\begin{subequations}	
\begin{align}
    &\pa_t u + u_1 \pa_x u + u \pa_x u_2 + w_1 \pa_z u + w \pa_z u_2 + \pa_x p + \nu_h \lh^\alpha u  = 0, \label{eqn-unique-1}
    \\
    &\partial_z p =0,
    \\
    &\partial_x u + \partial_z w = 0.
\end{align}
\end{subequations}

We perform a priori energy estimate on $\|u\|_{L^2}^2$. Take $L^2$ inner product of \eqref{eqn-unique-1} with $u$ to get 
\begin{equation}\label{est:unique-1}
    \frac12\ddt \|u\|_{L^2}^2 + \nu_h \|\lh^{\frac\alpha2} u\|_{L^2}^2 =
    - \int_\Omega \Big( u (u_1\pa_x+w_1\pa_z)u  + u^2\pa_x u_2 + w\omega_2 u \Big)\dxdz := B_0 + B_1 + B_2.
\end{equation}
The term $B_0$ vanishes by using integration by parts and the incompressibility \eqref{FPE-3} on $(u_1,w_1)$:
\[B_0 = -\int_\Omega (u_1\pa_x+w_1\pa_z)\big(\frac{u^2}{2}\big)\dxdz = \int_\Omega(\pa_xu_1+\pa_zw_1)\,\frac{u^2}{2}\dxdz = 0.\]
Next, for the term $B_1$, we apply anisotropic estimates, Poincar\'e's inequality, and Minkowski's inequality to obtain:
\begin{align}
    B_1 & = \int_\Omega \lh^{\frac{\alpha-1}2}(u^2)\cdot\lh^{\frac{3-\alpha}2}Hu_2 \dxdz 
    \leq \int_\T \|\lh^{\frac{\alpha-1}2}(u^2)\|_{L^1_z}\|\lh^{\frac{3-\alpha}2}Hu_2\|_{L^\infty_z}\,\dd x\nonumber\\
    &\lesssim \int_\T \|\lh^{\frac{\alpha-1}2}(u^2)\|_{L^1_z}\|\lh^{\frac{3-\alpha}2}H\omega_2\|_{L^2_z}\,\dd x
    \lesssim \Big\| \|\lh^{\frac{\alpha-1}2}(u^2)\|_{L^2_x}\Big\|_{L^1_z} \|\lh^{\frac{3-\alpha}2}\omega_2\|_{L^2},\label{eq:B1bound}
\end{align}
where $H=-\partial_x\Lambda_h^{-1}$ is the Hilbert transform in $x$-variable. We continue to apply the fractional Leibniz rule and interpolation inequalities to get:
\[
 \|\lh^{\frac{\alpha-1}2}(u^2)\|_{L^2_x}  \lesssim	\|u\|_{L^\infty_x}\|\lh^{\frac{\alpha-1}2}u\|_{L^2_x}\lesssim \|u\|_{L^2_x}\|\lh^{\frac\alpha2}u\|_{L^2_x},
\]
which leads to the control:
\[
	B_1  \lesssim \|u\|_{L^2}\|\lh^{\frac\alpha2}u\|_{L^2}  \|\lh^{\frac{3-\alpha}2}\omega_2\|_{L^2} \leq \frac14\nu_h\|\lh^{\frac\alpha2}u\|_{L^2}^2 + C\|\lh^{\frac{3-\alpha}2}\omega_2\|_{L^2}^2\|u\|_{L^2}^2.
\]
For the term $B_2$, let 
\[\beta = \begin{cases}
\frac{3-\alpha}{2}&\frac32\leq\alpha<2,\\ \frac\alpha2 &1<\alpha<\frac32.	
 \end{cases}
\]
As $w$ has zero mean in the $x$ direction, a similar procedure as \eqref{eq:B1bound} yields
\begin{align*}
B_2 & =-\int_\Omega \lh^\beta(\omega_2u)\cdot\lh^{-\beta} w \dxdz\lesssim \Big\| \|\lh^\beta(\omega_2u)\|_{L^2_x}\Big\|_{L^1_z} \|\lh^{1-\beta}u\|_{L^2}\\
& \lesssim \Big\|\Big(\|\omega_2\|_{L^\infty_x}\|\lh^\beta u\|_{L^2_x}+\|\lh^\beta\omega_2\|_{L^2_x}\|u\|_{L^\infty_x}\Big)\Big\|_{L^1_z} \|\lh^{1-\beta}u\|_{L^2}\\
& \lesssim \Big(\|\omega_2\|_{L^\infty} \|\lh^\beta u\|_{L^2} + \|\lh^\beta\omega_2\|_{L^2}\|u\|_{L^2}^{\frac{\alpha-1}\alpha}\|\lh^{\frac\alpha2} u\|_{L^2}^{\frac1\alpha}\Big)\, \|\lh^{1-\beta}u\|_{L^2} := B_{21}+B_{22}.
\end{align*}
Here, we point out that we cannot take $\beta=\frac{3-\alpha}{2}$ for $\alpha<\frac32$, since otherwise $\|\lh^\beta u\|_{L^2_x}$ in $B_{21}$ cannot be controlled by the dissipation. Further interpolation  yields
\begin{align*}
 B_{21} & \lesssim \|\lh^\beta u\|_{L^2}\|\lh^{1-\beta}u\|_{L^2} \leq \frac18\nu_h\|\lh^{\frac\alpha2}u\|_{L^2}^2 + C\|u\|_{L^2}^2,\\
 B_{22} & \lesssim \|\lh^\beta\omega_2\|_{L^2}\|u\|_{L^2}^{2-\frac{3-2\beta}\alpha}\|\lh^{\frac\alpha2}u\|_{L^2}^{\frac{3-2\beta}\alpha}\leq \frac18\nu_h\|\lh^{\frac\alpha2}u\|_{L^2}^2 + C\|\lh^\beta\omega_2\|_{L^2}^{\frac{2\alpha}{2\alpha+2\beta-3}}\|u\|_{L^2}^2
\end{align*}
Collecting all the estimates and applying Gr\"Gronwall's inequality, we conclude with
\[
 \|u(T)\|_{L^2}^2\leq \|u_0\|_{L^2}^2\exp\int_0^T C \Big( \|\lh^{\frac{3-\alpha}2}\omega_2(t)\|_{L^2}^2 + \|\lh^\beta\omega_2(t)\|_{L^2}^{\frac{2\alpha}{2\alpha+2\beta-3}} \Big)\,\dd t,
\]
as long as the integral is finite. This implies continuous dependence on the initial condition. In particular, when $u_{10}=u_{20}$, we conclude the uniqueness of solutions.

We would like to highlight that from the definition of $\beta$, we have 
\[
\|\lh^\beta\omega_2(t)\|_{L^2}^{\frac{2\alpha}{2\alpha+2\beta-3}}=\begin{cases}
 \|\lh^{\frac{3-\alpha}2}\omega_2(t)\|_{L^2}^2&\frac32\leq\alpha<2,\\ \|\lh^{\frac\alpha2} \omega_2(t)\|_{L^2}^{\frac{2\alpha}{3\alpha-3}}&1<\alpha<\frac32.\end{cases}
\]
Therefore, when $\alpha\in[\frac32,2)$, uniqueness is guaranteed as long as
\[\int_0^T\|\lh^{\frac{3-\alpha}2}\omega_2(t)\|_{L^2}^2\,\dd t<\infty,\]
which is clearly true as $\omega_2$ is a strong solution so that $\omega_2\in L^2(0,T; \D(\lh^{\frac\alpha2}))$.

When $\alpha\in(1,\frac32)$, we apply interpolation and obtain
\begin{align*}
\int_0^T\|\lh^{\frac\alpha2} \omega_2(t)\|_{L^2}^{\frac{2\alpha}{3\alpha-3}}\,\dd t & \leq \int_0^T\|\lh^{\frac32-\alpha} \omega_2(t)\|_{L^2}^{\frac{2(3-2\alpha)}{3\alpha-3}}\|\lh^{\frac{3-\alpha}2} \omega_2(t)\|_{L^2}^2\,\dd t\\
& \leq\Big\|\|\lh^{\frac32-\alpha} \omega_2(t)\|_{L^2}^2\Big\|_{L^\infty_t}^{\frac{3-2\alpha}{3\alpha-3}}	\int_0^T\|\lh^{\frac{3-\alpha}2} \omega_2(t)\|_{L^2}^2\,\dd t.
\end{align*}
Hence, the integral is bounded as long as \eqref{eq:uniqueness} holds.
\end{proof}

When $\alpha\geq\frac32$, we have $\frac\alpha2\geq\frac{3-\alpha}2$. Thus, the a priori bound \eqref{eq:Energy1} at the energy level $\frac\alpha2$ is sufficient to guarantee uniqueness. However, for $\alpha < \frac{3}{2}$, the bound \eqref{eq:Energy1} is no longer sufficient, and uniqueness of strong solutions may fail. In this case, we need to consider higher regularity. It is enough to study the next discrete energy level $\alpha$, since $\frac\alpha2 > \frac{3}{2}-\alpha$ for any $\alpha > 1$, which we will address next.

\subsection{Energy estimate on \texorpdfstring{$E_2$}{E2}}
Recall the energy at level $\alpha$:
\[E_2 = \|\lh^\alpha u\|_{L^2}^2 + \|\lh^{\frac\alpha2}\omega\|_{L^2}^2 + \|\pa_z\omega\|_{L^2}^2,\]
and the corresponding dissipation
\[\Et_3 = \|\lh^{\frac{3\alpha}{2}} u\|_{L^2}^2 + \|\lh^\alpha\omega\|_{L^2}^2 + \|\lh^{\frac\alpha2}\pa_z\omega\|_{L^2}^2.\]

We are going to establish the following energy bound, offering a sufficient control of the nonlinear drift term. 
\begin{proposition}\label{thm:E2}
Under the regularity criterion
\begin{equation}\label{eq:BKM}
	\int_0^T \|\lh^{\frac{3-\alpha}2}\omega(t)\|_{L^2}^2\,\dd t<\infty,
\end{equation}
we have the control at the energy level $\alpha$:
\begin{equation}\label{eq:Energy2}
 E_2 \in L^\infty(0,T),\quad \Et_3\in L^1(0,T), \quad \text{and therefore}\quad Y_2\in L^\infty(0,T).
\end{equation} 
\end{proposition}

\begin{remark}\label{rem:BKM}
 The regularity criterion \eqref{eq:BKM} corresponds to the control of the energy level $\frac{3-\alpha}{2}$, the same as the uniqueness requirement \eqref{eq:uniqueness}. This condition is crucial for the global well-posedness and  propagation of higher regularity.
  
Proposition~\ref{thm:E2} establishes that an a priori control of the energy level $\frac{3-\alpha}{2}$ leads to the control at the higher energy level $\alpha$. Moreover, as we will see later in Proposition~\ref{prop:Ek}, the same criterion \eqref{eq:BKM} enables the propagation of higher energies at the levels $\frac{k\alpha}2$ for any $k\in\N$.
\end{remark}

To prove Proposition~\ref{thm:E2}, we first establish the following weaker bound, which follows directly from standard energy estimates.  We note that a direct application of Gr\"onwall's inequality to \eqref{eq:E2bound} would yield Proposition~\ref{thm:E2}, provided that $\|w\|_{L^\infty}$ is a priori bounded. This nontrivial control will be established later in Proposition \ref{prop:winf}.

\begin{proposition}\label{prop:E2}
 The following a priori energy bound holds:
\begin{equation}\label{eq:E2bound}
	\ddt E_2 + \nu_h\Et_3\lesssim \Big(\|\omega_0\|_{L^\infty}^\frac{\alpha}{\alpha-1}+\|\omega_0\|_{L^\infty}^2+\|w\|_{L^\infty}^2+\|\lh^{\frac{3-\alpha}2}\omega\|_{L^2}^2\Big) E_2.	
\end{equation}
\end{proposition}

\begin{proof}
	We estimate the three terms in $E_2$ one by one.

    Let us start with the first component $\|\lh^\alpha u\|_{L^2}^2$. By taking the $L^2$ inner product of equation \eqref{FPE-1} with $\lh^{2\alpha} u$, we obtain
    \[
        \frac12\ddt \|\lh^\alpha u\|_{L^2}^2 + \nu_h \|\lh^{\frac{3\alpha}{2}} u\|_{L^2}^2 = -\int_{\Omega} \lh^{\alpha}u \cdot \lh^{\alpha}(u\pa_xu + w\pa_zu) \dxdz := \II_{11} + \II_{12}.
    \]
    For $\II_{11}$, we have 
    \begin{align*}
        \II_{11} = &-\int_{\Omega} \lh^{\alpha}u \cdot \lh^{\alpha}(u\pa_xu) \dxdz = -\frac12 \int_{\Omega} \lh^{\frac{3\alpha}2}u \cdot \lh^{\frac\alpha2}\pa_x(u^2) \dxdz 
        \\
        \lesssim &~ \|\lh^{\frac{3\alpha}2} u\|_{L^2} \|\lh^{\frac\alpha2+1} u\|_{L^2} \|u\|_{L^\infty} 
        \lesssim\|\omega_0\|_{L^\infty}\|\lh^{\frac{3\alpha}{2}} u\|_{L^2}^{\frac2\alpha}\|\lh^\alpha u\|_{L^2}^{2-\frac2\alpha} \\
        \leq &~ \frac1{4} \nu_h \|\lh^{\frac{3\alpha}{2}} u\|_{L^2}^2 + C\|\omega_0\|_{L^\infty}^{\frac\alpha{\alpha-1}}  \|\lh^\alpha u\|_{L^2}^2,
    \end{align*}
    where we have used the fractional Leibniz rule, the a priori bound \eqref{eq:uLinfty} on $\|u\|_{L^\infty}$, and Young's inequality. The first term on the right-hand side can be absorbed by dissipation, and the second term can be handled by Gr\"onwall inequality. Here, we keep track of $\|\omega_0\|_{L^\infty}$ dependency for the purpose of proving Theorem \ref{thm:GWPsmall} later on.
  
    The term $\II_{12}$ involves vertical drift,  requiring a far more delicate treatment. 
    \begin{align*}
        \II_{12} = &-\int_{\Omega} \lh^{\alpha}u \cdot \lh^{\alpha}(w\omega) \dxdz = -\int_{\Omega} \lh^{\frac32\alpha}u \cdot \lh^{\frac\alpha2}(w\omega) \dxdz 
        \\
        \lesssim &~ \|\lh^{\frac{3\alpha}{2}} u\|_{L^2}\big( \|\lh^{\frac\alpha2} w\|_{L^2} \|\omega\|_{L^\infty} + \|\lh^{\frac\alpha2} \omega\|_{L^2} \| w\|_{L^\infty}\big):=\II_{121}+\II_{122}.
    \end{align*}
    Note that $\|\omega\|_{L^\infty}$ has an a priori bound given by \eqref{eq:omegaMP}. Moreover, we apply the Poincar\'e inequality in $z$ on $\|\lh^{\frac\alpha2}w\|_{L^2}$ and obtain: 
    \begin{align*}
	\II_{121}\lesssim &~ \|\omega_0\|_{L^\infty}\|\lh^{\frac{3\alpha}{2}} u\|_{L^2} \|\lh^{\frac\alpha2}\pa_xu\|_{L^2}\lesssim \|\omega_0\|_{L^\infty}\|\lh^{\frac{3\alpha}{2}} u\|_{L^2}^{\frac2\alpha} \|\lh^\alpha u\|_{L^2}^{2-\frac2\alpha}\\
	\leq&~\frac18\nu_h \|\lh^{\frac{3\alpha}{2}} u\|_{L^2}^2+C\|\omega_0\|_{L^\infty}^\frac{\alpha}{\alpha-1}\|\lh^\alpha u\|_{L^2}^2,  	
    \end{align*}
	where we have used interpolation and Young's inequality. 
    For the second term $\II_{122}$, we have the estimate:    
    \begin{align*}
        \II_{122} \leq &~ \frac18 \nu_h \|\lh^{\frac{3\alpha}{2}} u\|_{L^2}^2 + C\|w\|_{L^\infty}^2 \|\lh^{\frac\alpha2}\omega\|_{L^2}^2.
    \end{align*}
    Combining the estimates above, we have 
    \begin{equation}\label{eq:E2u}
        \ddt \|\lh^\alpha u\|_{L^2}^2 + \nu_h \|\lh^{\frac{3\alpha}{2}} u\|_{L^2}^2  \lesssim \big(\|\omega_0\|_{L^\infty}^{\frac{\alpha}{\alpha-1}}+\|w\|_{L^\infty}^2\big) E_2.
    \end{equation}

	Next, we estimate the second component of the energy $\|\lh^{\frac\alpha2}\omega\|_{L^2}^2$. By taking the $L^2$ inner product of equation \eqref{eq:omega} with $\lh^{\alpha} \omega$, we obtain
   \[
       \frac12\ddt \|\lh^{\frac\alpha2} \omega\|_{L^2}^2 + \nu_h \|\lh^{\alpha} \omega\|_{L^2}^2 =  -\int_\Omega \lh^{\alpha} \omega \cdot \big(u\pa_x\omega + w\pa_z\omega\big) \dxdz := \II_{21}+ \II_{22}.
   \]
   The term $\II_{21}$ can be estimated in a similar manner to $\II_{121}$:   \begin{align*}
       \II_{21} \leq \|u\|_{L^\infty} \|\lh^{\alpha} \omega\|_{L^2} \|\pa_x\omega\|_{L^2} \leq \frac14\nu_h \|\lh^{\alpha} \omega\|_{L^2}^2 + C\|\omega_0\|_{L^\infty}^{\frac{\alpha}{\alpha-1}}\|\lh^{\frac\alpha2}\omega\|_{L^2}^2.
   \end{align*}
   For $\II_{22}$, we have
   \begin{align*}
       \II_{22} \leq &~\|w\|_{L^\infty} \|\lh^{\alpha} \omega\|_{L^2}  \|\pa_z\omega\|_{L^2}
       \leq \frac14\nu_h \|\lh^{\alpha} \omega\|_{L^2}^2 + C\|w\|_{L^\infty}^2 \|\pa_z\omega\|_{L^2}^2.
   \end{align*}
    
    Combining the estimates above, we have
    \begin{equation}\label{eq:E2omega}
        \ddt \|\lh^{\frac\alpha2} \omega\|_{L^2}^2 + \nu_h \|\lh^\alpha \omega\|_{L^2}^2  \lesssim \big(\|\omega_0\|_{L^2}^{\frac{\alpha}{\alpha-1}}+\|w\|_{L^\infty}^2\big) E_2.
    \end{equation}

   Finally, we estimate the third component $\|\pa_z\omega\|_{L^2}$. Differentiating \eqref{eq:omega} in $z$, we obtain the dynamics of $\pa_z\omega$:
   \begin{equation}\label{eq:omegaz}
   	\pa_t\pa_z\omega + (u\pa_x+w\pa_z)\pa_z\omega + \omega\pa_x\omega -\pa_xu\pa_z\omega +\nu_h\lh^\alpha\pa_z\omega = 0.
   \end{equation}
	By taking the $L^2$ inner product of equation \eqref{eq:omegaz} with $\pa_z\omega$, we obtain 
   \begin{align*}
       \frac12\ddt \|\pa_z\omega\|_{L^2}^2 &+ \nu_h \|\lh^{\frac\alpha2} \pa_z\omega\|_{L^2}^2 = \int_\Omega\big(\pa_xu+\pa_zw\big)\cdot\frac{\omega^2}{2}\dxdz -
\int_\Omega \pa_z\omega \cdot \big(\omega\pa_x\omega -\pa_xu\pa_z\omega\big) \dxdz\\
&  = -\int_\Omega \omega \pa_x\omega \pa_z\omega\dxdz + \int_\Omega \pa_xu (\pa_z\omega)^2 \dxdz := \II_{31}+\II_{32}.
   \end{align*}
   
   The term $\II_{31}$ can be controlled thanks to the a priori bound \eqref{eq:omegaMP} on $\|\omega\|_{L^\infty}$: 
 \begin{align}
    \II_{31} & = -\frac12 \int_{\Omega} \pa_z\omega \,\pa_x (\omega^2) \dxdz = \frac12 \int_{\Omega} \lh^{\frac12} \pa_z\omega \cdot \lh^{\frac12} H (\omega^2) \dxdz\nonumber\\
    & \leq \frac12 \|\lh^{\frac12} \pa_z\omega\|_{L^2} \|\lh^{\frac12} (\omega^2)\|_{L^2}\lesssim  \|\omega\|_{L^\infty} \|\lh^{\frac12}\omega\|_{L^2} \|\lh^{\frac12}\pa_z\omega\|_{L^2}\nonumber\\
    & \leq \frac14\nu_h \|\lh^{\frac\alpha2}\pa_z\omega\|_{L^2}^2 + C\|\omega_0\|_{L^\infty}^2 \|\lh^{\frac\alpha2}\omega\|_{L^2}^2.\label{eq:II31}
\end{align}

   The term $\II_{32}$ requires a more delicate treatment. Note that a rough estimate yields
   \[\II_{32}\leq \|\pa_xu\|_{L^\infty}\|\pa_z\omega\|_{L^2}^2. \]
  The control of the right-hand side requires an a priori bound on $\|\pa_xu\|_{L^\infty}$, which is too strong for our purpose. We derive a refined estimate which only requires a weaker a priori bound, controlled by the energy $E_{\frac32}$. To achieve this, we will apply an anisotropic estimate and make use of the dissipation. 
\begin{align*}
	\II_{32} = &~ -\int_\Omega \lh^{\frac{\alpha-1}{2}}\big((\pa_z\omega)^2\big)\cdot\lh^{\frac{3-\alpha}{2}}Hu\dxdz
	\leq \int_\T\|\lh^{\frac{\alpha-1}{2}}\big((\pa_z\omega)^2\big)\|_{L^1_z}\|\lh^{\frac{3-\alpha}{2}}Hu\|_{L^\infty_z}\,\dd x\\
	\lesssim &~ \int_\T\|\lh^{\frac{\alpha-1}{2}}\big((\pa_z\omega)^2\big)\|_{L^1_z}\|\lh^{\frac{3-\alpha}{2}}H\omega\|_{L^2_z}\,\dd x
	\lesssim \Big\|\|\lh^{\frac{\alpha-1}{2}}\big((\pa_z\omega)^2\big)\|_{L^1_z}\Big\|_{L^2_x}\|\lh^{\frac{3-\alpha}{2}}H\omega\|_{L^2}\\
	\lesssim &~ \Big\|\|\lh^{\frac{\alpha-1}{2}}\big((\pa_z\omega)^2\big)\|_{L^2_x}\Big\|_{L^1_z}\|\lh^{\frac{3-\alpha}{2}}\omega\|_{L^2}.
\end{align*}
For the first term, we apply the fractional Leibniz rule and interpolation to get:
\[
 \|\lh^{\frac{\alpha-1}{2}}\big((\pa_z\omega)^2\big)\|_{L^2_x} \leq \|\lh^{\frac{\alpha-1}{2}}(\pa_z\omega)\|_{L^2_x}\|\pa_z\omega\|_{L^\infty_x} \lesssim \|\lh^{\frac\alpha2}\pa_z\omega\|_{L^2_x}\|\pa_z\omega\|_{L^2_x}.
\]
We continue our estimate:
\begin{align*}
	\II_{32} \lesssim &~ \Big\|\|\lh^{\frac\alpha2}\pa_z\omega\|_{L^2_x}\|\pa_z\omega\|_{L^2_x}\Big\|_{L^1_z} \|\lh^{\frac{3-\alpha}{2}}\omega\|_{L^2}
	\leq \|\lh^{\frac\alpha2}\pa_z\omega\|_{L^2}\|\pa_z\omega\|_{L^2}\|\lh^{\frac{3-\alpha}{2}}\omega\|_{L^2}\\
	\leq &~ \frac14\nu_h\|\lh^{\frac\alpha2}\pa_z\omega\|_{L^2}^2+C\|\lh^{\frac{3-\alpha}{2}}\omega\|_{L^2}^2\|\pa_z\omega\|_{L^2}^2.
\end{align*}

Combining the estimates above, we have
    \begin{equation}\label{eq:E2omegaz}
        \ddt \|\pa_z\omega\|_{L^2}^2 + \nu_h \|\lh^{\frac\alpha2} \pa_z\omega\|_{L^2}^2  \lesssim \big(\|\omega_0\|_{L^\infty}^2+\|\lh^{\frac{3-\alpha}{2}}\omega\|_{L^2}^2\big) E_2.
    \end{equation}

Now, we collect the estimates on the three energy components \eqref{eq:E2u}, \eqref{eq:E2omega} and \eqref{eq:E2omegaz}. It yields the desired estimate \eqref{eq:E2bound}.
\end{proof}

From Proposition \ref{prop:E2}, we apply Gr\"onwall's inequality and obtain the control \eqref{eq:Energy2}, provided the regularity criterion
\[
\int_0^T\Big(\|w(t)\|_{L^\infty}^2+\|\lh^{\frac{3-\alpha}2}\omega(t)\|_{L^2}^2\Big)\,\dd t<\infty.
\]
In particular, we know $\|\lh^{\frac{3-\alpha}2}\omega\|_{L^2(0,T;L^2)}$ corresponds to the energy level $\frac{3-\alpha}{2}$. We argue that the remaining term $\|w\|_{L^2(0,T;L^\infty)}$ is controlled by the energy level $\frac{3-\alpha}{2}+\epsilon$, with arbitrarily small $\epsilon>0$. Indeed, we use the Poincar\'e and Minkowski inequalities and Sobolev embeddings to obtain  
  \begin{equation}\label{eq:wenergylevel}
\|w\|_{L^\infty}\lesssim \big\|\|\pa_zw\|_{L^2_z}\big\|_{L^\infty_x}\lesssim \big\|\|\pa_xu\|_{L^\infty_x}\big\|_{L^2_z}\lesssim\|\lh^{\frac32+\epsilon}u\|_{L^2},	
  \end{equation}
 and $\|\lh^{\frac32+\epsilon}u\|_{L^2(0,T;L^2)}$ belongs to the energy level $\frac{3-\alpha}{2}+\epsilon$. 
 
 The presence of the extra $\epsilon$ is due to the borderline Sobolev embedding. In the following Proposition, we remove the $\epsilon$ and obtain the stronger regularity criterion \eqref{eq:BKM}, which belongs to the energy level $\frac{3-\alpha}{2}$.  
 \begin{proposition}\label{prop:winf}
 	Suppose $\|\pa_xu_0\|_{L^2}<\infty$ and the regularity criterion \eqref{eq:BKM} holds. Then we have 
 	\begin{equation}\label{eq:winfbound}
		\int_0^T\|w(t)\|_{L^\infty}^2\,\dd t<\infty. 		
 	\end{equation}
 \end{proposition}
 \begin{proof}
  The main idea of obtaining \eqref{eq:winfbound} is to perform an energy estimate on an intermediate energy level between $\frac\alpha2$ and $\alpha$. To proceed, we choose the energy level $1$, and estimate $\|\pa_xu\|_{L^2}^2$.
  
  Take the  $L^2$ inner product of \eqref{FPE-1} with $-\pa_x^2u$ to get
    \begin{align*}
        &\frac12\ddt \|\pa_x u\|_{L^2}^2 + \nu_h \|\lh^{1+\frac{\alpha}{2}} u\|_{L^2}^2 = -\int_{\Omega} \pa_xu \cdot \pa_x(u\pa_xu + w\pa_zu) \dxdz\\
        & = -\int_\Omega \Big(\pa_xu\cdot (u\pa_x+w\pa_z)(\pa_xu) + (\pa_xu)^3 + \pa_xu \pa_xw \pa_zu \Big)\dxdz:= J_{10} + J_{11} + J_{12}.
    \end{align*}
    
  The term $J_{10}$ vanishes by incompressibility \eqref{FPE-3}:
  \begin{equation}\label{eq:J10}
J_{10}=-\int_\Omega (u\pa_x+w\pa_z)\Big(\frac{(\pa_xu)^2}{2}\Big)\dxdz = \int_\Omega(\pa_xu+\pa_zw)\,\frac{(\pa_xu)^2}{2}\dxdz = 0.  	
  \end{equation}
  The term $J_{11}$ can be handled by:
\[
 J_{11}=-\int_\Omega (\pa_xu)^3 \dxdz = \int \lh^{\frac{3-\alpha}2} Hu \cdot\lh^{\frac{\alpha-1}2}\big((\pa_xu)^2\big) \dxdz \lesssim \|\lh^{\frac{3-\alpha}2} \omega\|_{L^2} \big\|\lh^{\frac{\alpha-1}2}\big((\pa_xu)^2\big) \big\|_{L^2},
\]
applying the Poincar\'e inequality in $z$. Then, we further use Leibniz's rule and interpolation to estimate
\[
 \big\|\lh^{\frac{\alpha-1}2}\big((\pa_xu)^2\big)
 \big\|_{L^2}  \lesssim \|\pa_xu\|_{L^\infty} \|\lh^{\frac{\alpha+1}2} u\|_{L^2} \lesssim
 \|\pa_xu\|_{L^2} \|\lh^{1+\frac\alpha2} u\|_{L^2},
\]
and it yields
\begin{equation}\label{eq:J11}
 J_{11}\lesssim \|\lh^{\frac{3-\alpha}2} \omega\|_{L^2}	\|\pa_xu\|_{L^2} \|\lh^{1+\frac\alpha2} u\|_{L^2}\leq \frac14\nu_h\|\lh^{1+\frac\alpha2} u\|_{L^2}^2 + C\|\lh^{\frac{3-\alpha}2} \omega\|_{L^2}^2 \|\pa_xu\|_{L^2}^2.
\end{equation}
Next, for $J_{12}$, we apply anisotropic estimates:
\begin{align}
 J_{12} & = -\int_\Omega \pa_xw\,\omega\,\pa_xu \dxdz = \int_\Omega \lh^{\frac12}Hw\cdot \lh^{\frac12}(\omega\pa_xu)\dxdz 
 \leq \int_\T \|\lh^{\frac12}Hw\|_{L^\infty_z} \|\lh^{\frac12}(\omega\pa_xu)\|_{L^1_z}\,\dd x\nonumber\\
 & \lesssim  \Big\|\|\lh^{\frac32}u\|_{L^2_z}\Big\|_{L^p_x} \Big\|\|\lh^{\frac12}(\omega\pa_xu)\|_{L^1_z}\Big\|_{L^q_x} 
 \lesssim \Big\|\|\lh^{\frac32}u\|_{L^p_x}\Big\|_{L^2_z} \Big\|\|\lh^{\frac12}(\omega\pa_xu)\|_{L^q_x}\Big\|_{L^1_z},\label{eq:J12-1}
\end{align}
where we choose $\frac1p = 1-\frac\alpha2$ and $\frac1q=\frac\alpha2$.
By making use of Sobolev embeddings and interpolation inequalities, we have
\[\|\lh^{\frac32}u\|_{L^p_x}\lesssim\|\lh^{1+\frac\alpha2}u\|_{L^2_x}\]
and
\[
  \|\lh^{\frac12} (\omega \pa_xu)\|_{L^q_x} \lesssim \|\omega\|_{L^\infty_x} \|\lh^{\frac32} u\|_{L^{\frac2\alpha}_x} + \|\lh^{\frac12} \omega\|_{L^{\frac{2}{\alpha-1}}_x} \|\pa_xu\|_{L^2_x}\lesssim \|\lh^{2-\frac\alpha2}u\|_{L^2_x}+\|\lh^{\frac{3-\alpha}2} \omega\|_{L^2_x}\|\pa_xu\|_{L^2_x}. 
\]
Therefore,
\begin{align}
 J_{12} &\lesssim \|\lh^{1+\frac\alpha2}u\|_{L^2}	\Big(\|\lh^{2-\frac\alpha2}u\|_{L^2}+\|\lh^{\frac{3-\alpha}2} \omega\|_{L^2}\|\pa_xu\|_{L^2}\Big)\nonumber\\
 &\leq \frac14\nu_h\|\lh^{1+\frac\alpha2} u\|_{L^2}^2 + C\big(1+\|\lh^{\frac{3-\alpha}2} \omega\|_{L^2}^2\big) \|\pa_xu\|_{L^2}^2.\label{eq:J12}
\end{align}

Combining the estimates \eqref{eq:J10}, \eqref{eq:J11} and \eqref{eq:J12} would yield:
\[
\ddt \|\pa_x u\|_{L^2}^2 + \nu_h \|\lh^{1+\frac{\alpha}{2}} u\|_{L^2}^2 \lesssim \big(1+\|\lh^{\frac{3-\alpha}2} \omega\|_{L^2}^2\big) \|\pa_xu\|_{L^2}^2.
\]
Applying Gr\"onwall's inequality, and the criterion \eqref{eq:BKM}, we obtain the bounds:
\[
  \pa_xu \in L^\infty(0,T; L^2),\quad \text{and}\quad
  \lh^{1+\frac{\alpha}{2}} u \in L^2(0,T; L^2).
\]
Finally, using the estimate \eqref{eq:wenergylevel} with $\epsilon=\frac{\alpha-1}2$, we conclude with
\[\int_0^T\|w(t)\|_{L^\infty}^2\,\dd t\lesssim \int_0^T\|\lh^{1+\frac{\alpha}{2}} u(t)\|_{L^2}^2\,\dd t=\|\lh^{1+\frac{\alpha}{2}} u\|_{L^2(0,T;L^2)}^2<\infty.\]
\end{proof}

Proposition~\ref{thm:E2} follows directly from Proposition~\ref{prop:E2} and Proposition~\ref{prop:winf}.

\subsection{Local and global well-posedness for classical solutions}
In this section, we establish the local and global well-posedness theory for the FPE system \eqref{FPE} at the energy level $\alpha$. We refer to solutions that lie in this energy level  as \emph{classical solutions}.

\begin{definition}[Classical solution]\label{def:classical}
    Let $T>0$, and 
    \begin{equation}\label{eq:smoothinit}
 u_0\in \D(\lh^\alpha)\cap\H, \quad \omega_0 = \pa_z u_0 \in \D(\lh^{\frac\alpha2})\cap L^\infty, \quad \pa_z\omega_0\in L^2.
    \end{equation}
	We say $u$ is a \emph{classical solution} of \eqref{FPE} on $[0,T]$ with initial condition $u(0)=u_0$ if 
    \begin{align*}
        &u\in C([0,T];\D(\lh^\alpha)\cap\H)\cap L^2(0, T; \D(\lh^{\frac{3\alpha}2})),
        \\
        &\omega=\pa_z u\in L^\infty(0, T; L^\infty)\cap C([0, T];  \D(\lh^{\frac\alpha2}))\cap L^2(0, T; \D(\lh^\alpha)),\\
        &\pa_z\omega\in C([0, T]; L^2)\cap L^2(0, T; \D(\lh^{\frac\alpha2})).        
    \end{align*}
\end{definition}

We first state the uniqueness result for classical solutions, which is a direct consequence of Theorem \ref{thm:uniqueness}.
\begin{theorem}[Uniqueness of classical solution]\label{thm:uniqueness2}
 For $\alpha\in(1,2)$, classical solutions to the FPE system \eqref{FPE} are unique.	
\end{theorem}

Next, we focus on the existence of classical solutions, starting with a local existence result.
\begin{theorem}[Local existence of classical solution]\label{thm:LWP}
	Let $\alpha\in(1,2)$. Suppose $u_0$ satisfies \eqref{eq:smoothinit}. Then there exists a time $T>0$ such that classical solution to \eqref{FPE} with $u(0)=u_0$ exists on $[0,T]$. 
\end{theorem}
\begin{proof}
 We carry out a priori energy estimates. Applying \eqref{eq:wenergylevel} with $\epsilon=\frac{\alpha-1}{2}$ and get
 \[\|w\|_{L^\infty}^2+\|\lh^{\frac{3-\alpha}2}\omega\|_{L^2}^2\lesssim \|\lh^{\frac\alpha2+1}u\|_{L^2}^2+\|\lh\omega\|_{L^2}^2\lesssim E_2^{\frac{2(\alpha-1)}{\alpha}}\Et_3^{\frac{2-\alpha}{\alpha}},\]
where we have used Poincar\'e and interpolation inequalities, respectively.
Plugging into \eqref{eq:E2bound} and applying Young's inequality, we have the a priori bound
\[
 \ddt E_2 + \nu_h\Et_3\lesssim (\|\omega_0\|_{L^\infty}^{\frac\alpha{\alpha-1}}+\|\omega_0\|_{L^\infty}^2)E_2+E_2^{1+\frac{2(\alpha-1)}{\alpha}}\Et_3^{\frac{2-\alpha}{\alpha}}\leq \frac12\nu_h\Et_3+C\big(E_2+E_2^{\frac{3\alpha-2}{2(\alpha-1)}}\big).
\]
By standard Cauchy-Lipschitz theory, there exists a time $T>0$ such that \eqref{eq:Energy2} holds, namely
\[E_2\in C([0,T]),\quad\text{and}\quad \Et_3\in L^1(0,T).\]
With the a priori bounds, we can construct a classical solution via an argument similar to that in Theorem~\ref{thm:sub-1}.
\end{proof}

Global existence for small initial data then follows as an immediate consequence.
\begin{theorem}[Global existence of classical solution: small initial data]\label{thm:GWPsmall}
	Let $\alpha\in(1,2)$. Suppose $u_0$ satisfies \eqref{eq:smoothinit}. There exists a small constant $\delta\in(0,1)$, such that if
		\[E_2(0)=\|\lh^\alpha u_0\|_{L^2}^2 + \|\lh^{\frac\alpha2}\omega_0\|_{L^2}^2 + \|\pa_z\omega_0\|_{L^2}^2\leq\delta,\quad\text{and}\quad\|\omega_0\|^2_{L^\infty}\leq\delta,\]
	 then a classical solution 
     to \eqref{FPE} with $u(0)=u_0$ exists globally in time, with
	 \[E_2(t)\leq\delta,\quad\forall~t\geq0.\] 
\end{theorem}
\begin{proof}
    Recall the estimate \eqref{eq:E2bound}:
    \[
    	\ddt E_2 + \nu_h\Et_3\lesssim \Big(\|\omega_0\|_{L^\infty}^\frac{\alpha}{\alpha-1}+\|\omega_0\|_{L^\infty}^2+\|w\|_{L^\infty}^2+\|\lh^{\frac{3-\alpha}2}\omega\|_{L^2}^2\Big) \big(\Et_2 + \|\pa_z\omega\|_{L^2}^2\big).
    \]
    We revisit the proof of Proposition~\ref{prop:E2} and observe that the terms $\II_{22}$ and $\II_{32}$ that involves $\|\pa_z\omega\|_{L^2}^2$ do not depend on $\omega_0$. This leads to a refined estimate:
    \[
    	\ddt E_2 + \nu_h\Et_3\leq C\Big(\|\omega_0\|_{L^\infty}^\frac{\alpha}{\alpha-1}+\|\omega_0\|_{L^\infty}^2\Big)\Et_2 + C\Big(\|w\|_{L^\infty}^2+\|\lh^{\frac{3-\alpha}2}\omega\|_{L^2}^2\Big) E_2.
    \]
    Moreover, we have the following two Poincar\'e inequalities:
    \[
      \|w\|_{L^\infty}^2+\|\lh^{\frac{3-\alpha}2}\omega\|_{L^2}^2\leq C_1 \Et_3,\quad\text{and}\quad
      \Et_2\leq C_2\Et_3.
    \]
	Note that we need to separate $\|\pa_z\omega\|_{L^2}^2$ from $E_2$ as the Poincar\'e inequality $E_2\lesssim \Et_3$ does not apply, due to the fact that $\pa_z \omega$ does not necessarily have a zero mean in the $x$-variable.
      Then, we obtain
	\[
	 \ddt E_2\leq -\nu_h \Et_3 + C\cdot(\delta^{\frac{\alpha}{2(\alpha-1)}}+\delta)\cdot C_2\Et_3 + C\cdot C_1\Et_3\cdot E_2.
	\]
	Pick a small $\delta<\min\{1,\frac{\nu_h}{C(C_1+2C_2)}\}$. Then, whenever $E_2(t)=\delta$, we have
	\[
	 \ddt E_2\leq \Big(-\nu_h+ C(C_1+2C_2)\delta\Big)\Et_3<0.
	\]
	Therefore, $E_2$ cannot go beyond $\delta$.
\end{proof}

Next, we move to global existence of classical solutions for large initial data. According to Proposition~\ref{thm:E2}, global existence is conditional on the regularity criterion \eqref{eq:BKM} for any $T>0$, which requires control of the energy level $\frac{3-\alpha}{2}$. This criterion is automatically satisfied when 
\[\frac\alpha2 \geq \frac{3-\alpha}{2},\quad\text{or equivalently}\quad \alpha\geq\frac32,\]
thanks to the a priori bound \eqref{eq:Energy1} on the energy level  $\frac\alpha2$. We call this case the \emph{energy subcritical / critical} regime, in which global well-posedness follows. 

\begin{theorem}[Global existence of classical solution: large initial data]\label{thm:GWPlarge}
	Let $\alpha\in[\frac32,2)$. Suppose $u_0$ satisfies \eqref{eq:smoothinit}. Then there exists a global classical solution  
    to \eqref{FPE} with $u(0)=u_0$.
\end{theorem}
\begin{proof}
We apply the Poincar\'e inequality and the estimate \eqref{bound:omega-1} to verify the regularity criterion \eqref{eq:BKM}:
\[\int_0^T \|\lh^{\frac{3-\alpha}2}\omega(t)\|_{L^2}^2\,\dd t\lesssim \int_0^T \|\lh^{\frac\alpha2}\omega(t)\|_{L^2}^2\,\dd t<\infty,\]
for any $T>0$. Then, Proposition~\ref{thm:E2} yields the energy control \eqref{eq:Energy2}. Global existence follows.
 \end{proof}

\subsection{Higher order energy estimate}
In this section, we derive energy bounds at higher energy levels, analogous to Proposition~\ref{thm:E2}. Notably, the regularity criterion \eqref{eq:BKM} remains sufficient to propagate higher regularity. To this end, we define \emph{smooth solutions} of order $k\geq3$ as follows:

 \begin{definition}[Smooth solution of order $k$]\label{def:smooth}
    Let $T>0$. Suppose that the initial data satisfy
   \begin{equation}\label{eq:smoothinitial}
   u_0\in\H,\quad \omega_0\in L^\infty,\quad\text{and}\quad E_j(0) < \infty,\quad\forall~j=0,1,\ldots,k.
   \end{equation}
 We say $u$ is a smooth solution of order $k$ of \eqref{FPE} on $[0,T]$ with initial condition $u(0)=u_0$ if $u$ is a classical solution and 
 \[Y_j\in L^\infty(0,T), \quad\forall~j=0,1,\ldots,k.
 \]
\end{definition}

Note that when $k = 1$ and $k = 2$, this definition is compatible with strong solutions and classical solutions, respectively.

\begin{theorem}[Global well-posedness of smooth solutions]\label{thm:smooth}
 Let $k\geq3$ and $T > 0$. Suppose the regularity criterion \eqref{eq:BKM} holds, and the initial data satisfies \eqref{eq:smoothinitial}. Then, there exists a unique smooth solution of order $k$ of \eqref{FPE} on $[0,T]$.
\end{theorem}

The proof of Theorem \ref{thm:smooth} is based on the following proposition:

\begin{proposition}[Higher order regularity]\label{prop:Ek}
 Let $k\geq 2$. Suppose the regularity criterion \eqref{eq:BKM} holds, and the initial data satisfies 
 \[\omega_0\in L^\infty,\quad\text{and}\quad
 E_j(0) < \infty,\quad \forall~j=0,1,\dots,k.\]
 Then we have the control at the energy level $\frac{k\alpha}2$, namely
 \begin{equation}\label{eq:Energyk}
 E_k\in L^\infty(0,T),\quad  \Et_{k+1}\in L^1(0,T),	 \quad \text{and therefore}\quad Y_k \in L^\infty(0,T),
\end{equation}
for any $T>0$.
 \end{proposition}

\begin{proof}
Recall that 
\[
 E_k = \sum_{j=0}^k \|\lh^{\frac{(k-j)\alpha}2} \partial_z^{j} u\|_{L^2}^2,\quad 
 \Et_{k+1} = \sum_{j=0}^k \|\lh^{\frac{(k-j+1)\alpha}2} \partial_z^{j} u\|_{L^2}^2.
\]
We have proved the case when $k=2$ in Proposition~\ref{thm:E2}. We will use induction to prove the cases when $k\geq 3$. To this end, assume that $E_j\in L^\infty(0,T)$ and $\Et_{j+1}\in L^1(0,T)$ for all $j = 0,1,\dots,k-1$, and we denote by
\begin{equation}\label{eq:induction}
 F_{k-1} = \sum_{j=0}^{k-1} E_j \in L^\infty(0,T), \quad \Ft_k = \sum_{j=0}^{k-1} \Et_{j+1} \in L^1(0,T).
\end{equation}

We proceed to the estimate of $E_k$.

     For $j=0,\ldots, k$, the dynamics of $\pa_z^ju$ are described by 
 \begin{equation}\label{eq:omegazk}
     \pa_t \pa_z^ju + \pa_z^j\big(u\pa_x u + w \pa_z u\big) + \pa_z^j\pa_xp + \nu_h \pa_z^j\lh^\alpha u  = 0.
 \end{equation}
 By taking the $L^2$ inner product of \eqref{eq:omegazk} with $\lh^{(k-j)\alpha}\pa_z^ju$, we obtain
 \[
  \frac12\ddt\|\lh^{\frac{(k-j)\alpha}2}\pa_z^ju\|_{L^2}^2 + \nu_h \|\lh^{\frac{(k-j+1)\alpha}2}\pa_z^ju\|_{L^2}^2
  = -\int_\Omega \lh^{(k-j)\alpha}\pa_z^ju\cdot\pa_z^j\big(u\pa_x u + w \pa_z u\big)\dxdz:=\III_j.
 \]
Summing over all indices $j \in \left\{0, 1, \dots, k\right\}$, we have
 \[
 \frac12\ddt E_k + \nu_h \Et_{k+1} = \sum_{j=0}^k \III_j.
 \]
 
 Now, we estimate $\III_j$ term by term.
 
 First, the control of $\III_0$ follows analogously to that of the estimates of $\|\lh^\alpha u\|_{L^2}$, namely $\II_{11}$ and $\II_{12}$. We sketch the proof without details.
 \[
  \III_0 = - \int_\Omega \lh^{\frac{(k+1)\alpha}2}u\cdot\lh^{\frac{(k-1)\alpha}2}\big(u\pa_x u + w \pa_z u\big)\dxdz
  := \III_{01} + \III_{02},
 \]
where we have
 \begin{align*}
        \III_{01} = & -\frac12 \int_{\Omega} \lh^{\frac{(k+1)\alpha}2}u\cdot\lh^{\frac{(k-1)\alpha}2}\pa_x(u^2) \dxdz 
        \lesssim\|\lh^{\frac{(k+1)\alpha}2} u\|_{L^2} \|\lh^{\frac{(k-1)\alpha}2+1} u\|_{L^2} \|u\|_{L^\infty} \\
        \lesssim &~\|\lh^{\frac{(k+1)\alpha}2} u\|_{L^2}^{\frac2\alpha}\|\lh^{\frac{k\alpha}2} u\|_{L^2}^{2-\frac2\alpha}
        \leq \frac{\nu_h}{4(k+1)}  \|\lh^{\frac{(k+1)\alpha}2} u\|_{L^2}^2 + C\|\lh^{\frac{k\alpha}2} u\|_{L^2}^2,\\  
        \III_{02} \lesssim &~ \|\lh^{\frac{(k+1)\alpha}2} u\|_{L^2}\Big( \|\lh^{\frac{(k-1)\alpha}2} w\|_{L^2} \|\omega\|_{L^\infty} + \|\lh^{\frac{(k-1)\alpha}2} \omega\|_{L^2} \| w\|_{L^\infty}\Big)\\
        \leq &~ \frac{\nu_h}{4(k+1)} \|\lh^{\frac{(k+1)\alpha}2} u\|_{L^2}^2 + C\Big(\|\lh^{\frac{k\alpha}2} u\|_{L^2}^2+\|w\|_{L^\infty}^2\|\lh^{\frac{(k-1)\alpha}2} \omega\|_{L^2}^2\Big).
    \end{align*}
To conclude, we have
\begin{equation}\label{eq:III0}
 \III_0 \leq \frac{\nu_h}{2(k+1)} \Et_{k+1} + C\big(1+\|w\|_{L^\infty}^2\big)E_k.
\end{equation}

Next, we control $\III_j$ where $j=1,\ldots, k-1$. The estimate requires the control of intermediate terms, which needs to be carefully treated. To begin with, we apply \eqref{eqn:l-commute} to get
\[
 \III_j=  -\int_\Omega \lh^{\frac{(k-j+1)\alpha}2} \pa_z^j u \cdot \lh^{\frac{(k-j-1)\alpha}2} \pa_z^j (u\pa_x u + w \pa_z u) \dxdz:=\III_{j1}+\III_{j2},
\]
where, by H\"older's inequality, it follows that
 \begin{align*}
 	\III_{j1}\leq  &~ \|\lh^{\frac{(k-j+1)\alpha}2} \pa_z^j u\|_{L^2}\Big\|\lh^{\frac{(k-j-1)\alpha}2} \pa_z^j (u\pa_x u)\Big\|_{L^2}
 	\lesssim \Et_{k+1}^{\frac12} \Big\|\lh^{\frac{(k-j-1)\alpha}2+1} \pa_z^j (u^2)\Big\|_{L^2},\\
 	\III_{j2}\leq  &~ \|\lh^{\frac{(k-j+1)\alpha}2} \pa_z^j u\|_{L^2}\Big\|\lh^{\frac{(k-j-1)\alpha}2} \pa_z^j (w \pa_z u)\Big\|_{L^2}
 	\lesssim \Et_{k+1}^{\frac12} \Big\|\lh^{\frac{(k-j-1)\alpha}2} \pa_z^j (w \pa_z u)\Big\|_{L^2}.
 \end{align*} 
  
 To estimate $\III_{j1}$, we apply Leibniz rule in $z$ and decompose
 \[
 \pa_z^j (u^2) = \sum_{\ell=0}^j \binom{j}{\ell} \pa_z^{j-\ell} u \cdot \pa_z^\ell u.
 \]
 By fractional Leibniz rule in $x$, we have
\[
    \III_{j1}\lesssim \Et_{k+1}^{\frac12}\sum_{\ell=0}^j \|\lh^{\frac{ (k-j-1)\alpha}2+1} \pa_z^{j-\ell} u\|_{L^2} \|\pa_z^\ell u\|_{L^\infty} :=  \sum_{l=0}^j \III_{j1\ell}.
\]
We treat $\III_{j10}$ separately by:
\begin{align*}
\III_{j10} \leq &\,  \Et_{k+1}^{\frac12}\|\lh^{\frac{ (k-j-1)\alpha}2+1} \pa_z^{j} u\|_{L^2} \|u\|_{L^\infty}
\lesssim \Et_{k+1}^{\frac12}\|\lh^{\frac{ (k-j)\alpha}2} \pa_z^{j} u\|_{L^2}^{\frac{2(\alpha-1)}\alpha}\|\lh^{\frac{ (k-j+1)\alpha}2} \pa_z^j u\|_{L^2}^{1-\frac{2(\alpha-1)}\alpha}\\
\leq &\, E_k^{\frac{\alpha-1}\alpha}\Et_{k+1}^{\frac1\alpha} \leq \tfrac{\nu_h}{6(k+1)(j+1)} \Et_{k+1} + C E_k.
\end{align*}
Note that $\|\pa_z^\ell u\|_{L^\infty}$ is a priori bounded for $\ell=0, 1$ (recall \eqref{eq:uLinfty} and \eqref{eq:omegaMP}). For $\ell\geq 2$, we control the term as follows:
\begin{align}
 \|\pa_z^\ell u\|_{L^\infty}\lesssim &\, \Big\|\|\pa_z^\ell u\|_{L^2_z}+\|\pa_z^{\ell+1} u\|_{L^2_z}\Big\|_{L^\infty_x}
 \lesssim \Big\|\|\pa_z^\ell u\|_{L^\infty_x}\Big\|_{L^2_z} + \Big\|\|\pa_z^{\ell+1} u\|_{L^\infty_x}\Big\|_{L^2_z}\nonumber\\
 \lesssim &\, \|\pa_z^\ell u\|_{L^2} + \|\pa_z^{\ell+1} u\|_{L^2} + \|\lh^{\frac\alpha2}\pa_z^\ell u\|_{L^2} + \|\lh^{\frac12+\epsilon}\pa_z^{\ell+1} u\|_{L^2},\label{eq:uzinf}
\end{align}
for a small $\epsilon>0$. Note that $\pa_z^\ell u$ does not have zero mean in either $x$ or $z$ direction. Hence, we need to keep the lower order terms, i.e. the first three terms on the right-hand side of \eqref{eq:uzinf}.
Take $\epsilon=\frac{\alpha-1}2$. The estimate becomes
\begin{equation}\label{eq:uzinf2}	
\|\pa_z^\ell u\|_{L^\infty}\lesssim E_\ell^{\frac12} +E_{\ell+1}^{\frac12} + \Et_{\ell+1}^{\frac12}+\Et_{\ell+2}^{\frac12}\leq F_{\ell+1}^{\frac12} + \Ft_{\ell+2}^{\frac12},
\end{equation}
Therefore, the intermediate terms $\III_{j1\ell}$ with $\ell=1,\ldots,j-1$ have the bound
\[
 \III_{j1\ell}\lesssim \Et_{k+1}^{\frac12} E_k^{\frac12} \big(F_{k-1}^{\frac12} + \Ft_k^{\frac12}\big)\lesssim \tfrac{\nu_h}{6(k+1)(j+1)} \Et_{k+1} + C(F_{k-1}+\Ft_k)E_k,
\]
where we note that $\ell+2\leq k$ since $j\leq k-1$, and as $\ell\ge1$ we have 
\[ \|\lh^{\frac{ (k-j-1)\alpha}2+1} \pa_z^{j-\ell} u\|_{L^2}\lesssim \|\lh^{\frac{ (k-j+\ell)\alpha}2} \pa_z^{j-\ell} u\|_{L^2} \leq E_k^{\frac12}.\]
For the remaining term $\III_{j1j}$, we can bound
\[
\|\lh^{\frac{ (k-j-1)\alpha}2+1} u\|_{L^2}\lesssim \|\lh^{\frac{ (k-j+1)\alpha}2} u\|_{L^2}\leq \begin{cases}
	E_k^{\frac12} & j=1,\\
	F_{k-1}^{\frac12} & j\geq2.
 \end{cases}
\]
For $j=1$, since $\|\pa_zu\|_{L^\infty}$ is a priori bounded, we have 
\[\III_{111}\lesssim \Et_{k+1}^{\frac12}E_k^{\frac12}\leq \tfrac{\nu_h}{6(k+1)(1+1)} \Et_{k+1} + CE_k.\]
For $j=2,\ldots k-2$, we apply \eqref{eq:uzinf2} to get
\[
\|\pa_z^ju\|_{L^\infty}\lesssim F_{k-1}^{\frac12}+\Ft_k^{\frac12}\lesssim F_{k-1}^{\frac12}+E_k^{\frac12},
\]
and therefore
\[
 \III_{j1j} \lesssim \Et_{k+1}^{\frac12}F_{k-1}^{\frac12}(F_{k-1}^{\frac12}+\Ft_k^{\frac12})\leq\tfrac{\nu_h}{6(k+1)(j+1)} \Et_{k+1} + C\big(F_{k-1}^2+F_{k-1}E_k\big).
\]
For $j=k-1$, we use \eqref{eq:uzinf} with $\epsilon=\frac{\alpha-1}{4}$ and get
\begin{align*}
\|\pa_z^{k-1} u\|_{L^\infty}\lesssim &\,\|\pa_z^{k-1} u\|_{L^2} + \|\pa_z^k u\|_{L^2} + \|\lh^{\frac\alpha2}\pa_z^{k-1} u\|_{L^2} + \|\lh^{\frac{1+\alpha}4}\pa_z^k u\|_{L^2}\\
\lesssim &\, F_{k-1}^{\frac12} + E_k^{\frac12} + \|\lh^{\frac\alpha2}\pa_z^k u\|_{L^2}^{1-\frac{\alpha-1}{2\alpha}}\|\pa_z^k u\|_{L^2}^{\frac{\alpha-1}{2\alpha}}
\lesssim F_{k-1}^{\frac12} + E_k^{\frac12} + \Et_{k+1}^{\frac12-\frac{\alpha-1}{4\alpha}}E_k^{\frac{\alpha-1}{4\alpha}}.
\end{align*}
Hence, we obtain 
\begin{align*}
 \III_{(k-1)1(k-1)} \lesssim &\,\Et_{k+1}^{\frac12} F_{k-1}^{\frac12} \Big( F_{k-1}^{\frac12} + E_k^{\frac12} + \Et_{k+1}^{\frac12-\frac{\alpha-1}{4\alpha}}E_k^{\frac{\alpha-1}{4\alpha}} \Big)\\
 \leq &\, \tfrac{\nu_h}{6(k+1)k} \Et_{k+1} + C\Big( F_{k-1}^2 +\big(F_{k-1}+F_{k-1}^{\frac{2\alpha}{\alpha-1}}\big)E_k\Big).
\end{align*}
Collecting all the estimates, we conclude that for all $j=1,\dots,k-1$,
\begin{equation}\label{eq:IIIj1}
 \III_{j1}\leq \frac{\nu_h}{6(k+1)} \Et_{k+1} + C\Big( F_{k-1}^2 +\big(1+F_{k-1}+F_{k-1}^{\frac{2\alpha}{\alpha-1}}+\Ft_k\big)E_k\Big).
\end{equation}

Now we move to the estimate of $\III_{j2}$. Start with the Leibniz rule in $z$
 \[
 \pa_z^j (w\pa_zu) = \sum_{\ell=0}^j \binom{j}{\ell} \pa_z^{j-\ell} w \cdot \pa_z^{\ell+1} u = -\sum_{\ell=0}^{j-1} \binom{j}{\ell} \pa_x\pa_z^{j-\ell-1}u\cdot\pa_z^{\ell+1}u+w\pa_z^{j+1}u.
 \]
By the fractional Leibniz rule \eqref{eq:Leibniz}, we have
\begin{align*}
    \III_{j2}\lesssim &\, \Et_{k+1}^{\frac12}\sum_{\ell=0}^{j-1}\Big( \|\lh^{\frac{ (k-j-1)\alpha}2+1} \pa_z^{j-\ell-1} u\|_{L^2} \|\pa_z^{\ell+1} u\|_{L^\infty} +  
    \|\lh^{\frac{ (k-j-1)\alpha}2} \pa_z^{\ell+1} u\|_{L^\infty} \|\pa_x\pa_z^{j-\ell-1} u\|_{L^2} \Big)\\
    &\,+ \Et_{k+1}^{\frac12}\|\lh^{\frac{ (k-j-1)\alpha}2}(w\pa_z^{j+1}u)\|_{L^2} := \sum_{l=0}^{j-1} \Big(\III_{j2\ell}^A + \III_{j2\ell}^B\Big)+\III_{j2j},
\end{align*}
We first work on $\III_{j2\ell}^A$, for $\ell=0,\ldots,j-1$. Observe that
\[ \III_{j2\ell}^A = \III_{j1(\ell+1)},\quad\forall~\ell=0\ldots,j-1.\]
Hence, estimates in \eqref{eq:IIIj1} apply to $\III_{j2\ell}^A$ as well. Next, let us focus on the bounds on $\III_{j2\ell}^B$. A similar estimate as \eqref{eq:uzinf} yields
\[
 \|\lh^{\frac{ (k-j-1)\alpha}2} \pa_z^{\ell+1} u\|_{L^\infty}\lesssim E_{k-j+\ell}^{\frac12}+E_{k-j+\ell+1}^{\frac12}+\Et_{k-j+\ell+2}^{\frac12}\leq F_{k-1}^{\frac12} + E_k^{\frac12} \quad\text{if}\quad\ell\leq j-2,
\]
and for $\ell=j-1$,
\[
 \|\lh^{\frac{ (k-j-1)\alpha}2} \pa_z^j u\|_{L^\infty}\lesssim F_{k-1}^{\frac12}+E_k^{\frac12}+\Et_{k+1}^{\frac12-\frac{\alpha-1}{4\alpha}}E_k^{\frac{\alpha-1}{4\alpha}}.
\]
Moreover, we have
\[
\|\pa_x\pa_z^{j-\ell-1} u\|_{L^2}\leq \|\lh^\alpha\pa_z^{j-\ell-1} u\|_{L^2}\leq \Et_{j-\ell+1}\leq\begin{cases}
	E_k^{\frac12} & (\ell=0,\,\, j=k-1),\\
	F_{k-1}^{\frac12} & \text{otherwise}.
	\end{cases}
\]
This leads to the bound
\begin{align*}
 \III_{j2\ell}^B \lesssim &\,\Et_{k+1}^{\frac12} F_{k-1}^{\frac12} \Big( F_{k-1}^{\frac12} + E_k^{\frac12} + \Et_{k+1}^{\frac12-\frac{\alpha-1}{4\alpha}}E_k^{\frac{\alpha-1}{4\alpha}} \Big)\\
 \leq &\, \tfrac{\nu_h}{6(k+1)(j+1)} \Et_{k+1} + C\Big( F_{k-1}^2 +\big(F_{k-1}+F_{k-1}^{\frac{2\alpha}{\alpha-1}}\big)E_k\Big),
\end{align*}
except for the case when $\ell=0$ and $j=k-1$, where $\|\lh^{\frac{ (k-j-1)\alpha}2} \pa_z^{\ell+1} u\|_{L^\infty}=\|\omega\|_{L^\infty}$ is a priori bounded, and thus
\[
 \III_{(k-1)20}^B\lesssim \Et_{k+1}^{\frac12}E_k^{\frac12}\leq \tfrac{\nu_h}{6(k+1)(j+1)} \Et_{k+1} + CE_k.
\]
We are left to bound the term $\III_{j2j}$. Applying the fractional Leibniz rule and anisotropic estimates, we have
\begin{align*}
 \III_{j2j}\leq &\, \Et_{k+1}^{\frac12}\Big\|\|w\|_{L^\infty_x}\|\lh^{\frac{(k-j-1)\alpha}2}\pa_z^{j+1}u\|_{L^2_x}+\|\lh^{\frac{(k-j-1)\alpha}2}w\|_{L^2_x}\|\pa_z^{j+1}u\|_{L^\infty_x}\Big\|_{L^2_z}\\
  \leq &\, \Et_{k+1}^{\frac12}\Big(\|w\|_{L^\infty}\|\lh^{\frac{(k-j-1)\alpha}2}\pa_z^{j+1}u\|_{L^2}+\Big\|\|\lh^{\frac{(k-j-1)\alpha}2}w\|_{L^2_x}\Big\|_{L^\infty_z}\Big\|\|\pa_z^{j+1}u\|_{L^\infty_x}\Big\|_{L^2_z}\Big)\\
  := &\,\III_{j2j}^A+\III_{j2j}^B,
\end{align*}
where the first term can be easily controlled by
\[
 \III_{j2j}^A\leq \|w\|_{L^\infty}\Et_{k+1}^{\frac12}E_k^{\frac12}\leq\tfrac{\nu_h}{6(k+1)(j+1)} \Et_{k+1} + C\|w\|_{L^\infty}^2E_k.
\]
For the second term, from the Poincar\'e and Minkowski inequalities, it follows that
\[
\Big\|\|\lh^{\frac{(k-j-1)\alpha}2}w\|_{L^2_x}\Big\|_{L^\infty_z}\lesssim \Big\|\|\lh^{\frac{(k-j-1)\alpha}2}w\|_{L^\infty_z}\Big\|_{L^2_x}
\lesssim\|\lh^{\frac{(k-j-1)\alpha}2}\pa_xu\|_{L^2}
\lesssim \begin{cases}
E_k^{\frac12} & j=1,\\ F_{k-1}^{\frac12} & j\geq2.	
\end{cases}
\]
For $j=1$, we have 
\[
 \III_{121}^B\lesssim \Et_{k+1}^{\frac12}E_k^{\frac12} ( \|\pa_z^2u\|_{L^2}+\|\lh^{\frac\alpha2}\pa_z^2u\|_{L^2})\leq \Et_{k+1}^{\frac12}E_k^{\frac12}(\Ft_3^{\frac12}+F_2^{\frac12})
 \leq \tfrac{\nu_h}{6(k+1)(j+1)} \Et_{k+1} + C(\Ft_k+F_{k-1})E_k.
\]
For $2\leq j\leq k-2$, we obtain
\[
\Big\|\|\pa_z^{j+1}u\|_{L^\infty_x}\Big\|_{L^2_z}\lesssim \|\pa_z^{j+1}u\|_{L^2}+\|\lh^{\frac\alpha2}\pa_z^{j+1}u\|_{L^2}\leq F_{k-1}^{\frac12} + E_k^{\frac12},
\]
and then
\[
 \III_{j2j}^B \lesssim \Et_{k+1}^{\frac12}F_{k-1}^{\frac12}(F_{k-1}^{\frac12} + E_k^{\frac12})\leq \tfrac{\nu_h}{6(k+1)(j+1)} \Et_{k+1} + C\Big(F_{k-1}^2+F_{k-1}E_k\Big).
\]
For $j=k-1$, we  use the dissipation to estimate
\[
\Big\|\|\pa_z^{j+1}u\|_{L^\infty_x}\Big\|_{L^2_z}\lesssim \|\pa_z^ku\|_{L^2}+\|\lh^{\frac{1+\alpha}4}\pa_z^ku\|_{L^2}\leq E_k^{\frac12}+\Et_{k+1}^{\frac12-\frac{\alpha-1}{4\alpha}}E_k^{\frac{\alpha-1}{4\alpha}},
\]
and therefore
\[
 \III_{(k-1)2(k-1)}^B\lesssim \Et_{k+1}^{\frac12}F_{k-1}^{\frac12}(E_k^{\frac12}+\Et_{k+1}^{\frac12-\frac{\alpha-1}{4\alpha}}E_k^{\frac{\alpha-1}{4\alpha}})
 \leq \tfrac{\nu_h}{6(k+1)(j+1)} \Et_{k+1} + C\Big(F_{k-1}+F_{k-1}^{\frac{2\alpha}{\alpha-1}}\Big)E_k.
\]
Collecting all the estimates for the terms in $\III_{j2}$, we conclude that
\begin{equation}\label{eq:IIIj2}
 \III_{j2}\leq \frac{\nu_h}{3(k+1)} \Et_{k+1} + C\Big( F_{k-1}^2 +\big(1+F_{k-1}+F_{k-1}^{\frac{2\alpha}{\alpha-1}}+\Ft_k\big)E_k\Big).
\end{equation}

Finally, for the remaining term
\[
  \III_k = - \int_\Omega \pa_z^k u \cdot \pa_z^k (u\pa_x u + w \pa_z u) \dxdz,
\]
the estimate is analogous to that of the term $\II_3$, in addition to similar control of the intermediate terms. Applying the Leibniz rule in $z$, we  decompose
\[
 \III_k = -\sum_{\ell=0}^k \binom{k}{\ell}\int_\Omega \Big(\pa_z^ku\cdot\pa_z^\ell u \cdot\pa_z^{k-\ell}\pa_xu + \pa_z^ku\cdot\pa_z^\ell w \cdot\pa_z^{k-\ell+1}u \Big)\dxdz:=\sum_{\ell=0}^k\Big(\III_{k1\ell}+\III_{k2\ell}\Big).
\]
For $\ell=0$, we use the incompressibility \eqref{FPE-3} to deduce the cancelation
\[\III_{k10}+\III_{k20} = -\int_\Omega \pa_z^ku \cdot(u\pa_x+w\pa_z)(\pa_z^ku)\dxdz
= \int_\Omega (\pa_xu + \pa_zw)\cdot \frac{(\pa_z^ku)^2}{2} \dxdz =0.\]
For $\ell=1,\ldots,k-1$, we group the terms
\[
 \III_{k1\ell}+\III_{k2(k+1-\ell)}\lesssim \int_\Omega \big|\pa_z^ku\cdot\pa_z^\ell u \cdot\pa_z^{k-\ell}\pa_xu\big| \dxdz \leq \|\pa_z^ku\cdot\pa_z^\ell u\|_{L^2}\|\pa_z^{k-\ell}\pa_xu\|_{L^2}.
\]
Using anisotropic estimates, we obtain
\begin{align*}
 \|\pa_z^ku\cdot\pa_z^\ell u\|_{L^2} \leq &\,\Big\|\|\pa_z^ku\|_{L^\infty_x}\Big\|_{L^2_z}\Big\|\|\pa_z^\ell u\|_{L^2_x}\Big\|_{L^\infty_z}
 \lesssim \big(\|\pa_z^ku\|_{L^2}+\|\lh^{\frac{1+\alpha}4}\pa_z^ku\|_{L^2}\big)\big(\|\pa_z^\ell u\|_{L^2}+\|\pa_z^{\ell+1} u\|_{L^2}\big)\\
 \lesssim &\, \Big(E_k^{\frac12}+\Et_{k+1}^{\frac12}\Big)(E_\ell^{\frac12}+E_{\ell+1}^{\frac12}).
\end{align*}
Moreover, we have $\|\pa_z^{k-\ell}\pa_xu\|_{L^2}\lesssim \|\lh^{\frac\alpha2}\pa_z^{k-\ell}u\|_{L^2}\leq E_{k-\ell+1}^{\frac12}$. It yields
\begin{align*}
\III_{k1\ell}+\III_{k2(k+1-\ell)}\lesssim &\, (E_k^{\frac12}+\Et_{k+1}^{\frac12})(E_\ell^{\frac12}+E_{\ell+1}^{\frac12})E_{k-\ell+1}^{\frac12} \lesssim (E_k^{\frac12}+\Et_{k+1}^{\frac12})(F_{k-1}^{\frac12}+E_k^{\frac12})F_{k-1}^{\frac12}\\
 \leq &\, \tfrac{\nu_h}{2(k+1)k}\Et_{k+1}+C\Big(F_{k-1}^2+(1+F_{k-1})E_k\Big).
\end{align*}
The remaining two terms with $\ell=k$ take the following form:
\[
\III_{k1k} + \III_{k21} = (k-1)\int_\Omega \pa_xu\cdot (\pa_z^ku)^2\dxdz.
\]
Following a similar procedure as in the estimate of the term $\II_{32}$ in Proposition \ref{prop:E2}, we obtain the bound: 
\[
\III_{k1k} + \III_{k21} \lesssim \|\lh^{\frac\alpha2}\pa_z^ku\|_{L^2}\|\pa_z^ku\|_{L^2}\|\lh^{\frac{3-\alpha}2}\omega\|_{L^2}
\leq \tfrac{\nu_h}{2(k+1)k}\Et_{k+1}+C\|\lh^{\frac{3-\alpha}2}\omega\|_{L^2}^2E_k.
\]
Collecting all the estimates, we conclude that
\begin{equation}\label{eq:IIIk}
 \III_k\leq \frac{\nu_h}{2(k+1)} \Et_{k+1} + C\Big(F_{k-1}^2+\big(1+F_{k-1}+\|\lh^{\frac{3-\alpha}2}\omega\|_{L^2}^2\big)E_k\Big).
\end{equation}

Putting together \eqref{eq:III0}, \eqref{eq:IIIj1}, \eqref{eq:IIIj2} and \eqref{eq:IIIk}, we end up with the bound:
\[
\ddt E_k + \nu_h \Et_{k+1} \lesssim F_{k-1}^2+\big(1+F_{k-1}^{\frac{2\alpha}{\alpha-1}}+\Ft_k+\|w\|_{L^\infty}^2+\|\lh^{\frac{3-\alpha}2}\omega\|_{L^2}^2\big)E_k.
\]
From the regularity criterion \eqref{eq:BKM}, Proposition \ref{prop:winf}, and the induction hypotheses \eqref{eq:induction}, we infer that 
\[F_{k-1}^{\frac{2\alpha}{\alpha-1}}+\Ft_k+\|w\|_{L^\infty}^2+\|\lh^{\frac{3-\alpha}2}\omega\|_{L^2}^2\in L^1(0,T).\]
Therefore, a direct application of Gr\"onwall inequality leads to the desired bounds in \eqref{eq:Energyk} at the energy level $\frac{k\alpha}2$.
\end{proof}

\subsection{Improved global well-posedness}
In this section, we improve the global existence result from Theorem \ref{thm:GWPlarge} and establish global well-posedness for $\alpha \geq \frac{6}{5}$. Since $\alpha < \frac{3}{2}$ is \emph{energy supercritical}, addressing global well-posedness in this regime requires developing sharper estimates that go beyond the energy scale employed in the earlier analysis.

The main idea is to replace the Poincar\'e-type inequality in $z$ of the form
\[\|u\|_{L^\infty_z}\lesssim\|\omega\|_{L^2_z}\]
by a sharper interpolation inequality
\[\|u\|_{L^\infty_z}\lesssim\|u\|_{L^2_z}^{\frac12}\|\omega\|_{L^2_z}^{\frac12}.\]
The improved estimate breaks the energy scaling, allowing us to the handle energy-supercritical regime. 

We apply the idea above to obtain the following bounds on $\pa_z\omega$, which lie on the energy level $\alpha$. 
\begin{proposition}\label{prop:improvedomegaz}
	Let $\alpha\in[\frac65,2)$. Then for any strong solution of \eqref{FPE} with initial condition satisfying $\partial_z \omega_0\in L^2$, we have 
\begin{equation}\label{bound:omegaz-1}
     \pa_z \omega \in L^\infty(0, T; L^2)\cap L^2(0, T; \mathcal D(\lh^{\frac\alpha2})) .
\end{equation}
\end{proposition}
\begin{proof}
 The proof follows from a refined estimate of \eqref{eq:E2omegaz}. The improvement is on the bound of the term $\II_{32}$. We start with
\[ \II_{32} = - \int_\Omega \lh^{1-\frac\alpha2}\big((\pa_z\omega)^2\big)\cdot\lh^{\frac\alpha2}Hu\dxdz
	\leq \int_\T\|\lh^{1-\frac\alpha2}\big((\pa_z\omega)^2\big)\|_{L^1_z}\|\lh^{\frac\alpha2}Hu\|_{L^\infty_z}\,\dd x.\]
 Instead of a Poincar\'e type estimate $\|\lh^{\frac\alpha2}Hu\|_{L^\infty_z}\lesssim\|\lh^{\frac\alpha2}H\omega\|_{L^2}$, we use interpolation to get a sharper bound:
 \begin{equation}\label{eq:omegazimprove}
 	\|\lh^{\frac\alpha2}Hu\|_{L^\infty_z}\lesssim\|\lh^{\frac\alpha2}Hu\|_{L^2_z}^{\frac12}\|\lh^{\frac\alpha2}H\omega\|_{L^2_z}^{\frac12}.
 \end{equation}
 Now, we continue the estimate:
 \begin{align*}
	\II_{32} \lesssim &~ \int_\T\|\lh^{1-\frac\alpha2}\big((\pa_z\omega)^2\big)\|_{L^1_z}\|\lh^{\frac\alpha2}Hu\|_{L^2_z}^{\frac12}\|\lh^{\frac\alpha2}H\omega\|_{L^2_z}^{\frac12}\,\dd x\\
	\leq &~ \Big\|\|\lh^{1-\frac\alpha2}\big((\pa_z\omega)^2\big)\|_{L^1_z}\Big\|_{L^2_x}~\Big\|\|\lh^{\frac\alpha2}Hu\|_{L^2_z}^{\frac12}\Big\|_{L^4_x}~\Big\|\|\lh^{\frac\alpha2}H\omega\|_{L^2_z}^{\frac12}\Big\|_{L^4_x}\\
	\lesssim &~ \Big\|\|\lh^{1-\frac\alpha2}\big((\pa_z\omega)^2\big)\|_{L^2_x}\Big\|_{L^1_z}\|\lh^{\frac\alpha2}u\|_{L^2}^{\frac12}\|\lh^{\frac\alpha2}\omega\|_{L^2}^{\frac12}.
\end{align*}
For the first term, we apply the fractional Leibniz rule 
and interpolation to get:
\[
 \|\lh^{1-\frac\alpha2}\big((\pa_z\omega)^2\big)\|_{L^2_x} \leq \|\lh^{1-\frac\alpha2}(\pa_z\omega)\|_{L^2_x}\|\pa_z\omega\|_{L^\infty_x} \lesssim \|\lh^{\frac\alpha2}\pa_z\omega\|_{L^2_x}^{\frac{3-\alpha}\alpha}\|\pa_z\omega\|_{L^2_x}^{\frac{3\alpha-3}\alpha},
\]
and then
\[
\Big\|\|\lh^{1-\frac\alpha2}\big((\pa_z\omega)^2\big)\|_{L^2_x}\Big\|_{L^1_z}
\lesssim \Big\|\|\lh^{\frac\alpha2}\pa_z\omega\|_{L^2_x}^{\frac{3-\alpha}\alpha}\Big\|_{L^{\frac{2\alpha}{3-\alpha}}_z}~\Big\|\|\pa_z\omega\|_{L^2_x}^{\frac{3\alpha-3}\alpha}\Big\|_{L^{\frac{2\alpha}{3\alpha-3}}_z}=\|\lh^{\frac\alpha2}\pa_z\omega\|_{L^2}^{\frac{3-\alpha}\alpha}\|\pa_z\omega\|_{L^2}^{\frac{3\alpha-3}\alpha}.
\]
This leads to the bound
\begin{align*}
	\II_{32} \lesssim &~ \|\lh^{\frac\alpha2}\pa_z\omega\|_{L^2}^{\frac{3-\alpha}\alpha}\|\pa_z\omega\|_{L^2}^{\frac{3\alpha-3}\alpha}\|\lh^{\frac\alpha2}u\|_{L^2}^{\frac12}\|\lh^{\frac\alpha2}\omega\|_{L^2}^{\frac12}\\
	\leq &~ \frac14\nu_h\|\lh^{\frac\alpha2}\pa_z\omega\|_{L^2}^2+C\|\lh^{\frac\alpha2} u\|_{L^2}^{\frac{\alpha}{3\alpha-3}} \|\lh^{\frac\alpha2} \omega\|_{L^2}^{\frac{\alpha}{3\alpha-3}}\|\pa_z\omega\|_{L^2}^2.
\end{align*}
Together with \eqref{eq:II31}, we arrive at
\begin{equation}\label{eq:omegazp}
  \ddt \|\pa_z\omega\|_{L^2}^2 + \nu_h \|\lh^{\frac\alpha2} \pa_z\omega\|_{L^2}^2  \lesssim \|\omega_0\|_{L^\infty}^2 \|\lh^{\frac\alpha2}\omega\|_{L^2}^2+\|\lh^{\frac\alpha2} u\|_{L^2}^{\frac{\alpha}{3\alpha-3}} \|\lh^{\frac\alpha2} \omega\|_{L^2}^{\frac{\alpha}{3\alpha-3}}\|\pa_z\omega\|_{L^2}^2.	
\end{equation}
Recall the a priori bounds of the energy level $\frac\alpha2$ in \eqref{eq:Energy1}: 
\[\|\lh^{\frac\alpha2} u\|_{L^2}\in L^\infty(0,T),\quad\text{and}\quad
\|\lh^{\frac\alpha2} \omega\|_{L^2} \in L^2(0,T).\]
When $\alpha\geq\frac65$, we have $\frac{\alpha}{3\alpha-3}\leq2$. Consequently,
\[\mu(t):=\exp\int_0^tC\|\lh^{\frac\alpha2} u(\tau)\|_{L^2}^{\frac{\alpha}{3\alpha-3}} \|\lh^{\frac\alpha2} \omega(\tau)\|_{L^2}^{\frac{\alpha}{3\alpha-3}}\,\dd\tau<\infty,\quad\forall~t\in[0,T].\]
We apply Gr\"onwall inequality to \eqref{eq:omegazp} and obtain
\[
\|\pa_z\omega(t)\|_{L^2}^2\leq\mu(t)\Big(\|\pa_z\omega_0\|_{L^2}^2+C\int_0^t\|\omega_0\|_{L^\infty}^2 \|\lh^{\frac\alpha2}\omega(\tau)\|_{L^2}^2\,\dd\tau\Big)<\infty,\quad\forall~t\in[0,T].
\]
We conclude with the bound \eqref{bound:omegaz-1}.
\end{proof}

With the bounds \eqref{bound:omegaz-1} on $\pa_z\omega$, we continue to derive the crucial estimates \eqref{eq:uniqueness} on $\omega$.
\begin{proposition}\label{prop:improvedomega}
	Let $\alpha\in[\frac65,\frac32)$. Then for any strong solution of \eqref{FPE} with initial condition satisfying $\partial_z \omega_0\in L^2$ and $\omega_0\in \mathcal D(\lh^{\frac32-\alpha})$, the condition \eqref{eq:uniqueness} holds, that is,
  \begin{equation}\label{eq:omegap}
	\omega\in L^\infty(0, T; \D(\lh^{\frac32-\alpha}))\cap L^2(0, T; \D(\lh^{\frac{3-\alpha}2})).  	
  \end{equation}
\end{proposition}
\begin{proof}
 We perform energy estimates on 	$\|\lh^{\frac32-\alpha}\omega\|_{L^2}^2$. Taking the $L^2$ inner product of \eqref{FPE-1} with $\lh^{3-2\alpha}\omega$, we get
\begin{align*}
    \frac12\ddt \|\lh^{\frac32-\alpha}\omega\|_{L^2}^2 + \nu_h \|\lh^{\frac{3-\alpha}{2}}\omega\|_{L^2}^2 = -\int_{\Omega} (u\pa_x \omega + w\pa_z \omega) \lh^{3-2\alpha}\omega \dxdz:= K_1+K_2.
\end{align*}
For $K_1$, integration by parts in $x$ yields
\[
K_1 = \int_{\Omega} \pa_xu \cdot \omega \lh^{3-2\alpha}\omega \dxdz + \int_{\Omega} u\cdot \omega \lh^{3-2\alpha}\pa_x\omega \dxdz := K_{11}+K_{12}.
\]
By the H\"older, Young, and Sobolev inequalities, and the a priori bounds \eqref{eq:omegaMP} and \eqref{eq:uLinfty}, we have 
\[
    K_{11} \leq \|\pa_xu\|_{L^2} \|\omega\|_{L^\infty} \|\lh^{3-2\alpha}\omega\|_{L^2} 
    \leq \frac14\nu_h \|\lh^{\frac{3-\alpha}{2}}\omega\|_{L^2}^2 + C \|\lh^{\alpha} u\|_{L^2}^2,
\]
where we note that $3-2\alpha<\frac{3-\alpha}2$ as $\alpha>1$, and
\begin{align*}
    K_{12} & = \int_{\Omega} \lh^{\frac{1-\alpha}2}\pa_x\omega\cdot \lh^{\frac{5-3\alpha}2} (u\omega) \dxdz
    \\
    & \lesssim  \|\lh^{\frac{3-\alpha}2}\omega\|_{L^2} \Big(\|\lh^{\frac{5-3\alpha}2} u\|_{L^2} \|\omega\|_{L^\infty} + \|u\|_{L^\infty} \|\lh^{\frac{5-3\alpha}2}\omega\|_{L^2}\Big)
    \\
    & \lesssim  \|\lh^{\frac{3-\alpha}2}\omega\|_{L^2} \|\lh^{\frac{5-3\alpha}2} u\|_{L^2}  + \|\lh^{\frac{3-\alpha}2}\omega\|_{L^2}^{2-\frac{2\alpha-2}{3-\alpha}} \|\omega\|_{L^2}^{\frac{2\alpha-2}{3-\alpha}}
    \\
    &\leq \frac14\nu_h \|\lh^{\beta+\frac\alpha2}\omega\|_{L^2}^2 + C\big(\|\lh^{\alpha} u\|_{L^2}^2 + \|\omega\|_{L^2}^2\big),
\end{align*}
where we note that $\frac{5-3\alpha}2<\alpha$ as $\alpha>1$.

For $K_2$, integration by parts in $z$ yields
\begin{align*}
    K_2 & = -\int_{\Omega} \pa_xu\cdot \omega \lh^{3-2\alpha}\omega \dxdz + \int_{\Omega} w \cdot \omega \lh^{3-2\alpha}\pa_z\omega \dxdz
    \\
    & \leq \|\omega\|_{L^\infty}\Big(\|\pa_x u\|_{L^2} \|\lh^{3-2\alpha}\omega\|_{L^2} + \|w\|_{L^2}\|\lh^{3-2\alpha}\pa_z\omega\|_{L^2}\Big)
    \\
    &\lesssim \|\pa_x u\|_{L^2}\Big(\|\lh^{3-2\alpha}\omega\|_{L^2} + \|\lh^{3-2\alpha}\pa_z\omega\|_{L^2}\Big)
    \lesssim \|\lh^{\alpha} u\|_{L^2}\Big(\|\lh^{\frac\alpha2}\omega\|_{L^2} + \|\lh^{\frac\alpha2}\pa_z\omega\|_{L^2}\Big)\\
    &\lesssim \|\lh^{\alpha} u\|_{L^2}^2 + \|\lh^{\frac\alpha2}\omega\|_{L^2}^2 + \|\lh^{\frac\alpha2}\pa_z\omega\|_{L^2}^2,
\end{align*}
where we have used the Poincar\'e inequality $\|w\|_{L^2}\lesssim\|\pa_xu\|_{L^2}$ in the second inequality, and for the penultimate inequality, we note that $3-2\alpha\leq\frac\alpha2$ as $\alpha\geq\frac65$. 

Combining the estimates above, we have
\[
    \ddt \|\lh^{\frac32-\alpha}\omega\|_{L^2}^2 + \nu_h \|\lh^{\frac{3-\alpha}{2}}\omega\|_{L^2}^2 \lesssim 1+\|\lh^{\alpha} u\|_{L^2}^2 + \|\lh^{\frac\alpha2}\omega\|_{L^2}^2 + \|\lh^{\frac\alpha2}\pa_z\omega\|_{L^2}^2.
\]
Thanks to \eqref{bound:u-4}, \eqref{bound:omega-1}, and \eqref{bound:omegaz-1}, we can apply  Gr\"onwall's inequality to conclude that
\[
    \|\lh^{\frac32-\alpha}\omega(T)\|_{L^2}^2 + \int_0^T \nu_h \|\lh^{\frac{3-\alpha}{2}}\omega(t)\|_{L^2}^2 \dd t \leq C,
\]
for all $T>0$. Therefore, \eqref{eq:omegap} holds. 
\end{proof}

A direct application of Proposition \ref{prop:improvedomega} to Theorem \ref{thm:uniqueness} yields an improved uniqueness result for strong solutions in the range $\alpha\in[\frac65,\frac32)$.
\begin{theorem}[Improved uniqueness of strong solutions]\label{thm:uniqueness3}
    Let $\alpha\in[\frac65,\frac32)$. Suppose $u_0\in \mathcal D(\lh^{\frac\alpha2})\cap \H$, 
    $\omega_0 \in L^\infty\cap \mathcal D(\lh^{\frac32-\alpha})$, and $\pa_z \omega_0\in L^2$. Then the global strong solution obtained in Theorem~\ref{thm:sub-1} is unique and continuously depends on the initial data.
\end{theorem}

Moreover, since \eqref{eq:omegap} implies the regularity critetrion \eqref{eq:BKM}, we also obtain an improved version of Theorem \ref{thm:GWPlarge} for classical solutions in the range $\alpha\in[\frac65,\frac32)$.

\begin{theorem}[Improved global existence of classical solution: large initial data]\label{thm:GWPlarge2}
	Let $\alpha\in[\frac65,\frac32)$. Suppose $u_0$ satisfies \eqref{eq:smoothinit}. Then there exists a global classical solution  
    to \eqref{FPE} with $u(0)=u_0$.
\end{theorem}

By applying Theorem~\ref{thm:smooth}, the improved global existence result in Theorem \ref{thm:GWPlarge2} can be extended to smooth solutions.

\subsection{Summary of results}\label{sec:subsummary}
We now summarize the results obtained in this section for the subcritical FPE system:
\begin{itemize}
\item Local well-posedness when $\alpha>1$.
\item Global well-posedness for small initial data when $\alpha>1$.	
\item Global well-posedness for general initial data when $\alpha\geq\frac65$.
\end{itemize}

More precisely, the well-posedness results are established for strong, classical, and smooth solutions. The detailed statements for each type of solution, together with references to the corresponding theorems, are given below.

\begin{itemize}
\item Strong solution (Definition~\ref{def:strong})
\begin{itemize}
\item Local and global existence: $\alpha>1$ (Theorem \ref{thm:sub-1}).
\item Uniqueness: $\alpha\geq\frac32$ (Theorem \ref{thm:uniqueness}), $\alpha\geq\frac65$ (Theorem \ref{thm:uniqueness3}).
\end{itemize}
\item Classical solution (Definition~\ref{def:classical})
\begin{itemize}
\item Local existence: $\alpha>1$ (Theorem \ref{thm:LWP}).
\item Global existence for small initial data: $\alpha>1$ (Theorem \ref{thm:GWPsmall}).
\item Global existence for general initial data: $\alpha\geq\frac32$ (Theorem \ref{thm:GWPlarge}), $\alpha\geq\frac65$ (Theorem \ref{thm:GWPlarge2}).
\item Uniqueness: $\alpha>1$ (Theorem \ref{thm:uniqueness2}).
\end{itemize}
\item Smooth solution (Definition~\ref{def:smooth})
\begin{itemize}
\item Same existence and uniqueness results as the classical solution (Theorem \ref{thm:smooth}).
\end{itemize}
\end{itemize}\medskip

The local well-posedness result stands in sharp contrast to the ill-posedness results for the supercritical FPE system (Theorem \ref{thm:nonlinear-ill}). The transition occurs at the critical exponent $\alpha = 1$, which will be the focus of the next section.

The global well-posedness for general initial data when $\alpha \in (1, \frac65)$ remains an open problem.

\section{The critical case}\label{sec:critical}
In this section, we study the critical case $\alpha = 1$ and demonstrate a sharp transition between ill-posedness and well-posedness. When the initial data $u_0$ is large relative to the horizontal dissipation coefficient $\nu_h$, the solution exhibits a Kelvin–Helmholtz-type instability, similar to the supercritical case. In contrast, if $u_0$ is sufficiently small compared to $\nu_h$, the solution remains globally well-posed, like the subcritical case.

\subsection{Ill-posedness for large initial data}
We start with the linear ill-posedness result. From the same argument as in Proposition~\ref{prop:supercritical-ill}, we know that the linearized system \eqref{sys:psi} has solutions of the form
	\[
	\psi_n(x,z,t)=\chi(z) e^{2\pi inx} e^{n \beta t},
    \]
    where $\chi(z)$ is given \eqref{eqn:chi}, and $\beta$ satisfies
    \begin{equation}\label{eq:beta}
    	\beta = 2\pi\gamma_0 - 2\pi\nu_h.
    \end{equation}
Therefore, we obtain linear instability when $\beta>0, or$ $\gamma_0>\nu_h$. 

Observe from the relation \eqref{eq:Ucond} that $\gamma_0$ is linearly related to the size of the shear flow $U$. Indeed, if $(U,\gamma_0)$ satisfies \eqref{eq:Ucond}, so does $(\lambda U,\lambda\gamma_0)$. Hence, by scaling we have
\begin{equation}\label{eq:gammaU}
\gamma_0 = C\|U\|,	
\end{equation}
where the constant $C$ depends on the choice of the norm. This leads to the following linear ill-posedness result.
\begin{theorem}\label{thm:critical-ill}
	Consider the linearized system \eqref{pert-system} with the shear flow $U$ satisfying \eqref{eq:Ucond} and 
	\begin{equation}\label{eq:ill:crit}
	 \|U\|\gg\nu_h,		
	\end{equation}
	or more precisely $C\|U\|\geq\nu_h$ with the constant $C$ in \eqref{eq:gammaU}.
	Let $s\geq0$. There exists a solution $\psi$, such that $\psi_0\in H^s$, but $\psi(t)\not\in H^s$ for any $t>0$.
\end{theorem}
The proof of Theorem \ref{thm:critical-ill} follows the same strategy as that of Theorem \ref{thm:illposedness}. Observe that the initial velocity of the linearized system, $u_0 = U + \epsilon \tilde u$ is a small perturbation of the background shear flow $U$. Hence, with an appropriate choice of norm, we have $\|u_0\|\approx\|U\|$, and the condition \eqref{eq:ill:crit} is equivalent to $\|u_0\|\gg \nu_h$.

Next, we obtain the nonlinear ill-posedness result. 
\begin{theorem}\label{thm:critical-nonlinear-ill}
    Let $\alpha=1$. Suppose $U$ is a shear flow satisfying \eqref{eq:Ucond} and \eqref{eq:ill:crit}. Then Theorem \ref{thm:nonlinear-ill} holds. More precisely, let $s\geq0$. Denote $\{\omega_n\}_{n\in\Z_+}$ the solution of the equation \eqref{eq:omeganl} with initial condition $\omega_{n0}(x,z)=\chi''(z)e^{2\pi inx}$. Then, we have
\begin{equation}\label{eq:illnl2}
\lim_{n \to \infty} \frac{ \| \omega_n \|_{L^2([0,t_n] \times \Omega)}}{\| \omega_{n0} \|_{H^s(\Omega)}}=+\infty,
\end{equation}
with $t_n=\mathcal{O}(n^{-1}\log n)$ which goes to 0 as $n\to\infty$.
\end{theorem}

 The proof of Theorem \ref{thm:critical-nonlinear-ill} follows the framework in \cite{han2016ill}. The main difference compared with the supercritical case is that the term $R_1\omega$ in \eqref{eq:omegasy} is no longer small when $\eps$ is small. Therefore, it needs to be merged into the leading order term $L$. We have
 \begin{equation}\label{eq:Lcrit}
L\omega:= - U \pa_y \omega  + U'' \pa_y \psi - \nu_h\Lambda_y\omega. 	
 \end{equation}
 Since the quadratic term $Q$ is the same as the inviscid PE system, we only have to verify the hypotheses {\bf (H.1)}--{\bf (H.2)} on the linear operator $L$ to apply the abstract framework in \cite{han2016ill}. The rest of the proof is identical to \cite{han2016ill} and we omit the details.
 
 The hypothesis {\bf (H.1)} describes the instability of the linearized operator $L$. Our linear ill-posedness result verifies that there exists an eigenfunction $\omega_{n0} = \chi''(z)e^{2\pi inx}$ for $L$ with a positive eigenvalue $\beta>0$. Hence, {\bf (H.1)} holds.
 
 The proof of the hypothesis {\bf (H.2)} modifies \cite[Proposition 3.2]{han2016ill}. More precisely, we need the following bound on the semigroup $e^{L\s}$.
 \begin{proposition}\label{prop:H.2}
 Let $\delta, \delta'>0$. Then for any $\gamma>\beta$ and $\omega_0\in X_{\delta,\delta'}$,
 \begin{equation}\label{eq:H2}
\|e^{L\s}\omega_0\|_{\delta-\gamma\s, \delta'}\lesssim_\gamma\|\omega_0\|_{\delta,\delta'},	
 \end{equation}
 where $\delta'$ is small, and $\delta-\gamma\s>0$. 
 \end{proposition}
 
 The inequality \eqref{eq:H2} provides sharp control of the semigroup $e^{L\s}$, up to a loss of analytic regularity in the $y$-variable. In the presence of dissipation, this estimate is stronger than the one in the inviscid case: for the inviscid PE system, \eqref{eq:H2} holds only for $\gamma > \gamma_0$, whereas dissipation allows us to take smaller values of $\gamma$.

We include a sketch of the proof of Proposition \ref{prop:H.2}, focusing primarily on how the dissipation is incorporated.

For the rest of this subsection, we consider the linear equation for $\omega=\omega(y,z,\s)$: 
\begin{equation}\label{eq:omegaLcrit}
 \pa_\s\omega - L\omega = 0,\quad \omega\vert_{\s=0} = \omega_0,
\end{equation}
where $L$ is defined in \eqref{eq:Lcrit}. The solution is denoted by $\omega(\s) = e^{L\s}\omega_0$.
Applying Fourier transform in $y$, we obtain the dynamics of the $n$-th Fourier mode:
\begin{equation}\label{eq:omegahatn}
\pa_\s\hat\omega_n=\widehat{L}_n\hat\omega_n:=-2\pi inU\hat\omega_n+2\pi in U''\hat\psi_n - 2\pi |n|\nu_h \hat\omega_n, \quad \hat\omega_n\vert_{\s=0}=(\widehat{\omega_0})_n=:h_n,
\end{equation}
where
 \begin{equation}\label{eq:psihatn}
  \pa_z^2\hat\psi_n = \hat\omega_n,\quad \hat\psi_n\vert_{z=0,1} = 0.
 \end{equation}
 
We aim to control the growth of $\hat\omega_n$, starting by the following $L^2_z$ estimate.
\begin{lemma}\label{lem:H2L2}
For any $\gamma>\beta$, the system \eqref{eq:omegahatn}-\eqref{eq:psihatn} has a solution, satisfying
\begin{equation}\label{eq:H2L2}
\|\hat\omega_n(\s)\|_{L^2_z}\lesssim_\gamma e^{|n|\gamma\s} \|h_n\|_{L^2_z}.	
\end{equation}
\end{lemma}
\begin{proof}
 We focus on the case when $n=1$. Following \cite[Lemma 3.4]{han2016ill}, we apply inverse Laplace transform on $\s$ and obtain
 \begin{equation}\label{eq:Laplace}
   \hat\omega_1(\s) = \frac{1}{2\pi i} \int_{\gamma + i\R} e^{\lambda\s} (\lambda-\widehat{L}_1)^{-1}h_1\,\dd \lambda,
 \end{equation}
Set $\omega_\lambda:=(\lambda-\widehat{L}_1)^{-1}h_1$. The resolvent equation reads:
\[
  \lambda\omega_\lambda + 2\pi iU\omega_\lambda - 2\pi iU''\psi_\lambda + 2\pi \nu_h\omega_\lambda = h_1,\quad \pa_z^2\psi_\lambda = \omega_\lambda, \quad \psi_\lambda|_{z=0,1}=0,
\]
 which can be equivalently expressed as the nonhomogeneous hydrostatic Orr-Sommerfeld equation on $\psi_\lambda$:
 \begin{equation}\label{eq:resolvent}
\big(U-c\big) \pa_z^2\psi_\lambda -  U'' \psi_\lambda = \frac{h_1}{2\pi i},\quad\text{where}\quad c =  i\left(\frac{\lambda}{2\pi}+\nu_h\right).
 \end{equation}
The equation has a unique solution when $|\text{Im}(c)|>\gamma_0$, or equivalently
\[\frac{\text{Re}(\lambda)}{2\pi}+\nu_h>\gamma_0.\]
This condition can be verified by $\text{Re}(\lambda)=\gamma$, the assumption $\gamma>\beta$, and the definition of $\beta$ in \eqref{eq:beta}. 
Moreover, the solution satisfies the following elliptic bound:
\[
 \|\psi_\lambda\|_{H^2_z}\lesssim_\gamma \frac{1}{1+|\zeta|}\|h_1\|_{L^2_z},\quad \forall~ \lambda = \gamma+i\zeta.
\]
Here, we denote $\zeta:=\text{Im}(\lambda)$ for simplicity.

Applying the resolvent equation \eqref{eq:resolvent} to \eqref{eq:Laplace}, we deduce the identity: 
\begin{align*}
 \hat\omega_1(\s) & = \frac1{2\pi i}\int_{\gamma + i\R} e^{\lambda\s} \omega_\lambda \,\dd \lambda
 =\frac1{2\pi i}\int_{\gamma + i\R} e^{\lambda\s} \left(\frac{U''}{U-c}\psi_\lambda + \frac{h_1}{2\pi i(U-c)}\right)\,\dd \lambda\\
 & = \frac{e^{\gamma s}}{2\pi}\int_\R e^{i\zeta\s} \frac{U''}{U-c}\psi_\lambda\,\dd\zeta + e^{-2\pi (\nu_h+iU)\s}h_1.
\end{align*}
Now we apply the $L^2_z$ norm and obtain
\begin{align*}
 \|\hat\omega_1(\s)\|_{L^2_z} & \leq  \frac{e^{\gamma s}}{2\pi}\int_\R \left\|\frac{U''}{U-c}\right\|_{L^\infty_z}\|\psi_\lambda\|_{L^2_z} \,\dd\zeta + e^{-2\pi\nu_h s}\|h_1\|_{L^2_z}\\
 & \lesssim_\gamma \Big(e^{\gamma s}\int_\R  \left\|\frac{U''}{U-c}\right\|_{L^\infty_z}\frac{1}{1+|\zeta|} \,\dd\zeta + 1\Big)\|h_1\|_{L_z^2}.
\end{align*}
It remains to show that the integrand above is finite. Note that
\[
|U(z)-c|=\frac{1}{2\pi}\Big|\big(2\pi U(z)+\zeta\big)-i(\gamma+2\pi\nu_h)\Big|\geq \frac{1}{2\pi}\max\big\{\beta+2\pi\nu_h, |\zeta|-2\pi\|U\|_{L^\infty_z}\big\}\gtrsim 1+|\zeta|,
\]
uniformly in $z$. This ensures that
\[
 \int_\R  \left\|\frac{U''}{U-c}\right\|_{L^\infty_z}\frac{1}{1+|\zeta|} \,\dd\zeta \lesssim \|U''\|_{L^\infty_z}\int_\R \frac{1}{(1+|\zeta|)^2}\,\dd\zeta
 <\infty.
\]
The proof of \eqref{eq:H2L2} is finished for $n=1$.

For $n\in\Z_+$, observe that the corresponding linear operator $\widehat{L}_n$ in \eqref{eq:omegahatn} has the form $\widehat{L}_n = n\widehat{L}_1$.
Therefore, if we rescale time $\s \mapsto n\s$, it reduces to the case $n=1$, and \eqref{eq:H2L2} follows immediately.

The cases $n=-1$ and general $n\in\Z_-$ can be treated analogously.
\end{proof}

Next, we estimate higher-order derivatives of $\omega$ in the $z$-direction by working with the analytic function space $X_{\delta'}$ equipped with the norm
\[
 \|f\|_{\delta'} :=\sum_{k\geq0}\|\pa_z^kf\|_{L^2_z}\frac{|\delta'|^k}{k!}.
\]
The resulting growth rate matches that of the $L^2_z$ estimate in \eqref{eq:H2L2}.
\begin{lemma}\label{lem:H2delta}
For any $\gamma>\beta$, and $\delta'>0$ sufficiently small, we have
\begin{equation}\label{eq:H2delta}
  \|\hat\omega_n\|_{\delta'}\lesssim_\gamma e^{|n|\gamma\s}\|h_n\|_{\delta'}.
\end{equation}
\end{lemma}
\begin{proof}
We perform $L^2_z$ energy estimate on $\pa_z^k\hat\omega_n$.
\[
\frac12\frac{\dd}{\dd s}\|\pa_z^k\hat\omega_n\|_{L^2_z}^2=\int_0^1\pa_z^k\hat\omega_n \cdot \pa_z^k\Big(-2\pi inU\hat\omega_n+2\pi in U''\hat\psi_n\Big)\,\dd z - 2\pi|n|\nu_h\|\pa_z^k\hat\omega_n\|_{L^2_z}^2.
\]
Since the contribution from the dissipation has a favorable sign, we may simply drop this term, and remaining analysis then follows similarly to the inviscid PE system. In particular, from \cite[Lemma 3.3]{han2016ill} we have
\newcommand{\seminorm}[1]{\left\lvert\hspace{-1 pt}\left\lvert\hspace{-1 pt}\left\lvert {#1}\right\lvert\hspace{-1 pt}\right\lvert\hspace{-1 pt}\right\lvert}
\begin{align*}
 \frac{\dd}{\dd\s}\|\hat\omega_n\|_{\delta'} \leq 2\pi|n|\Big(\delta'\seminorm{U}_{\delta'} \|\hat\omega_n\|_{\delta'} + \|U''\hat\psi_n\|_{\delta'}-\nu_h\|\hat\omega_n\|_{\delta'} \Big),
\end{align*}
where the norm $\seminorm{\cdot}_{\delta'}$ is defined as
\[\seminorm{U}_{\delta'} :=\sum_{k\geq0}\|\pa_z^{k+1}U\|_{L^\infty_z}\frac{|\delta'|^k}{k!}<\infty,\]
and the second term can be further estimated by
\begin{align*}
 \|U''\hat\psi_n\|_{\delta'}\lesssim \|U''\|_{\delta'} \|\hat\psi_n\|_{\delta'}	\lesssim \|U''\|_{\delta'}\big(\|\hat\psi_n\|_{H^1_z} + |\delta'|^2 \|\hat\omega_n\|_{\delta'}\big).
\end{align*}
By the Poincar\'e inequality and Lemma \ref{lem:H2L2}, we know that for any $\gamma>\beta$, 
\[\|\hat\psi_n(\s)\|_{H^1_z}\lesssim \|\hat\omega_n(\s)\|_{L^2_z}\lesssim_\gamma e^{|n|\gamma\s}\|h_n\|_{L^2_z}.\]
Then, it yields
\begin{align*}
 \frac{\dd}{\dd s}\|\hat\omega_n\|_{\delta'} \leq 2\pi|n|\Big(\delta'\seminorm{U}_{\delta'}  + C|\delta'|^2\|U''\|_{\delta'}-\nu_h\Big)\|\hat\omega_n\|_{\delta'} + C|n|e^{|n|\gamma\s}\|U''\|_{\delta'}\|h_n\|_{L^2_z}.
\end{align*}
Choose a sufficiently small $\delta'$ such that the first term on the right-hand side is negative. We then integrate the inequality and conclude with 
\begin{equation}\label{eq:H2deltaimprove}
 \|\hat\omega_n(\s)\|_{\delta'}\lesssim_\gamma e^{|n|\gamma\s}\|h_n\|_{L^2_z} + \|h_n\|_{\delta'},
\end{equation}
which directly implies \eqref{eq:H2delta}.
\end{proof}
Note that the inequality we obtained in \eqref{eq:H2deltaimprove} is stronger than \eqref{eq:H2delta}, due to the presence of the dissipation. However, to finish the proof of Proposition \ref{prop:H.2}, we will only make use of the weaker bound \eqref{eq:H2delta}, together with the definition of the norm $\|\cdot\|_{\delta,\delta'}$ in \eqref{vort-norm0}. Direct calculation yields
\begin{align*}
 \|e^{L\s}\omega_0\|_{\delta-\gamma\s, \delta'} = \|\omega(\s)\|_{\delta-\gamma\s,\delta'} = \sum_{n\in\Z}e^{|n|(\delta-\gamma\s)}\|\hat\omega_n(\s)\|_{\delta'}
 \lesssim \sum_{n\in\Z}e^{|n|(\delta-\gamma\s)} e^{|n|\gamma\s}\|h_n\|_{\delta'} = \|\omega_0\|_{\delta,\delta'}.
\end{align*}

Similar to Section~\ref{subsec:sup-non-ill} we consider $\W=(\omega,\pa_y\omega, \pa_z\omega)^\top$, which solves 
\begin{equation}
  \partial_s \W - \mathrm{L}\W = \mathrm{Q}(\W,\W),
\end{equation}
where
\begin{equation}
 \mathrm{L}: = \begin{pmatrix} L &0 &0 \\ 0 &  L &0\\ 0& -U' + U'' \pa_z \psi(\cdot) + U''' \psi(\cdot) & - U \partial_y-\nu_h\Lambda_y \end{pmatrix}.
\end{equation}
Note that the last entry changes to $- U \partial_y-\nu_h\Lambda_y$ due to the dissipation. One can proceed similarly to the proof of \cite[Proposition 3.2]{han2016ill} to verify \textbf{(H.2)}, and we omit the details.

\subsection{Well-posedness for small initial data}
We repeat the a priori energy estimates in Section \ref{sec:sub}. The focus will be on the adaptation of the estimates to the critical regime $\alpha=1$.

The maximum principle of $\omega$ and the energy bound \eqref{eq:Energy0} on $E_0$ follow the same as in the subcritical case. 

For $E_1$, the first component $\|\lh^{\frac12} u\|_{L^2}^2$ can be controlled as follows.
\begin{align*}
     \frac12 \ddt \|\lh^{\frac12} u\|_{L^2}^2 + \nu_h \|\lh u\|_{L^2}^2 & = -\int_\Omega u\pa_x u \lh u\dxdz -\int_\Omega w\omega \lh u\dxdz\\
    &\leq \|u\|_{L^\infty}\|\lh u\|_{L^2}^2+\|\omega\|_{L^2}\|w\|_{L^2}\|\lh u\|_{L^2}
    \leq C\|\omega_0\|_{L^\infty}\|\lh u\|_{L^2}^2.
\end{align*}
Hence, if $\|\omega_0\|_{L^\infty}\leq \frac{\nu_h}{2C}$, then
\[\ddt \|\lh^{\frac12} u\|_{L^2}^2 + \nu_h \|\lh u\|_{L^2}^2\leq0,\]
which leads to the bound \eqref{bound:u-4}.
The second component $\|\omega\|_{L^2}^2$ can be controlled by the same procedure as in the subcritical case. Combined, they yield the a priori bound \eqref{eq:Energy1} on $E_1$. Consequently, we obtain the following existence theorem, whose proof follows directly from Theorem \ref{thm:sub-1}.

\begin{theorem}[Global existence of strong solutions]\label{thm:E1existence:crit}
    Let $u_0\in \mathcal D(\lh^{\frac12})\cap \H$ and 
    \[\|\omega_0\|_{L^\infty}\lesssim \nu_h.\]
    For any time $T>0$, there exists at least one strong solution to \eqref{FPE} with $u(0)=u_0$ on $[0,T]$. 
\end{theorem}

In light of Theorem \ref{thm:uniqueness}, strong solutions may fail to be unique. However, the condition \eqref{eq:uniqueness} indicates that uniqueness may be ensured if the energy level $\frac{3-\alpha}2=1$ is bounded. For $\alpha=1$, this corresponds to the classical solution setting. Therefore, we will establish a uniqueness result later in Theorem \ref{thm:uniqueness:crit}, after proving the existence of classical solutions.

Next, we work on $E_2$. For the first component $\|\pa_xu\|_{L^2}^2$, we adapt the estimates in Proposition \ref{prop:winf}. The only term that requires a different treatment is $J_{12}$. Recall the estimate \eqref{eq:J12-1}:
\[J_{12}\lesssim \|\lh^{\frac32}u\|_{L^2}\Big\|\|\lh^{\frac12}(\omega\pa_xu)\|_{L^2_x}\Big\|_{L^1_z}.\]
Using Lemma \ref{lem:improvedLeib} and Sobolev embedding, we obtain
\[
 \|\lh^{\frac12}(\omega\pa_xu)\|_{L^2_x}
 \lesssim \|\omega_0\|_{L^\infty_x}\|\lh^{\frac32}u\|_{L^2_x} + \|\lh\omega\|_{L^2_x}\|\pa_xu\|_{L^2_x},
\]
and therefore
\[
 J_{12}\lesssim \Et_3^{\frac12}\big( \|\omega_0\|_{L^\infty}\Et_3^{\frac12} + \Et_3^{\frac12} E_2^{\frac12} \big) = \big( \|\omega_0\|_{L^\infty} +  E_2^{\frac12} \big) \Et_3.
\]
Together with the estimate \eqref{eq:J11} on $J_{11}$, we deduce the bound
\begin{equation}\label{eq:E21:crit}
 \frac12\ddt \|\pa_x u\|_{L^2}^2 + \nu_h \|\lh^{\frac{3}{2}} u\|_{L^2}^2 \lesssim \big( \|\omega_0\|_{L^\infty} +  E_2^{\frac12} \big) \Et_3.
\end{equation}
For the second component $\|\lh^{\frac12}\omega\|_{L^2}^2$, we have
\begin{align*}
 & \frac12\ddt \|\lh^{\frac12}\omega\|_{L^2}^2 + \nu_h \|\lh\omega\|_{L^2}^2 = 
 -\int_\Omega \lh\omega\cdot(u\pa_x\omega+w\pa_z\omega)\dxdz\\
 & = \int_\Omega \Big(-\lh\omega\cdot u\cdot\pa_x\omega + \omega\cdot(\lh\omega\cdot\pa_zw + \lh\pa_z\omega\cdot w)\Big)\dxdz\\
 & \leq \|u\|_{L^\infty}\|\lh\omega\|_{L^2}^2+\|\omega\|_{L^\infty}\|\lh\omega\|_{L^2}\|\pa_xu\|_{L^2}+\int_\Omega \lh^{\frac12}\pa_z\omega\cdot\lh^{\frac12}(\omega w)\dxdz\\
 & \lesssim \|\omega_0\|_{L^\infty}(\Et_3+\Et_3^{\frac12}\Et_2^{\frac12})+\Et_3^{\frac12}\,\|\lh^{\frac12}(\omega w)\|_{L^2}
 \lesssim \|\omega_0\|_{L^\infty}\Et_3+\Et_3^{\frac12}\,\|\lh^{\frac12}(\omega w)\|_{L^2},
\end{align*}
where we can further apply Lemma \ref{lem:improvedLeib} and estimate
\[
  \|\lh^{\frac12}(\omega w)\|_{L^2_x}\lesssim \|\omega\|_{L^\infty_x}\|\lh^{\frac12}w\|_{L^2_x}+\|\lh\omega\|_{L^2_x}\|w\|_{L^2_x}
  \lesssim \|\omega_0\|_{L^\infty_x}\|\lh^{\frac32}u\|_{L^2_x} + \|\lh\omega\|_{L^2_x}\|\pa_xu\|_{L^2_x}.
\]
It yields the bound
\begin{equation}\label{eq:E22:crit}
 \frac12\ddt \|\lh^{\frac12}\omega\|_{L^2}^2 + \nu_h \|\lh\omega\|_{L^2}^2 \lesssim \big( \|\omega_0\|_{L^\infty} +  E_2^{\frac12} \big) \Et_3.
\end{equation}
Lastly, for the third component $\|\pa_z\omega\|_{L^2}^2$, we follow the same approach as for the subcritical case, and obtain an estimate analogous to \eqref{eq:E2omegaz}. In particular, we control the term $\II_{31}$ as in \eqref{eq:II31} by
\[
 \II_{31} = -\frac12 \int_{\Omega} \pa_z\omega \,\pa_x (\omega^2) \dxdz
 \lesssim \|\omega\|_{L^\infty}\|\lh^{\frac12}\omega\|_{L^2}\|\lh^{\frac12}\pa_z\omega\|_{L^2}\lesssim \|\omega\|_{L^\infty} \Et_3,
\]
and the term $\II_{32}$ by
\begin{align*}
	\II_{32} = &~ \int_\Omega \pa_xu(\pa_z\omega)^2\dxdz
	\leq \int_\T\|(\pa_z\omega)^2\|_{L^1_z}\|\pa_xu\|_{L^\infty_z}\,\dd x
	\lesssim \int_\T\|(\pa_z\omega)^2\|_{L^1_z}\|\pa_x\omega\|_{L^2_z}\,\dd x\\
	\lesssim &~ \Big\|\|(\pa_z\omega)^2\|_{L^1_z}\Big\|_{L^2_x}\|\lh\omega\|_{L^2}
	\lesssim \Big\|\|\pa_z\omega\|_{L^4_x}^2\Big\|_{L^1_z}\|\lh\omega\|_{L^2}
	\lesssim \Big\|\|\pa_z\omega\|_{L^2_x}\|\lh^{\frac12}\pa_z\omega\|_{L^2_x}\Big\|_{L^1_z}\|\lh\omega\|_{L^2}\\
	\lesssim &~ \|\pa_z\omega\|_{L^2}\|\lh^{\frac12}\pa_z\omega\|_{L^2}\|\lh\omega\|_{L^2}\leq E_2^{\frac12}\Et_3.
\end{align*}
Therefore, we reach the bound
\begin{equation}\label{eq:E23:crit}
 \frac12\ddt \|\pa_z\omega\|_{L^2}^2 + \nu_h \|\lh^{\frac12}\pa_z\omega\|_{L^2}^2 \lesssim \big( \|\omega_0\|_{L^\infty} +  E_2^{\frac12} \big) \Et_3.
\end{equation}
Collecting the estimates \eqref{eq:E21:crit}, \eqref{eq:E22:crit} and \eqref{eq:E23:crit}, we conclude with the a priori bound
\begin{equation}\label{eq:Energy2:crit}
 \frac12\ddt E_2+\nu_h\Et_3\lesssim	\big( \|\omega_0\|_{L^\infty} +  E_2^{\frac12} \big) \Et_3.
\end{equation}
This allows us to obtain a global well-posedness result, analogous to Theorem \ref{thm:GWPsmall}.

\begin{theorem}\label{thm:existence:crit}
	Let $\alpha=1$. Suppose $u_0$ satisfies \eqref{eq:smoothinit} and
	\begin{equation}\label{eq:smallness:crit}
	 E_2(0)^{\frac12}+\|\omega_0\|_{L^\infty}\lesssim\nu_h.		
	\end{equation}
	Then the classical solution to \eqref{FPE} with $u(0)=u_0$ exists globally in time, with
	\[E_2(t)\leq E_2(0),\quad\forall~t\geq0.\]
\end{theorem}
\begin{proof}
	From \eqref{eq:Energy2:crit}, we have
    \[
    	\frac12\ddt E_2 \leq  \Big(-\nu_h+C	\big( \|\omega_0\|_{L^\infty} +  E_2^{\frac12} \big)\Big)\Et_3.
    \]
	Pick a small $\delta<\frac{1}{C}$. Then, if $E_2(0)^{\frac12}+\|\omega_0\|_{L^\infty}\leq\delta\nu_h$, whenever $E_2(t)=E_2(0)$, we have
	\[
	 \ddt E_2\leq  \Big(-\nu_h+C	\big( \|\omega_0\|_{L^\infty} +  E_2(0)^{\frac12}\big)\Big)\Et_3\leq -(1-C\delta)\nu_h\Et_3\leq0.
	\]
	Therefore, $E_2$ cannot go beyond $E_2(0)$.
\end{proof}

Next, we obtain a uniqueness result for classical solutions, analogous to Theorems \ref{thm:uniqueness} and \ref{thm:uniqueness2}.
\begin{theorem}\label{thm:uniqueness:crit}
	The classical solution in Theorem \ref{thm:existence:crit} is unique.
\end{theorem}
\begin{proof}
The proof follows the outline of Theorem \ref{thm:uniqueness}, with significant modifications in the estimate of the term $B_2$. We provide a sketch of the argument below.

Start from the expression \eqref{est:unique-1}:
\[
    \frac12\ddt \|u\|_{L^2}^2 + \nu_h \|\lh^{\frac12} u\|_{L^2}^2 =
    - \int_\Omega \Big( u (u_1\pa_x+w_1\pa_z)u  + u^2\pa_x u_2 + w\omega_2 u \Big)\dxdz := B_0 + B_1 + B_2,
\]
where $B_0=0$ by incompressibility. The control of $B_1$ in \eqref{eq:B1bound} implies
\[
B_1\lesssim \|u\|_{L^2}\|\lh^{\frac12}u\|_{L^2}\|\lh\omega_2\|_{L^2}
\leq \frac14\nu_h\|\lh^{\frac12}u\|_{L^2}^2 + C\|\lh\omega_2\|_{L^2}^2\|u\|_{L^2}^2.
\]
Now, for the term $B_2$, we use \eqref{eq:improvedLeib} and get
\begin{align*}
B_2 & =-\int_\Omega \lh^{\frac12}(\omega_2u)\cdot\lh^{-\frac12} w \dxdz\lesssim \Big\| \|\lh^{\frac12}(\omega_2u)\|_{L^2_x}\Big\|_{L^1_z} \|\lh^{\frac12}u\|_{L^2}\\
& \lesssim \Big\|\Big(\|\omega_2\|_{L^\infty_x}\|\lh^{\frac12} u\|_{L^2_x}+\|\lh\omega_2\|_{L^2_x}\|u\|_{L^2_x}\Big)\Big\|_{L^1_z} \|\lh^{\frac12}u\|_{L^2}\\
& \lesssim \Big(\|\omega_2\|_{L^\infty} \|\lh^{\frac12} u\|_{L^2} + \|\lh\omega_2\|_{L^2}\|u\|_{L^2}\Big)\, \|\lh^{\frac12}u\|_{L^2}\\
& \leq C\|\omega_{20}\|_{L^\infty} \|\lh^{\frac12} u\|_{L^2}^2 +  \frac14\nu_h\|\lh^{\frac12}u\|_{L^2}^2 + C\|\lh\omega_2\|_{L^2}^2\|u\|_{L^2}^2.
\end{align*}
Combining the estimates, we obtain
\[
 \ddt \|u\|_{L^2}^2 + (\nu_h-2C\|\omega_{20}\|_{L^\infty}) \|\lh^{\frac12} u\|_{L^2}^2\leq C\|\lh\omega_2\|_{L^2}^2\|u\|_{L^2}^2.
\]
The smallness assumption \eqref{eq:smallness:crit} ensures $\nu_h-2C\|\omega_{20}\|_{L^\infty}\geq0$. We apply Gr\"onwall's inequality and get
\[
 \|u(T)\|_{L^2}^2\leq \|u_0\|_{L^2}^2\exp\int_0^T C\|\lh\omega_2(t)\|_{L^2}^2\,\dd t.
\]
Since $u_2$ is a classical solution, the time integral is finite. Hence, $u_0=0$ implies $u(T)=0$. Therefore, we obtain the desired uniqueness property.
\end{proof}

Combining the existence and uniqueness results in Theorems \ref{thm:existence:crit} and \ref{thm:uniqueness:crit}, we conclude with the global well-posedness of the critical FPE system with small initial data.
\begin{theorem}\label{thm:GWP:crit}
	Let $\alpha=1$. Suppose $u_0$ satisfies \eqref{eq:smoothinit} and the smallness condition
	\[
	\|u_0\|\ll\nu_h,
	\]
	or more precisely \eqref{eq:smallness:crit}. Then there exists a global-in-time unique classical solution to \eqref{FPE} with $u(0)=u_0$.
\end{theorem}

Note that the well-posedness result in Theorem \ref{thm:existence:crit} stands in sharp contrast to the ill-posedness result in Theorem \ref{thm:critical-nonlinear-ill}. As the size of the initial data $u_0$ grows relative to the viscosity coefficient $\nu_h$, the system undergoes a transition from well-posedness to ill-posedness. This phenomenon is new for the critical FPE system and is distinct from behavior observed in other critical fluid dynamic systems.

Theorem \ref{thm:existence:crit} can be further extended to yield global well-posedness for smooth solutions, by employing a higher-order energy estimate in the spirit of Proposition~\ref{prop:Ek}. The detailed argument is omitted here and left to the interested reader.

\appendix

\section{Borderline fractional Leibniz rules on torus}\label{sec:improvedLeib}
In this section, we provide a complete proof of Lemma \ref{lem:improvedLeib}.
In the whole space case $x\in\R$, the inequality \eqref{eq:improvedLeib} follows directly from the estimate \eqref{eq:BMO}. Our goal here is to verify that the same estimate holds in the periodic setting.

For convenience in representing integrals, we take $\T = [-\pi, \pi]$ and assume all functions are $2\pi$-periodic.
We also let $\Lambda_{\T}^s$ and $\Lambda_{\R}^s$ denote the fractional Laplacian of order $s$ on the periodic domain and the whole space, respectively.

We now state the following proposition, which is a generalized version of Lemma \ref{lem:improvedLeib}.
Note that inequality \eqref{eq:improvedLeib} corresponds to the special case $s = \frac12$ of \eqref{eq:border}.

\begin{proposition} \label{prop:nonlinear-est} Let $s \in (0,1)$. Let $g\in L^{\infty}(\T) \cap H^{s+\frac12}(\T)$ and $h \in H^s(\T)$. Then it holds that 
\begin{equation}\label{eq:border}
\|\Lambda_{\T}^s (gh)\|_{L^2(\T)} 
\lesssim \|g\|_{L^{\infty}(\T)} \|\Lambda_{\T}^s h\|_{L^2(\T)} + \|\Lambda_{\T}^{s+\frac12} g\|_{L^2(\T)} \|h\|_{L^2(\T)}.
\end{equation}
\end{proposition}

The proof of Proposition \ref{prop:nonlinear-est} is based on direct application of the inequality in $\R$ for functions $\chi\tilde{g}$ and $\chi\tilde{h}$:
\begin{equation}\label{eq:borderR}
\|\Lambda_{\R}^s (\chi \tilde{g}\chi\tilde{h})\|_{L^2(\R)} 
\lesssim \|\chi\tilde{g}\|_{L^{\infty}(\R)} \|\Lambda_{\R}^s (\chi\tilde{h})\|_{L^2(\R)} + \|\Lambda_{\R}^{s+\frac12} (\chi\tilde{g})\|_{L^2(\R)} \|\chi\tilde{h}\|_{L^2(\R)},
\end{equation}
where $\tilde{g}$ and $\tilde{h}$ denote the periodic extensions of $g$ and $h$ from $\T$ to $\R$, and $\chi$ is a smooth cutoff function ranged in $[0,1]$ such that $\chi = 1$ on $2\T$ and $\chi = 0$ on $(3\T)^c$. 

To derive \eqref{eq:border} from \eqref{eq:borderR}, we make use of the following lemmas.

\begin{lemma}\label{lem:cutoff1}
 Let $s\geq0$, and $f\in H^s(\T)$. Then we have
 \begin{equation}\label{eq:cutoff1}
  \|\Lambda_{\R}^s (\chi\tilde{f})\|_{L^2(\R)}\lesssim \|\Lambda_{\T}^s f\|_{L^2(\T)} + \|f\|_{L^2(\T)}.
 \end{equation}
\end{lemma}
\begin{proof}
For $s=0$, we have
\[\|\chi\tilde{f}\|_{L^2(\R)}\leq \|\tilde{f}\|_{L^2(3\T)}\leq 3\|f\|_{L^2(\T)}.\]

Next, we consider $s\in(0,1)$. In view of the pointwise identity \cite{cordoba2004maximum, constantin2015long}
\[ 
2f(x) \Lambda_{\T}^{2s} f(x) = \Lambda_{\T}^{2s} (f^2)(x) + c_s\text{ p.v.} \int_{\R} \frac{(\tilde f(x) - \tilde f(x+y))^2}{|y|^{1+2s}} \dd y, 
\] and dropping the \text{p.v.} notation for simplicity, we have 
\begin{align}
\|\Lambda_{\T}^s f\|_{L^2(\T)}^2 & = \frac13\int_{3\T} \tilde f(x) \Lambda_{\T}^{2s} \tilde f(x) \dd x 
= \frac{c_s}{6}\int_{3\T} \int_{\R} \frac{(\tilde{f}(x) - \tilde{f}(x+y))^2}{|y|^{1+2s}} \dd y \dd x\label{eq:TtoR}
\\&\ge \frac{c_s}{6}\int_{3\T}\int_{\R} \chi(x)^2 \frac{(\tilde f(x) - \tilde f(x+y))^2}{|y|^{1+2s}} \dd y \dd x
= \frac{c_s}{6}\int_{\R} \int_{\R} \chi(x)^2 \frac{(\tilde f(x) - \tilde f(x+y))^2}{|y|^{1+2s}} \dd y \dd x.\nonumber
\end{align}
Here, the second equality holds because the spatial integral of $\Lambda_{\T}^{2s}(f^2)$ vanishes; the inequality follows from the fact that $0 \le \chi \le 1$ and the non-negativity of the integrand; and the final equality holds since the cutoff function $\chi$ vanishes outside $3\T$.
Using the reverse triangle inequality, we further obtain 
\begin{align*}
\|\Lambda_{\T}^s f\|_{L^2(\T)}^2
&\ge \frac{c_s}{12}\int_{\R} \int_{\R} \frac{(\chi(x) \tilde{f}(x) - \chi(x+y) \tilde{f}(x+y))^2}{|y|^{1+2s}} \dd y \dd x 
\\&\quad- \frac{c_s}{6}\int_{\R} \int_{\R} \frac{(\chi(x+y) - \chi(x))^2\tilde{f}(x+y)^2}{|y|^{1+2s}} \dd y \dd x
\\&= \frac16\|\Lambda_{\R}^s (\chi \tilde{f})\|_{L^2(\R)}^2 - \frac{c_s}{6}\int_{\R} \int_{\T} \frac{(\chi(x+y) - \chi(x))^2\tilde{f}(x+y)^2}{|y|^{1+2s}} \dd y \dd x
\\&\quad-\frac{c_s}{6} \int_{\R} \int_{\R \setminus \T} \frac{(\chi(x+y) - \chi(x))^2\tilde{f}(x+y)^2}{|y|^{1+2s}} \dd y \dd x.
\end{align*}
For the second term, we use the fact that $\chi(x+y) = \chi(x) =0$ when $y \in \T$ and $x \in (4\T)^c$ and estimate 
\[
\int_{\R} \int_{\T} \frac{(\chi(x+y) - \chi(x))^2\tilde{f}(x+y)^2}{|y|^{1+2s}} \dd y \dd x
\le \int_{4\T} \int_{\T} \frac{(\|\chi'\|_{L^\infty(\R)}|y|)^2\tilde{f}(x+y)^2}{|y|^{1+2s}} \dd y \dd x \lesssim \|f\|_{L^2(\T)}^2.
\]
For the third term, we have the bound 
\begin{align*}
&\int_{\R} \int_{\R \setminus \T} \frac{(\chi(x+y) - \chi(x))^2\tilde{f}(x+y)^2}{|y|^{1+2s}} \dd y \dd x \\
&\le 2\int_{\R \setminus \T} \frac{1}{|y|^{1+2s}} \left(\int_{\R} \big(\chi(x+y)^2 +\chi(x)^2\big) \tilde{f}(x+y)^2 \dd x\right) \dd y\\
&\le 2\int_{\R \setminus \T} \frac{1}{|y|^{1+2s}}\cdot 6\pi\|f\|_{L^2(\T)}^2\,\dd y
\lesssim\|f\|_{L^2(\T)}^2.
\end{align*}
Putting all these estimates together, we obtain
\[ 
\|\Lambda_{\R}^s (\chi \tilde f)\|_{L^2(\R)}^2 \lesssim \|\Lambda_{\T}^s f\|_{L^2(\T)}^2 + \|f\|_{L^2(\T)}^2,
\] yielding the desired inequality \eqref{eq:cutoff1}.

Finally, for $s\ge1$, we split $s=[s]+\{s\}$ where $[s]$ is the integer part, and $\{s\}\in[0,1)$. Apply Leibniz rule and decompose
\[
\pa_x^{[s]}(\chi\tilde{f})=\sum_{j=0}^{[s]}\binom{[s]}{j}\pa_x^{[s]-j}\chi\cdot \pa_x^j\tilde{f}.
\]
Then we apply \eqref{eq:cutoff1} with $\{s\}\in[0,1)$ and deduce
\begin{align*}
 \|\Lambda_{\R}^s (\chi\tilde{f})\|_{L^2(\R)} & = \|\Lambda_{\R}^{\{s\}} \pa_x^{[s]}(\chi\tilde{f})\|_{L^2(\R)}\le \sum_{j=0}^{[s]}\binom{[s]}{j}\|\Lambda_{\R}^{\{s\}}( \pa_x^{[s]-j}\chi\cdot \pa_x^j\tilde{f})\|_{L^2(\R)}\\
 &\lesssim \sum_{j=0}^{[s]}\big(\|\Lambda_\T^{\{s\}}\pa_x^j\tilde{f}\|_{L^2(\T)}^2+\|\pa_x^j\tilde{f}\|_{L^2(\T)}^2\big)\lesssim \|\Lambda_{\T}^s f\|_{L^2(\T)}^2 + \|f\|_{L^2(\T)}^2,
\end{align*}
where the last inequality is due to interpolation. This completes the proof.
\end{proof}

\begin{lemma}\label{lem:cutoff2}
 Let $s\in[0,1)$, and $f\in H^s(\T)$. Then we have
 \begin{equation}\label{eq:cutoff2}
  \|\Lambda_{\T}^s f\|_{L^2(\T)} \lesssim \|\Lambda_{\R}^s (\chi\tilde{f})\|_{L^2(\R)} + \|f\|_{L^2(\T)}.
 \end{equation}
\end{lemma}

\begin{proof}
For $s=0$, the inequality \eqref{eq:cutoff2} holds trivially. 

For $s\in(0,1)$, we start with an argument similar to \eqref{eq:TtoR} and obtain: 
\begin{align*}
\|\Lambda_{\T}^s f\|_{L^2(\T)}^2 & = \int_{\T} f(x) \Lambda_{\T}^{2s} f(x) \dd x 
= \frac{c_s}{2}\int_{\T} \int_{\R} \frac{(f(x) - \tilde{f}(x+y))^2}{|y|^{1+2s}} \dd y \dd x
\\& = \frac{c_s}{2}\int_{\T} \int_{\R} \frac{(\chi(x)f(x) - \chi(x)\tilde{f}(x+y))^2}{|y|^{1+2s}} \dd y \dd x
\\& = \frac{c_s}{2}\int_{\T} \int_{\R} \frac{(\chi(x)f(x) - \chi(x+y)\tilde{f}(x+y))^2}{|y|^{1+2s}} \dd y \dd x + \mathcal{R}=:\mathcal{S}+\mathcal{R},
\end{align*}
where the remainder term $\mathcal{R}$ has the form: 
\begin{align*}
\mathcal{R} & =\frac{c_s}{2}\int_{\T} \int_{\R} \frac{(\chi(x)f(x) - \chi(x)\tilde{f}(x+y))^2-(\chi(x)f(x) - \chi(x+y)\tilde{f}(x+y))^2}{|y|^{1+2s}} \dd y \dd x
\\ & =\frac{c_s}{2}\int_{\T} \int_{\R} \frac{\big(\chi(x+y)-\chi(x)\big)\big(-(\chi(x+y)+\chi(x))\tilde f(x+y) + 2\chi(x)f(x)\big)\tilde{f}(x+y)}{|y|^{1+2s}} \dd y \dd x
\\ & = \int_{\T} \int_{\T} \star\,\dd y\dd x + \int_{\T} \int_{\R \setminus \T} \star\,\dd y\dd x=:\mathcal{R}_1 + \mathcal{R}_2.
\end{align*}
The term $\mathcal{R}_1$ vanishes because $\chi(x+y) = \chi(x) = 1$ when $x, y \in \T$. For $\mathcal{R}_2$, we have the following bound:
\begin{align*}
|\mathcal{R}_2| & \leq \frac{c_s}{2}\int_\T\int_{\R \setminus \T}\frac{2\big(2|\tilde f(x+y)| + 2|f(x)|\big)|\tilde{f}(x+y)|}{|y|^{1+2s}} \dd y \dd x\\
& = 2c_s \int_\T\int_\T \big(|\tilde f(x+y)| + |f(x)|\big)|\tilde{f}(x+y)|\sum_{k\neq0} \frac{1}{|y - 2\pi k|^{1+2s}} \dd y \dd x\\
&\leq 2c_s\sum_{k\neq0}\frac{1}{((2|k|-1)\pi)^{1+2s}} \int_\T\int_\T \big(|\tilde f(x+y)| + |f(x)|\big)|\tilde{f}(x+y)| \dd y \dd x\lesssim\|f\|_{L^2(\T)}^2.
\end{align*}
Here, we have used the fact that $\tilde f$ is $2\pi$-periodic. The penultimate inequality follows from the bound $|y - 2\pi k|^{1+2s} \ge ((2|k|-1)\pi)^{1+2s}$ for $y\in\T$. The infinite sum converges since $s>0$.

Finally, for the term $\mathcal{S}$, we have the direct bound:
\[
\mathcal{S}\leq \frac{c_s}{2}\int_{\R} \int_{\R} \frac{(\chi(x)\tilde{f}(x) - \chi(x+y)\tilde{f}(x+y))^2}{|y|^{1+2s}} \dd y \dd x = \|\Lambda_{\R}^s (\chi\tilde{f})\|_{L^2(\R)}^2.
\]

Combining the estimates on $\mathcal{S}$ and $\mathcal{R}$, we conclude with the desired inequality \eqref{eq:cutoff2}.
\end{proof}

Now we are ready to prove Proposition \ref{prop:nonlinear-est}, using estimate \eqref{eq:borderR}, Lemma \ref{lem:cutoff1}, and Lemma \ref{lem:cutoff2}.
\begin{proof}[Proof of Proposition \ref{prop:nonlinear-est}]
 We first assume $h$ is a mean-free function on $\T$. Apply Lemma \ref{lem:cutoff2} with $f=gh$ and cutoff function $\chi^2$. It yields
 \[
 \|\Lambda_\T^s(gh)\|_{L^2(\T)}\lesssim \|\Lambda_{\R}^s (\chi \tilde{g}\chi\tilde{h})\|_{L^2(\R)} + \|gh\|_{L^2(\T)}.
 \]
 For the first term, we apply \eqref{eq:borderR} and Lemma \ref{lem:cutoff1} to obtain
 \begin{align*}
  \|\Lambda_{\R}^s (\chi \tilde{g}\chi\tilde{h})\|_{L^2(\R)} 
& \lesssim \|\chi\tilde{g}\|_{L^{\infty}(\R)} \|\Lambda_{\R}^s (\chi\tilde{h})\|_{L^2(\R)} + \|\Lambda_{\R}^{s+\frac12} (\chi\tilde{g})\|_{L^2(\R)} \|\chi\tilde{h}\|_{L^2(\R)}\\
& \lesssim \|g\|_{L^\infty(\T)}(\|\Lambda_\T^sh\|_{L^2(\T)}+\|h\|_{L^2(\T)}) + (\|\Lambda_\T^{s+\frac12}g\|_{L^2(\T)}+\|g\|_{L^2(\T)})\|h\|_{L^2(\T)}\\
& \leq \big(\|g\|_{L^{\infty}(\T)} \|\Lambda_{\T}^s h\|_{L^2(\T)} + \|\Lambda_{\T}^{s+\frac12} g\|_{L^2(\T)} \|h\|_{L^2(\T)}\big) + (1+\sqrt{2\pi})\|g\|_{L^{\infty}(\T)}\|h\|_{L^2(\T)}.
 \end{align*}
Since $h$ is mean free, we apply the Poincar\'e inequality and get
\[
 \|gh\|_{L^2(\T)}\leq\|g\|_{L^\infty(\T)}\|h\|_{L^2(\T)}\lesssim\|g\|_{L^\infty(\T)}\|\Lambda_{\T}^s h\|_{L^2(\T)}.
\]
Therefore, we conclude with the bound \eqref{eq:border}.
 
For general function $h$, we decompose
\[h(x) = \langle h\rangle +h_{\neq}(x),\quad \langle h\rangle:=\frac{1}{2\pi}\int_\T h(x)\,\dd x.\]
Then, we have
\begin{align*}
 \|\Lambda_\T^s(gh)\|_{L^2(\T)} & = \|\Lambda_\T^s(gh_{\neq})\|_{L^2(\T)}+\langle h\rangle \|\Lambda_\T^sg\|_{L^2(\T)}\\
& \lesssim \|g\|_{L^{\infty}(\T)} \|\Lambda_{\T}^s h_{\neq}\|_{L^2(\T)} + \|\Lambda_{\T}^{s+\frac12} g\|_{L^2(\T)} \|h_{\neq}\|_{L^2(\T)}+\langle h\rangle \|\Lambda_\T^sg\|_{L^2(\T)}\\
& \lesssim \|g\|_{L^{\infty}(\T)} \|\Lambda_{\T}^s h\|_{L^2(\T)} + \|\Lambda_{\T}^{s+\frac12} g\|_{L^2(\T)} \|h\|_{L^2(\T)},
\end{align*}
where we apply the estimate \eqref{eq:border} for mean-free function $h_{\neq}$ in the first inequality. For the second inequality, we have used $\langle h\rangle\leq \frac{1}{2\pi}\|h\|_{L^1(\T)}\leq \frac{1}{\sqrt{2\pi}}\|h\|_{L^2(\T)}$ and the Poincar\'e inequality $\|\Lambda_{\T}^s g\|_{L^2(\T)}\lesssim \|\Lambda_{\T}^{s+\frac12} g\|_{L^2(\T)}$.
\end{proof}

Using the Leibniz rule and interpolation inequalities, Proposition \ref{prop:nonlinear-est} can be extended to any $s \geq 0$. We omit the details of the proof.

\section*{Acknowledgment}
Q.L. was partially supported by an AMS-Simons Travel Grant.
C.T. was partially supported by NSF grants DMS-2108264 and DMS-2238219.

\bibliographystyle{plain}
\bibliography{reference}

\begin{thebibliography}{10}

\bibitem{azerad2001mathematical}
Pascal Az{\'e}rad and Francisco Guill{\'e}n.
\newblock Mathematical justification of the hydrostatic approximation in the
  primitive equations of geophysical fluid dynamics.
\newblock {\em SIAM Journal on Mathematical Analysis}, 33(4):847--859, 2001.

\bibitem{benyi2025fractional}
{\'A}rp{\'a}d B{\'e}nyi, Tadahiro Oh, and Tengfei Zhao.
\newblock Fractional {L}eibniz rule on the torus.
\newblock {\em Proceedings of the American Mathematical Society},
  153(01):207--221, 2025.

\bibitem{blumen1972geostrophic}
William Blumen.
\newblock Geostrophic adjustment.
\newblock {\em Reviews of Geophysics}, 10(2):485--528, 1972.

\bibitem{brenier1999homogeneous}
Yann Brenier.
\newblock Homogeneous hydrostatic flows with convex velocity profiles.
\newblock {\em Nonlinearity}, 12(3):495, 1999.

\bibitem{brenier2003remarks}
Yann Brenier.
\newblock Remarks on the derivation of the hydrostatic {E}uler equations.
\newblock {\em Bulletin des Sciences Mathematiques}, 127(7):585--595, 2003.

\bibitem{caffarelli2010drift}
Luis~A Caffarelli and Alexis Vasseur.
\newblock Drift diffusion equations with fractional diffusion and the
  quasi-geostrophic equation.
\newblock {\em Annals of Mathematics}, pages 1903--1930, 2010.

\bibitem{cao2015finite}
Chongsheng Cao, Slim Ibrahim, Kenji Nakanishi, and Edriss~S Titi.
\newblock Finite-time blowup for the inviscid primitive equations of oceanic
  and atmospheric dynamics.
\newblock {\em Communications in Mathematical Physics}, 337(2):473--482, 2015.

\bibitem{cao2016global}
Chongsheng Cao, Jinkai Li, and Edriss~S Titi.
\newblock Global well-posedness of the three-dimensional primitive equations
  with only horizontal viscosity and diffusion.
\newblock {\em Communications on Pure and Applied Mathematics},
  69(8):1492--1531, 2016.

\bibitem{cao2017strong}
Chongsheng Cao, Jinkai Li, and Edriss~S Titi.
\newblock Strong solutions to the {3D} primitive equations with only horizontal
  dissipation: Near {$H^1$} initial data.
\newblock {\em Journal of Functional Analysis}, 272(11):4606--4641, 2017.

\bibitem{cao2020global}
Chongsheng Cao, Jinkai Li, and Edriss~S Titi.
\newblock Global well-posedness of the {3D} primitive equations with horizontal
  viscosity and vertical diffusivity.
\newblock {\em Physica D: Nonlinear Phenomena}, 412:132606, 2020.

\bibitem{cao2020well}
Chongsheng Cao, Quyuan Lin, and Edriss~S Titi.
\newblock On the well-posedness of reduced {3D} primitive geostrophic
  adjustment model with weak dissipation.
\newblock {\em Journal of Mathematical Fluid Mechanics}, 22:1--34, 2020.

\bibitem{cao2007global}
Chongsheng Cao and Edriss~S Titi.
\newblock Global well-posedness of the three-dimensional viscous primitive
  equations of large scale ocean and atmosphere dynamics.
\newblock {\em Annals of Mathematics}, 166(1):245--267, 2007.

\bibitem{chen1991sufficient}
Xiao~Ling Chen and Phillip~J Morrison.
\newblock A sufficient condition for the ideal instability of shear flow with
  parallel magnetic field.
\newblock {\em Physics of Fluids B: Plasma Physics}, 3(4):863--865, 1991.

\bibitem{collot2023stable}
Charles Collot, Slim Ibrahim, and Quyuan Lin.
\newblock Stable singularity formation for the inviscid primitive equations.
\newblock {\em Annales de l'Institut Henri Poincar{\'e} C}, 41(2):317--356,
  2023.

\bibitem{constantin2015long}
Peter Constantin, Andrei Tarfulea, and Vlad Vicol.
\newblock Long time dynamics of forced critical {SQG}.
\newblock {\em Communications in Mathematical Physics}, 335:93--141, 2015.

\bibitem{constantin2012nonlinear}
Peter Constantin and Vlad Vicol.
\newblock Nonlinear maximum principles for dissipative linear nonlocal
  operators and applications.
\newblock {\em Geometric And Functional Analysis}, 22(5):1289--1321, 2012.

\bibitem{constantin1999behavior}
Peter Constantin and Jiahong Wu.
\newblock Behavior of solutions of {2D} quasi-geostrophic equations.
\newblock {\em SIAM Journal on Mathematical Analysis}, 30(5):937--948, 1999.

\bibitem{cordoba2004maximum}
Antonio C{\'o}rdoba and Diego C{\'o}rdoba.
\newblock A maximum principle applied to quasi-geostrophic equations.
\newblock {\em Communications in mathematical physics}, 249(3):511--528, 2004.

\bibitem{do2018global}
Tam Do, Alexander Kiselev, Lenya Ryzhik, and Changhui Tan.
\newblock Global regularity for the fractional {E}uler alignment system.
\newblock {\em Archive for Rational Mechanics and Analysis}, 228(1):1--37,
  2018.

\bibitem{furukawa2020rigorous}
Ken Furukawa, Yoshikazu Giga, Matthias Hieber, Amru Hussein, Takahito
  Kashiwabara, and Marc Wrona.
\newblock Rigorous justification of the hydrostatic approximation for the
  primitive equations by scaled {N}avier--{S}tokes equations.
\newblock {\em Nonlinearity}, 33(12):6502, 2020.

\bibitem{gerard2020well}
David Gerard-Varet, Nader Masmoudi, and Vlad Vicol.
\newblock Well-posedness of the hydrostatic {N}avier--{S}tokes equations.
\newblock {\em Analysis \& PDE}, 13(5):1417--1455, 2020.

\bibitem{ghoul2022effect}
Tej~Eddine Ghoul, Slim Ibrahim, Quyuan Lin, and Edriss~S Titi.
\newblock On the effect of rotation on the life-span of analytic solutions to
  the {3D} inviscid primitive equations.
\newblock {\em Archive for Rational Mechanics and Analysis}, 243:747--806,
  2022.

\bibitem{gill1976adjustment}
Adrian~E Gill.
\newblock Adjustment under gravity in a rotating channel.
\newblock {\em Journal of Fluid Mechanics}, 77(3):603--621, 1976.

\bibitem{gill1982atmosphere}
Adrian~E Gill.
\newblock {\em Atmosphere-ocean dynamics}.
\newblock Elsevier, 2016.

\bibitem{grenier1999derivation}
Emmanuel Grenier.
\newblock On the derivation of homogeneous hydrostatic equations.
\newblock {\em ESAIM: Mathematical Modelling and Numerical Analysis},
  33(5):965--970, 1999.

\bibitem{han2016ill}
Daniel Han-Kwan and Toan~T Nguyen.
\newblock Ill-posedness of the hydrostatic {E}uler and singular {V}lasov
  equations.
\newblock {\em Archive for Rational Mechanics and Analysis}, 221(3):1317--1344,
  2016.

\bibitem{hermann1993energetics}
AJ~Hermann and WB~Owens.
\newblock Energetics of gravitational adjustment for mesoscale chimneys.
\newblock {\em Journal of Physical Oceanography}, 23(2):346--371, 1993.

\bibitem{hieber2016global}
Matthias Hieber and Takahito Kashiwabara.
\newblock Global strong well-posedness of the three dimensional primitive
  equations in {$L^p$}-spaces.
\newblock {\em Archive for Rational Mechanics and Analysis}, 221(3):1077--1115,
  2016.

\bibitem{holton1973introduction}
James~R Holton.
\newblock An introduction to dynamic meteorology.
\newblock {\em American Journal of Physics}, 41(5):752--754, 1973.

\bibitem{ibrahim2021finite}
Slim Ibrahim, Quyuan Lin, and Edriss~S Titi.
\newblock Finite-time blowup and ill-posedness in {S}obolev spaces of the
  inviscid primitive equations with rotation.
\newblock {\em Journal of Differential Equations}, 286:557--577, 2021.

\bibitem{kato1988commutator}
Tosio Kato and Gustavo Ponce.
\newblock Commutator estimates and the {E}uler and {N}avier-{S}tokes equations.
\newblock {\em Communications on Pure and Applied Mathematics}, 41(7):891--907,
  1988.

\bibitem{kiselev2008blow}
Alexander Kiselev, Fedor Nazarov, and Roman Shterenberg.
\newblock Blow up and regularity for fractal {B}urgers equation.
\newblock {\em Dynamics of PDE}, 5(3):211--240, 2008.

\bibitem{kiselev2007global}
Alexander Kiselev, Fedor Nazarov, and Alexander Volberg.
\newblock Global well-posedness for the critical {2D} dissipative
  quasi-geostrophic equation.
\newblock {\em Inventiones Mathematicae}, 167(3):445--453, 2007.

\bibitem{kobelkov2006existence}
Georgij~M Kobelkov.
\newblock Existence of a solution `in the large' for the {3D} large-scale ocean
  dynamics equations.
\newblock {\em Comptes Rendus Mathematique}, 343(4):283--286, 2006.

\bibitem{kukavica2014local}
Igor Kukavica, Nader Masmoudi, Vlad Vicol, and Tak~Kwong Wong.
\newblock On the local well-posedness of the {P}randtl and hydrostatic {E}uler
  equations with multiple monotonicity regions.
\newblock {\em SIAM Journal on Mathematical Analysis}, 46(6):3865--3890, 2014.

\bibitem{kukavica2011local}
Igor Kukavica, Roger Temam, Vlad~C Vicol, and Mohammed Ziane.
\newblock Local existence and uniqueness for the hydrostatic {E}uler equations
  on a bounded domain.
\newblock {\em Journal of Differential Equations}, 250(3):1719--1746, 2011.

\bibitem{kukavica2007regularity}
Igor Kukavica and Mohammed Ziane.
\newblock On the regularity of the primitive equations of the ocean.
\newblock {\em Nonlinearity}, 20(12):2739, 2007.

\bibitem{kuo1997time}
Allen~C Kuo and Lorenzo~M Polvani.
\newblock Time-dependent fully nonlinear geostrophic adjustment.
\newblock {\em Journal of Physical Oceanography}, 27(8):1614--1634, 1997.

\bibitem{li2019kato}
Dong Li.
\newblock On {K}ato--{P}once and fractional {L}eibniz.
\newblock {\em Revista Matem{\'a}tica Iberoamericana}, 35(1):23--100, 2019.

\bibitem{li2019primitive}
Jinkai Li and Edriss~S Titi.
\newblock The primitive equations as the small aspect ratio limit of the
  {N}avier--{S}tokes equations: Rigorous justification of the hydrostatic
  approximation.
\newblock {\em Journal de Math{\'e}matiques Pures et Appliqu{\'e}es},
  124:30--58, 2019.

\bibitem{li2022primitive}
Jinkai Li, Edriss~S Titi, and Guozhi Yuan.
\newblock The primitive equations approximation of the anisotropic horizontally
  viscous {3D} {N}avier--{S}tokes equations.
\newblock {\em Journal of Differential Equations}, 306:492--524, 2022.

\bibitem{lin2022effect}
Quyuan Lin, Xin Liu, and Edriss~S Titi.
\newblock On the effect of fast rotation and vertical viscosity on the lifespan
  of the {3D} primitive equations.
\newblock {\em Journal of Mathematical Fluid Mechanics}, 24:1--44, 2022.

\bibitem{masmoudi2012h}
Nader Masmoudi and Tak~Kwong Wong.
\newblock On the ${H}^s$ theory of hydrostatic {E}uler equations.
\newblock {\em Archive for Rational Mechanics and Analysis}, 204(1):231--271,
  2012.

\bibitem{paicu2020hydrostatic}
Marius Paicu, Ping Zhang, and Zhifei Zhang.
\newblock On the hydrostatic approximation of the {N}avier--{S}tokes equations
  in a thin strip.
\newblock {\em Advances in Mathematics}, 372:107293, 2020.

\bibitem{plougonven2005lagrangian}
R~Plougonven and V~Zeitlin.
\newblock Lagrangian approach to geostrophic adjustment of frontal anomalies in
  a stratified fluid.
\newblock {\em Geophysical \& Astrophysical Fluid Dynamics}, 99(2):101--135,
  2005.

\bibitem{renardy2009ill}
Michael Renardy.
\newblock Ill-posedness of the hydrostatic {E}uler and {N}avier--{S}tokes
  equations.
\newblock {\em Archive for Rational Mechanics and Analysis}, 194(3):877--886,
  2009.

\bibitem{rossby1938mutual}
Carl-Gustav Rossby.
\newblock On the mutual adjustment of pressure and velocity distributions in
  certain simple current systems, {II}.
\newblock {\em Journal of Marine Research}, 1(3):239--263, 1938.

\bibitem{shvydkoy2017eulerian}
Roman Shvydkoy and Eitan Tadmor.
\newblock Eulerian dynamics with a commutator forcing.
\newblock {\em Transactions of Mathematics and its Applications}, 1(1):tnx001,
  2017.

\bibitem{stefanov2025global}
Atanas Stefanov, Jiahong Wu, Xiaojing Xu, and Zhuan Ye.
\newblock Global regularity results of the {2D} fractional {B}oussinesq
  equations.
\newblock {\em Mathematische Annalen}, 391(4):5965--6012, 2025.

\bibitem{wong2015blowup}
Tak~Kwong Wong.
\newblock Blowup of solutions of the hydrostatic {E}uler equations.
\newblock {\em Proceedings of the American Mathematical Society},
  143(3):1119--1125, 2015.

\end{thebibliography}

\end{document}